\pdfoutput=1
\documentclass[11pt, a4paper, twoside,reqno]{amsart}
\usepackage[centering, totalwidth = 390pt, totalheight = 600pt]{geometry}
\usepackage{amssymb, amsmath, amsthm,scalerel}
\usepackage{microtype, stmaryrd, url, lmodern, eucal}
\usepackage[shortlabels]{enumitem}
\usepackage[latin1]{inputenc}
\usepackage{color}
\definecolor{darkgreen}{rgb}{0,0.45,0}
\usepackage[colorlinks,citecolor=darkgreen,linkcolor=darkgreen]{hyperref}
\usepackage[british]{babel}

\makeatletter
\def\@cite#1#2{[{#1\if@tempswa ,~#2\fi}]}
\makeatother

\DeclareMathAlphabet{\mathbf}{OT1}{cmr}{b}{n}

\usepackage[arrow, matrix, tips, curve, graph, rotate]{xy}
\SelectTips{cm}{10}

\makeatletter
\def\matrixobject@{%
  \edef \next@{={\DirectionfromtheDirection@ }}%
  \expandafter \toks@ \next@ \plainxy@
  \let\xy@@ix@=\xyq@@toksix@
  \xyFN@ \OBJECT@}
\let\xy@entry@@norm=\entry@@norm
\def\entry@@norm@patched{%
  \let\object@=\matrixobject@
  \xy@entry@@norm }
\AtBeginDocument{\let\entry@@norm\entry@@norm@patched}
\makeatother

\newcommand{\twocong}[2][0.5]{\ar@{}[#2] \save ?(#1)*{\cong}\restore}
\newcommand{\twoeq}[2][0.5]{\ar@{}[#2] \save ?(#1)*{=}\restore}
\newcommand{\rtwocell}[3][0.5]{\ar@{}[#2] \ar@{=>}?(#1)+/l 0.2cm/;?(#1)+/r 0.2cm/^{#3}}
\newcommand{\ltwocell}[3][0.5]{\ar@{}[#2] \ar@{=>}?(#1)+/r 0.2cm/;?(#1)+/l 0.2cm/^{#3}}
\newcommand{\ltwocello}[3][0.5]{\ar@{}[#2] \ar@{=>}?(#1)+/r 0.2cm/;?(#1)+/l 0.2cm/_{#3}}
\newcommand{\dtwocell}[3][0.5]{\ar@{}[#2] \ar@{=>}?(#1)+/u  0.2cm/;?(#1)+/d 0.2cm/^{#3}}
\newcommand{\dltwocell}[3][0.5]{\ar@{}[#2] \ar@{=>}?(#1)+/ur  0.2cm/;?(#1)+/dl 0.2cm/^{#3}}
\newcommand{\drtwocell}[3][0.5]{\ar@{}[#2] \ar@{=>}?(#1)+/ul  0.2cm/;?(#1)+/dr 0.2cm/^{#3}}
\newcommand{\dthreecell}[3][0.5]{\ar@{}[#2] \ar@3{->}?(#1)+/u  0.2cm/;?(#1)+/d 0.2cm/^{#3}}
\newcommand{\utwocell}[3][0.5]{\ar@{}[#2] \ar@{=>}?(#1)+/d 0.2cm/;?(#1)+/u 0.2cm/_{#3}}
\newcommand{\dtwocelltarg}[3][0.5]{\ar@{}#2 \ar@{=>}?(#1)+/u  0.2cm/;?(#1)+/d 0.2cm/^{#3}}
\newcommand{\utwocelltarg}[3][0.5]{\ar@{}#2 \ar@{=>}?(#1)+/d  0.2cm/;?(#1)+/u 0.2cm/_{#3}}

\newdir{(}{{}*!<0em,-.14em>-\cir<.14em>{l^r}}
\newdir{ (}{{}*!/-5pt/\dir{(}}
\newdir{ >}{{}*!/-5pt/\dir{>}}
\newcommand{\sh}[2]{**{!/#1 -#2/}}

\DeclareMathOperator{\ob}{ob}

\newcommand{\cat}[1]{\mathrm{\mathcal #1}}
\newcommand{\thg}{{\mathord{\text{--}}}}

\newcommand{\abs}[1]{{\left|{#1}\right|}}
\newcommand{\dbr}[1]{\left\llbracket{#1}\right\rrbracket}
\newcommand{\res}[2]{\left.{#1}\right|_{#2}}

\newcommand{\spn}[1]{{\langle{#1}\rangle}}

\newcommand{\defeq}{\mathrel{\mathop:}=}
\newcommand{\cd}[2][]{\vcenter{\hbox{\xymatrix#1{#2}}}}

\renewcommand{\phi}{\varphi}
\newcommand{\A}{{\mathcal A}}
\newcommand{\B}{{\mathcal B}}
\newcommand{\C}{{\mathcal C}}
\newcommand{\D}{{\mathcal D}}
\newcommand{\E}{{\mathcal E}}

\newcommand{\M}{{\mathcal M}}
\newcommand{\N}{{\mathcal N}}

\renewcommand{\P}{{\mathcal P}}

\let\sec=\S
\renewcommand{\S}{{\mathcal S}}

\newcommand{\V}{{\mathcal V}}
\newcommand{\W}{{\mathcal W}}

\newcommand{\xtor}[1]{\cdl[@1]{{} \ar[r]|-{\object@{|}}^{#1} & {}}}

\makeatletter

\def\hookleftarrowfill@{\arrowfill@\leftarrow\relbar{\relbar\joinrel\rhook}}
\def\twoheadleftarrowfill@{\arrowfill@\twoheadleftarrow\relbar\relbar}
\def\leftbararrowfill@{\arrowdoublefill@{\leftarrow\mkern-5mu}\relbar\mapstochar\relbar\relbar}
\def\Leftbararrowfill@{\arrowdoublefill@{\Leftarrow\mkern-2mu}\Relbar\Mapstochar\Relbar\Relbar}
\def\leftringarrowfill@{\arrowdoublefill@{\leftarrow\mkern-3mu}\relbar{\mkern-3mu\circ\mkern-2mu}\relbar\relbar}
\def\lefttriarrowfill@{\arrowfill@{\mathrel\triangleleft\mkern0.5mu\joinrel\relbar}\relbar\relbar}
\def\Lefttriarrowfill@{\arrowfill@{\mathrel\triangleleft\mkern1mu\joinrel\Relbar}\Relbar\Relbar}

\def\hookrightarrowfill@{\arrowfill@{\lhook\joinrel\relbar}\relbar\rightarrow}
\def\twoheadrightarrowfill@{\arrowfill@\relbar\relbar\twoheadrightarrow}
\def\rightbararrowfill@{\arrowdoublefill@{\relbar\mkern-0.5mu}\relbar\mapstochar\relbar\rightarrow}
\def\Rightbararrowfill@{\arrowdoublefill@{\Relbar\mkern-2mu}\Relbar\Mapstochar\Relbar\Rightarrow}
\def\rightringarrowfill@{\arrowdoublefill@\relbar\relbar{\mkern-2mu\circ\mkern-3mu}\relbar{\mkern-3mu\rightarrow}}
\def\righttriarrowfill@{\arrowfill@\relbar\relbar{\relbar\joinrel\mkern0.5mu\mathrel\triangleright}}
\def\Righttriarrowfill@{\arrowfill@\Relbar\Relbar{\Relbar\joinrel\mkern1mu\mathrel\triangleright}}

\def\leftrightarrowfill@{\arrowfill@\leftarrow\relbar\rightarrow}
\def\mapstofill@{\arrowfill@{\mapstochar\relbar}\relbar\rightarrow}

\renewcommand*\xleftarrow[2][]{\ext@arrow 20{20}0\leftarrowfill@{#1}{#2}}
\providecommand*\xLeftarrow[2][]{\ext@arrow 60{22}0{\Leftarrowfill@}{#1}{#2}}
\providecommand*\xhookleftarrow[2][]{\ext@arrow 10{20}0\hookleftarrowfill@{#1}{#2}}
\providecommand*\xtwoheadleftarrow[2][]{\ext@arrow 60{20}0\twoheadleftarrowfill@{#1}{#2}}
\providecommand*\xleftbararrow[2][]{\ext@arrow 10{22}0\leftbararrowfill@{#1}{#2}}
\providecommand*\xLeftbararrow[2][]{\ext@arrow 50{24}0\Leftbararrowfill@{#1}{#2}}
\providecommand*\xleftringarrow[2][]{\ext@arrow 10{26}0\leftringarrowfill@{#1}{#2}}
\providecommand*\xlefttriarrow[2][]{\ext@arrow 80{24}0\lefttriarrowfill@{#1}{#2}}
\providecommand*\xLefttriarrow[2][]{\ext@arrow 80{24}0\Lefttriarrowfill@{#1}{#2}}

\renewcommand*\xrightarrow[2][]{\ext@arrow 01{20}0\rightarrowfill@{#1}{#2}}
\providecommand*\xRightarrow[2][]{\ext@arrow 04{22}0{\Rightarrowfill@}{#1}{#2}}
\providecommand*\xhookrightarrow[2][]{\ext@arrow 00{20}0\hookrightarrowfill@{#1}{#2}}
\providecommand*\xtwoheadrightarrow[2][]{\ext@arrow 03{20}0\twoheadrightarrowfill@{#1}{#2}}
\providecommand*\xrightbararrow[2][]{\ext@arrow 01{22}0\rightbararrowfill@{#1}{#2}}
\providecommand*\xRightbararrow[2][]{\ext@arrow 04{24}0\Rightbararrowfill@{#1}{#2}}
\providecommand*\xrightringarrow[2][]{\ext@arrow 01{26}0\rightringarrowfill@{#1}{#2}}
\providecommand*\xrighttriarrow[2][]{\ext@arrow 07{24}0\righttriarrowfill@{#1}{#2}}
\providecommand*\xRighttriarrow[2][]{\ext@arrow 07{24}0\Righttriarrowfill@{#1}{#2}}

\providecommand*\xmapsto[2][]{\ext@arrow 01{20}0\mapstofill@{#1}{#2}}
\providecommand*\xleftrightarrow[2][]{\ext@arrow 10{22}0\leftrightarrowfill@{#1}{#2}}
\providecommand*\xLeftrightarrow[2][]{\ext@arrow 10{27}0{\Leftrightarrowfill@}{#1}{#2}}

\makeatother

\numberwithin{equation}{section}

\theoremstyle{plain}
\newtheorem{Thm}{Theorem}
\newtheorem*{Thm*}{Theorem}
\numberwithin{Thm}{section}
\newtheorem{Prop}[Thm]{Proposition}
\newtheorem{Cor}[Thm]{Corollary}
\newtheorem{Lemma}[Thm]{Lemma}

\theoremstyle{definition}
\newtheorem{Defn}[Thm]{Definition}
\newtheorem{Not}[Thm]{Notation}
\newtheorem{Ex}[Thm]{Example}
\newtheorem{Exs}[Thm]{Examples}
\newtheorem{Rk}[Thm]{Remark}

\newcommand{\CLA}{\mathrm{c\ell}k\cat{Lin}}
\newcommand{\CDiff}{\mathrm{c}\D\mathrm{iff}}
\newcommand{\cpdots}{\makebox[0.8em][c]{.\hfil.\hfil.}}

\newcommand{\faa}{\mathsf{Fa\grave a}}
\newcommand{\kmod}{k\text-\cat{Mod}}
\newcommand{\kmodw}{k\text-\cat{Mod}_w}
\newcommand{\phom}[1]{\llbracket{#1}\rrbracket}

\newcommand{\ots}{\mathord\otimes}
\newcommand{\XB}{\mathrm{B}} \newcommand{\AB}{\mathrm{H}}
\newcommand{\XA}{\mathrm{A}}

\begin{document}
\leftmargini=2em 
\title[Cartesian differential categories as skew enriched categories]{Cartesian differential categories\\as skew enriched categories}
\author{Richard Garner} 
\address{Centre of Australian Category Theory, Macquarie University, NSW, Australia} 
\email{richard.garner@mq.edu.au}
\author{Jean-Simon Pacaud Lemay} 
\address{University of Oxford, Department of Computer Science, UK} 
\email{jean-simon.lemay@kellogg.ox.ac.uk}

\subjclass[2000]{Primary: }
\date{\today}

\thanks{The first author was supported by Australian Research Council
  grants FT160100393 and DP190102432; the second author acknowledges
  the support of Kellogg College, the Clarendon Fund, and the Oxford
  Google-DeepMind Graduate Scholarship.}

\begin{abstract}
  We exhibit the cartesian differential categories of Blute, Cockett
  and Seely as a particular kind of enriched category. The base for
  the enrichment is the category of commutative monoids---or in a
  straightforward generalisation, the category of modules over a
  commutative rig $k$. However, the tensor product on this category is
  not the usual one, but rather a \emph{warping} of it by a certain
  monoidal comonad $Q$. Thus the enrichment base is not
  a monoidal category in the usual sense, but rather a \emph{skew
    monoidal} category in the sense of Szlach\'anyi. Our first main
  result is that cartesian differential categories are the same as
  categories with finite products enriched over this skew monoidal
  base.

  The comonad $Q$ involved is, in fact, an example of a differential
  modality. Differential modalities are a kind of comonad on a
  symmetric monoidal $k$-linear category with the characteristic
  feature that their co-Kleisli categories are cartesian differential
  categories. Using our first main result, we are able to prove our
  second one: that every small cartesian differential category admits
  a full, structure-preserving embedding into the cartesian
  differential category induced by a differential modality (in fact, a
  monoidal differential modality on a monoidal closed category---thus, a
  model of intuitionistic differential linear logic). This
  resolves an important open question in this area.
\end{abstract}

\maketitle

\tableofcontents 

\section{Introduction}
\label{sec:introduction}

This paper brings together two active strands of research in current
category theory. The first is concerned with a certain categorical
axiomatics for differential structure; it originates in the work of
Ehrhard and Regnier on the differential
$\lambda$-calculus~\cite{Ehrhard2003The-differential}, with the
definitive notions of \emph{tensor differential category} and
\emph{cartesian differential category} being identified by Blute,
Cockett and Seely in~\cite{Blute2006Differential,Blute2009Cartesian},
and further studied by the Canadian school of category
theorists~\cite{Blute2015Cartesian,Blute2016Derivations,
  Blute2019Differential, Cockett2011The-Faa-di-Bruno,
 Cruttwell2021Integral,Cockett2020Tangent,Lemay2018A-tangent}. This
has led to novel applications in computer
science~\cite{Bucciarelli2010Categorical, Clift2020Cofree, 
  Cockett2019Categorical,Ehrhard2018An-introduction,
  Fiore2007Differential,Manzonetto2012What} and in other areas such as
abelian functor calculus~\cite{Bauer2018Directional}.

The second strand which informs this paper is the study of \emph{skew
  monoidal categories}, a certain generalisation of Mac~Lane's
monoidal categories. Skew monoidal categories were introduced by
Szlach\'anyi~\cite{Szlachanyi2012Skew-monoidal} with motivation from
quantum algebra, and their general theory has been developed by the
Australian school of category theorists~\cite{Lack2012Skew,
  Lack2014Triangulations, Lack2015Skew-monoidal,
  Street2013Skew-closed}. This has led to novel applications in operad
theory~\cite{Lack2018Operadic}, two-dimensional category theory and
abstract homotopy theory~\cite{Bourke2017Skew}, and computer
science~\cite{Altenkirch2015Monads}.

These two strands meet in the first main result of this paper, which
for the purposes of this introduction we will term the
\emph{enrichment theorem}: it states that the cartesian differential
categories of~\cite{Blute2009Cartesian} are exactly the categories
with finite products \emph{enriched} over a certain skew monoidal
category $\V$. While the notion of a category enriched over a monoidal
category~\cite{Eilenberg1966Closed} is classical, and has been studied
extensively---see, for example,~\cite{Kelly1982Basic}---enrichment
over a skew monoidal base is much less well-developed, having been
considered only in~\cite{Campbell2018Skew-enriched,
  Street2013Skew-closed}, and with fewer compelling examples. We feel
our result clinches the argument for the value of skew enrichment, and
should serve as a useful test-bed for developing the theory further.

Of course, knowing that a certain structure can be exhibited as a kind
of enriched category is not \emph{a priori} useful. However, in
particular cases, it typically is so, and often because it makes
available the presheaf construction, allowing any instance of the
structure at issue to be embedded into a particularly well-behaved
one. This is this case here. Using the presheaf construction for our
enrichment base, we will prove our second main theorem, the
\emph{embedding theorem}, which states that every
cartesian differential category admits a full, structure-preserving
embedding into one induced via the co-Kleisli construction from a
tensor differential category. This answers an important open question
in the area.

In order to describe our results further, we now recall some more
details of the notions involved. We begin on the side of the
differential structures. The key tension here, reflective of the
subject's origins in linear logic, is between axiomatising a
category of \emph{non-linear} (smooth) maps, and a category of
\emph{linear} maps.

The first axiomatisation is perhaps more intuitive, and leads to the
\emph{cartesian differential categories} of~\cite{Blute2009Cartesian}:
these are categories with finite products $\A$ which are endowed with
a \emph{differential} operator providing for each
$f \colon A \rightarrow B$ a new map
$\mathrm{D}f \colon A \times A \rightarrow B$. This $\mathrm{D}f$ is
thought of as assigning to an input pair $(x,v)$ the directional
derivative of $f$ at $x$ in the direction of $v$. To express the
desired linearity of this operation in $v$ needs further structure on
$\A$: we ask that it be \emph{left additive}, meaning that each
hom-set of $\A$ has a commutative monoid structure $(+,0)$ which is
preserved by precomposition, but not necessarily by
postcomposition---this is reasonable since, after all, $\A$ is
supposed to be a category of \emph{non}-linear maps. With the
appropriate axioms, this is the notion of cartesian differential
category.

The second axiomatisation, in terms of a category of linear maps,
leads to the \emph{tensor differential categories}
of~\cite{Blute2006Differential} (there called merely
\emph{differential categories}; we say ``tensor'' to avoid ambiguity).
These are symmetric monoidal, additively enriched categories $\A$
equipped with a \emph{differential modality} ${!}$\,---a comonad on
$\A$ endowed with certain extra structure. Much as in linear logic,
this ${!}$ is intended to allow ``smooth maps'' from $X$ to~$Y$ to be
encoded as ``linear maps''---i.e., $\A$-maps---from ${!}X$ to~$Y$. The
extra structure of ${!}$ which allows this interpretation is
cocommutative coalgebra structure on each ${!}X$, modelling discard
and duplication of non-linear inputs; and a \emph{deriving
 transformation} $\mathsf{d} \colon {!}X \otimes X \rightarrow {!}X$,
precomposition with which implements the differential operator. This
interpretation is justified by the key result that, in the presence of
finite products, \emph{the co-Kleisli category $\cat{Kl}(!)$ of the
 differential modality on a tensor differential category is a
 cartesian differential~category}.

An important refinement of these notions makes explicit the connection
with linear logic. A differential modality is called \emph{monoidal}
if its underlying endofunctor ${!} \colon \A \rightarrow \A$ is (lax)
monoidal, in a manner which is compatible with the rest of the
structure. This makes ${!}$ a model of the exponential modality of
linear logic; if moreover the monoidal structure of $\A$ is
\emph{closed}, then we have a model of 
\emph{intuitionistic differential linear logic}~\cite{Ehrhard2018An-introduction}. In
this case, the co-Kleisli category $\cat{Kl}(!)$ is a cartesian \emph{closed}
differential category, and so a model of the \emph{differential
 $\lambda$-calculus}~\cite{Bucciarelli2010Categorical,
 Ehrhard2003The-differential}. 

With the refinement just noted, the embedding theorem can be stated as
follows:

%

\begin{Thm*}
  Any small cartesian differential category has a full,
  structure-preserving embedding into the co-Kleisli category
  $\cat{Kl}({!})$ of the monoidal differential modality associated to
  a model of intuitionistic differential linear logic.
\end{Thm*}
We will obtain this using our other main result, the
enrichment theorem, and to
describe that we must now turn to the other side of our story: \emph{skew
monoidal categories}. Recall that monoidal
structure~\cite{Mac-Lane1963Natural} on a category $\V$ involves a
unit object $I$, a tensor product functor $\otimes$, and invertible
\emph{coherence constraints}
$\alpha \colon (A \otimes B) \otimes C \rightarrow A \otimes (B
\otimes C)$, $\lambda \colon I \otimes A \rightarrow A$ and
$\rho \colon A \rightarrow A \otimes I$, subject to suitable axioms.
Skew monoidal structure~\cite{Szlachanyi2012Skew-monoidal} generalises
this by dropping invertibility of $\alpha$, $\lambda$ and
$\rho$---being careful to give them the stated orientations and no other.

Many aspects of the theory of monoidal categories can be adapted to
the skew context; in particular, the classical
notion~\cite{Eilenberg1966Closed} of enrichment over a monoidal
category. Following~\cite{Street2013Skew-closed}, a category $\A$
\emph{enriched} over a skew monoidal category $\V$ involves a set of
objects; a hom-object $\A(A,B)$ in $\V$ for every pair of such
objects; and composition and identities morphisms
$\A(B,C) \otimes \A(A,B) \rightarrow \A(A,C)$ and
$I \rightarrow \A(A,A)$ in $\V$, subject to the three usual
associativity and identity axioms---where suitable attention now has
to be paid to orienting these axioms correctly.

We will be interested in enrichment over skew monoidal categories arising in a
particular way. Given a genuine monoidal category
$\V = (\V, \otimes, I)$, one can \emph{warp} it~\cite[\sec
7]{Szlachanyi2012Skew-monoidal} by a monoidal comonad ${!}$ on $\V$ to
obtain a skew monoidal category $\V^! \defeq (\V, \otimes^{!}, I)$, where
$A \otimes^{!} B \defeq A \otimes {!}B$, and where the constraint maps
$\alpha, \lambda, \rho$ for $V^{!}$ come from those for $\V$ and  the structure maps of the monoidal comonad ${!}$.

A first indicator of the relevance of these ideas to cartesian
differential categories is the following observation, made by Cockett
and Lack in 2012, and recorded in passing in~\cite[\sec
5.1]{Blute2015Cartesian}. Consider the monoidal category $\cat{C\M on}$ of
commutative monoids with its usual tensor product. There is a monoidal
comonad $K$ induced by the (monoidal) forgetful--free adjunction
$\cat{C\M on} \leftrightarrows \cat{Set}$, with action on objects 
\begin{equation*}
  K(A) = \textstyle\bigoplus_{a \in A} \mathbb{N}\rlap{ ,}
\end{equation*}
and it is not hard to see that a category enriched over the
skew-warping $\cat{C\M on}^K$ is exactly a \emph{left additive}
category. Our enrichment theorem takes this observation further. It
turns out that to get from left additive to cartesian differential
structure, the key step is to replace $K$ with a more elaborate
monoidal comonad $Q$ on $\cat{C\M on}$, which acts on objects by
\begin{equation*}
  Q(A) = \textstyle\bigoplus_{a \in A} \mathrm{Sym}(A)
\end{equation*}
where $\mathrm{Sym}(A)$ is the free commutative rig (=semiring) on the 
commutative monoid~$A$. This $Q$ is not just a monoidal comonad, but
also a monoidal differential modality; in fact, it turns out to be the \emph{initial}
monoidal differential modality on $\cat{C\M on}$, and so our
enrichment theorem can be stated as:

\begin{Thm*}
  To give a cartesian differential category is equally to give a
  $\cat{C\M on}^Q$-enriched category with finite products, where $Q$
  is the initial monoidal differential modality on $(\cat{C\M on},
  \otimes, \mathbb{N})$.
\end{Thm*}

We derive the embedding theorem from the enrichment theorem via the mechanism
advertised above: enriched presheaves. As explained
in~\cite{Street2013Skew-closed}, the presheaf
construction for enrichment
over a monoidal category~\cite{Kelly1982Basic, Street1983Enriched} generalises without difficulty to the skew
monoidal case. Thus, for a small cartesian differential category $\A$,
seen as a $\cat{C\M on}^Q$-enriched category, its enriched Yoneda
embedding $\A \rightarrow \cat{Psh}(\A)$ corresponds to a full
structure-preserving embedding of cartesian differential categories. A
deeper analysis shows that, in fact, the cartesian differential
category $\cat{Psh}(\A)$ is induced from a monoidal differential modality on a
symmetric monoidal closed additively enriched category, so yielding our
embedding theorem. Since the proof is entirely constructive, we are
able to compute a concrete description of all aspects of the embedding
so obtained; and these are delicate enough that there seems to be little
chance of having arrived at them by any other means---so justifying
our approach.

Let us also say a few words about the proof of the enrichment theorem.
Perhaps the most interesting point is the manner in which the initial
monoidal differential modality $Q$ comes into the picture. One point
of reference is that the formula
$QA = \bigoplus_{a \in A} \mathrm{Sym}(A)$ is the same formula as for
the cofree cocommutative coalgebra over an algebraically closed field
$k$; see~\cite{Clift2020Cofree,Sweedler1969Hopf}. However, our
motivation comes from the striking~\cite{Cockett2011The-Faa-di-Bruno},
which proves that the forgetful functor from cartesian differential
categories to cartesian left additive categories has a \emph{right}
adjoint. The value of this right adjoint at $\A$ is the so-called
\emph{Fa\`a di Bruno category} $\faa(\A)$, whose objects are those of
$\A$; whose maps $f^{(\bullet)} \colon A \rightsquigarrow B$ are
$\mathbb{N}$-indexed families of maps
$f^{(n)} \colon A \times A^n \rightarrow B$ in $\A$ which are
symmetric multilinear in their last $n$ variables; and whose
composition law is given by the higher-order chain rule, the so-called
\emph{Fa\`a di Bruno} formula. 
This is analogous to the fact that the forgetful functor to
commutative rings from \emph{differential rings}---commutative rings
endowed with a derivation---has a right adjoint, which sends a ring
$R$ to its ring of Hurwitz series~\cite{Keigher1997On-the-ring}; this is the
ring whose elements are $\mathbb{N}$-indexed families of elements $r^{(n)} \in R$, endowed with
a suitable multiplication.

In particular, we may look at one of the most basic cartesian left
additive categories, the category $\cat{C\M on}_w$ of commutative
monoids and arbitrary functions, and construct its cofree cartesian
differential category $\faa(\cat{C\M on}_w)$. A natural question is
whether this is induced as the co-Kleisli category of a differential
modality on a symmetric monoidal additively enriched category. The
answer turns out to be \emph{yes}---with the differential modality
involved being precisely the initial differential modality $Q$ on the
category $\cat{C\M on}$ of commutative monoids.

We conclude this introduction with a brief overview of the contents of
the paper. In Section~\ref{sec:cart-diff-categ}, we recall the notion
of cartesian differential category, and give a range of examples. In
Section~\ref{sec:faa-di-bruno}, we describe the Fa\`a di Bruno
construction of~\cite{Cockett2011The-Faa-di-Bruno}, and give new
proofs of the main results of \emph{loc.~cit.} In
Section~\ref{sec:cokleisli}, we recall the notion of tensor
differential category and its relation to cartesian differential
structure, before showing that every Fa\`a di Bruno category
$\faa(\A)$ arises via a co-Kleisli construction (this will later
follow from the embedding theorem). Section~\ref{sec:skew-enrichment}
develops some of the basics of skew monoidal categories and enrichment
over a skew monoidal base, before Section~\ref{sec:cart-diff-categ-2}
provides the proof of our first main result, the enrichment theorem.
Section~\ref{sec:second-interl-enrich} then develops the theory of
\emph{presheaves} for categories enriched over a skew monoidal base;
before, finally, Section~\ref{sec:appl-an-embedd} exploits this and
our first main result to prove the embedding theorem for cartesian
differential categories.

\section{Cartesian differential categories}
\label{sec:cart-diff-categ}

In this purely expository section, we recall the necessary background
from~\cite{Blute2009Cartesian} on cartesian differential categories. As
explained in the introduction, a cartesian differential category is a
category endowed with an abstract notion of differentiation; the
motivating example is the category whose objects are Euclidean spaces
$\mathbb{R}^n$ and whose maps are smooth functions, but there are
other examples coming from algebraic geometry and linear logic, which
we will recall below.

\subsection{Cartesian \texorpdfstring{left-$k$-linear}{left-k-linear} categories}
\label{sec:left-k-linear-1}

In the introduction, we saw that cartesian differential structure on a
category involved commutative monoid structure $(+,0)$ on each
hom-set. For examples coming from differential or algebraic geometry,
this can generally be enhanced to the structure of a real or complex
vector space, or at least that of an $R$-module over a commutative
ring $R$. This is by contrast with examples coming from linear logic,
where negatives may not exist all.

To account for these distinctions, we will incorporate into the
definitions that follow the parameter of a commutative
\emph{rig} (or \emph{semiring}) of scalars
$k$. 
When $k = \mathbb{N}$ we re-find the original
definitions of~\cite{Blute2006Differential, Blute2009Cartesian}; when
$k=\mathbb{Z}$ we get variants with negatives; when $k$ is a field
we get versions involving $k$-vector spaces; and so on.

For the rest of the paper, then, $k$ will be a fixed commutative rig.
We write $\kmod$ for the category of modules over $k$, whose objects
are commutative monoids $M$ (written additively) with a multiplicative
action $k \times M \rightarrow M$ that respects addition in each
variable, and whose maps are \emph{$k$-linear} maps, i.e., functions
respecting the additive structure and the $k$-action. As with modules
over a commutative ring, we have a tensor product $\otimes$ on $\kmod$
which classifies bilinear maps, and this forms part of a symmetric
monoidal structure on $\kmod$ with unit $k$. We also have all limits
and colimits, but in particular \emph{finite biproducts} $M \oplus N$,
computed as cartesian products at the underlying set level.

The following is the basic notion on which the definition of
($k$-linear) cartesian differential category will be built.

\begin{Defn}
 \label{def:4}
 A \emph{left-$k$-linear category} is a category $\A$ in which each
 hom-set $\A(A,B)$ is endowed with $k$-module structure, and for each
 $f \in \A(A,B)$, the function
 $(\thg) \circ f \colon \A(B, C) \rightarrow \A(A,C)$ is $k$-linear.
 A left-$k$-linear category $\A$ is \emph{cartesian} if its
 underlying category has finite products, and the binary product
 isomorphisms as below are $k$-linear
 \begin{equation}\label{eq:1}
  \A(A,B \times C) \rightarrow \A(A,B) \oplus \A(A,C)\rlap{ .}
 \end{equation}
\end{Defn}

The notion of left-$k$-linear category should be compared with that of
\emph{$k$-linear category}, in which we also require that each
function $g \circ (\thg) \colon \A(A,B) \rightarrow \A(A,C)$ is
$k$-linear. While the basic example of a $k$-linear category is
$\kmod$, the basic example of a left-$k$-linear category is $\kmodw$,
the category of $k$-modules and \emph{arbitrary} maps. In this case,
the $k$-module structure on homs is given pointwise, and is still
preserved by precomposition, but not necessarily by postcomposition.

In fact, those maps of $\kmodw$ by which postcomposition does preserve
the $k$-module structure are precisely the ones lying in $\kmod$. This
motivates:

\begin{Defn}
 \label{def:5}
 A map $g \colon B \rightarrow C$ in a left-$k$-linear category $\A$
 is called $k$-\emph{linear} if, for every $A \in \A$, the function
 $g \circ (\thg) \colon \A(A,B) \rightarrow \A(A,C)$ is $k$-linear.
 More generally, a map
 $g \colon B_1 \times \dots \times B_n \rightarrow C$ in a cartesian
 left-$k$-linear category $\A$ is said to be \emph{$k$-linear in the
  $i$th variable} if, for each $A \in \A$, the function
 \begin{equation*}
  \A(A, B_1) \times \dots \times \A(A, B_n) \xrightarrow{\cong} 
  \A(A, B_1 \times \dots \times B_n) \xrightarrow{g \circ (\thg)} \A(A,C)
 \end{equation*}
 is $k$-linear in its $i$th argument. If $g$ is linear in all $n$ of
 its arguments separately, we may say that it is
 \emph{$k$-multilinear} or simply \emph{multilinear}.
\end{Defn}

We make three remarks. Firstly, for a map
$f \colon A_1 \times A_2 \times A_3 \rightarrow B$, say, we could ask
that it be $k$-linear in the first variable, and also $k$-linear in
the third variable---so a kind of bilinearity---but also that it be
linear in variables $1$ and $3$ taken together, i.e., that
$f \circ (g_0+g_1, h, k_0+k_1) = f \circ (g_0, h, k_0) + f \circ (g_1,
h, k_1)$. In the latter situation, we may say that $f$ is
\emph{jointly $k$-linear} in variables $1$ and $3$.

Secondly, we note that in~\cite{Blute2009Cartesian}, ``additive'' is
used for what we call ``$k$-linear'', while ``linear'' is reserved for
a stronger concept which we call ``$\mathrm{D}$-linear''; see
Definition~\ref{def:6}.\footnote{This change of nomenclature will be
  justified later, when we see that both $k$-linearity and
  $\mathrm{D}$-linearity of a map are examples of a general notion of
  ``linearity'' with respect to a skew enrichment; see
  Example~\ref{ex:15} and Theorem~\ref{thm:3}.} Finally, we can now equate our
notion of \emph{cartesianness} for a left-$k$-linear category with
that in~\cite{Blute2009Cartesian}. Indeed, to require the
$k$-linearity of the product isomorphisms in~\eqref{eq:1} is equally
to require the $k$-linearity of the two product projections
$\pi_0 \colon B \times C \rightarrow B$ and
$\pi_1 \colon B \times C \rightarrow C$---which,
by~\cite[Lemma~2.4]{Lemay2018A-tangent}, is equivalent to the original
definition of ``cartesian''.
 
We conclude this section with some examples of cartesian
left-$k$-linear categories.

\begin{Exs}
 \begin{enumerate}[(i),itemsep=0.35\baselineskip]
 \item As already noted, $\kmodw$ is a left-$k$-linear category; in
  fact, it is also cartesian, with finite products computed as in
  $\kmod$. Note that in $\kmodw$, these finite products are
  \emph{not} biproducts, and so we will write them as $A \times B$
  rather than $A \oplus B$. The $k$-multilinear maps in $\kmodw$ are
  multilinear maps in the usual sense.
   
\item A $k$-linear category with finite biproducts, such as $\kmod$,
  is the same thing as a cartesian left-$k$-linear category in which
  every map is $k$-linear. On the other hand, the only $k$-multilinear
  maps in such a category are the zero maps.

 \item The category $\cat{Smooth\E uc}$, whose objects are the
  Euclidean spaces $\mathbb{R}^n$, and whose maps are smooth
  functions, is cartesian left-$\mathbb{R}$-linear. Once again, the
  $\mathbb{R}$-linear and $\mathbb{R}$-multilinear maps take on
  their usual meaning, and finite products are simply cartesian
  products.


 \item \label{item:1} If $k$ is a commutative rig, we define the category
  $\cat{Poly}_k$ to have natural numbers as objects,
  maps $f \colon n \rightarrow m$ given by $m$-tuples of polynomials
  $f_1, \dots, f_m \in k[x_1, \dots, x_n]$, and composition given by
  polynomial substitution. $\cat{Poly}_k$ is left-$k$-linear under the
  pointwise addition of polynomials; it is moreover cartesian, with
  finite products given by addition of natural numbers, and
  ($k$-linear) projection maps ${n \leftarrow n+m \rightarrow m}$
  given by $\pi_0 = (x_1, \dots, x_n)$ and
  $\pi_1 = (x_{n+1}, \dots, x_m)$. The $k$-linear maps
  $f \colon n \rightarrow m$ in $\cat{Poly}_k$ are those for which
  each $f_i$ is of the form $\lambda_1 x_1 + \dots + \lambda_m x_m$
  for some $\lambda_1, \dots, \lambda_m \in k$; the $k$-bilinear maps
  $f \colon n+m \rightarrow r$ are likewise those for which each $f_i$
  is a $k$-linear combination of monomials $x_i x_j$ with $1 \leqslant
  i \leqslant n < j \leqslant n+m$.

 %
 \item \label{item:3} Generalising~\ref{item:1}, we have a category $\cat{Gen\P oly}_k$ with
  $k$-modules $M,N,\dots$ as objects, and maps from $N$ to $M$ being
  $k$-module maps $M \rightarrow \mathrm{Sym}(N)$, where
  \begin{equation}\label{eq:13}
   \mathrm{Sym}(N) = k \oplus N \oplus (N \otimes N) / \mathfrak{S}_2
   \oplus (N \otimes N \otimes N) / \mathfrak{S}_3 \oplus \cdots
  \end{equation}
  is the free symmetric $k$-algebra on $N$. Since $\mathrm{Sym}$ is a monad on
  $\kmod$, composition in $\cat{Gen\P oly}_k$ can be described as
  Kleisli composition for $\mathrm{Sym}$. Proceeding as before, \emph{mutatis
   mutandis}, yields a cartesian left-$k$-linear structure on
  $\cat{Gen\P oly}_k$, whose $k$-linear maps from $N$ to $M$ are maps
  $M \rightarrow \mathrm{Sym}(N)$ in $\kmod$ which factor through the unit
  $\eta \colon N \rightarrow \mathrm{Sym}(N)$; and whose bilinear maps from
  $N \times M$ to $R$ are maps in $\kmod$ of the form
  \begin{equation*}
   R \rightarrow N \otimes M \xrightarrow{\eta \otimes \eta}
   \mathrm{Sym}(N) \otimes \mathrm{Sym}(M) \cong \mathrm{Sym}(N
   \oplus M)\rlap{ .}
  \end{equation*}
 \end{enumerate}
 \label{ex:5}
\end{Exs}

\subsection{Cartesian differential categories}
\label{sec:cart-diff-categ-1}

We now recall the key notion from~\cite{Blute2009Cartesian}. Note that
we omit in (iii) the condition
$\mathrm{D}(f,g) = (\mathrm{D}f, \mathrm{D}g)$ which was originally taken
as axiomatic, since by~\cite[Lemma~2.8]{Lemay2018A-tangent}, this
follows from (iii) and (iv).

\begin{Defn}
 \label{def:8}
 A \emph{$k$-linear cartesian differential category} is a cartesian
 left-$k$-linear category $\A$ equipped with operators
 $\mathrm{D} \colon \A(A,B) \rightarrow \A(A \times A, B)$ such that:
 \begin{enumerate}[(i)]
 \item Each $\mathrm{D}$ is $k$-linear;
 \item Each $\mathrm{D}f \colon A \times A \rightarrow B$ is
  $k$-linear in its second argument;
 \item $\mathrm{D}(\pi_i) = \pi_i \pi_1 \colon (A_0 \times A_1)
  \times (A_0 \times A_1)
  \rightarrow A_i$ for $i = 0,1$;
 \item $\mathrm{D}(1_A) = \pi_1 \colon A \times A \rightarrow A$ for all $A \in \A$;
 \item $\mathrm{D}(gf) = \mathrm{D}g\bigl(f\pi_0,\mathrm{D}f\bigr)
  \colon A \times A \rightarrow C$ for
  all $f \colon A \rightarrow B$ and
  $g \colon B \rightarrow C$.
 \item $\mathrm{D}(\mathrm{D}f)(x,r,0,v) = \mathrm{D}f(x,v) \colon Z
  \rightarrow B$ for all
  $x,r,v \colon Z \rightarrow A$, $f \colon A \rightarrow B$;
 \item
  $\mathrm{D}(\mathrm{D}f)(x,r,s,0) =
  \mathrm{D}(\mathrm{D}f)(x,s,r,0) \colon Z \rightarrow B$ for all
  $x,r,s \colon Z \rightarrow A$, $f \colon A \rightarrow B$.
 \end{enumerate}
\end{Defn}


In the examples that follow, we emphasise the ground rig
$k$. However, subsequently we will typically write ``cartesian
differential category'' to mean
``$k$-linear cartesian differential category'', leaving the parameter
$k$ implicit.

\begin{Exs}
 \begin{enumerate}[(i),itemsep=0.35\baselineskip]
 \item $\cat{Smooth\E uc}$ is an $\mathbb{R}$-linear cartesian differential category,
  where for a smooth map
  $f \colon \mathbb{R}^n \rightarrow \mathbb{R}^m$, we take
  $\mathrm{D}f \colon \mathbb{R}^n \times \mathbb{R}^n \rightarrow
  \mathbb{R}^m$ to be the directional derivative
  \begin{equation*}
   \mathrm{D}f(\mathbf x, \mathbf v) = (\boldsymbol \nabla f)(\mathbf{x}) \cdot \mathbf{v}
  \end{equation*}
  where $(\boldsymbol \nabla f)(\mathbf{x})$ is the (vector-valued)
  gradient
  $(\res{\smash{\tfrac{\partial f}{\partial x_1}}}{\mathbf{x}}\ \dots\
  \res{\smash{\tfrac{\partial f}{\partial x_n}}}{\mathbf{x}})$.
 \item \label{item:2} $\cat{Poly}_k$ is a $k$-linear cartesian differential
  category, where for each map $f \colon n \rightarrow m$ we define
  $\mathrm{D}f \colon n+n \rightarrow m$ by
  \begin{equation*}
   (\mathrm{D}f)(x_1, \dots, x_n, v_1, \dots, v_n) = \Bigl(\textstyle\sum_{j=1}^n
   \tfrac{\partial f_1}{\partial x_j} v_j, \dots, \sum_{j=1}^n
   \tfrac{\partial f_m}{\partial x_j} v_j\Bigr)
  \end{equation*}
 \item $\cat{Gen\P
   oly}_k$ has a $k$-linear cartesian differential structure defined in a
  manner which extends~\ref{item:2}; we will obtain it
  rigorously in Examples~\ref{ex:7}\ref{item:6} below.
\item \label{item:4} Every $k$-linear category with finite biproducts
  is a $k$-linear cartesian differential category, where for a map
  $f\colon A \rightarrow B$ we define
  $\mathrm{D}f\colon A \oplus A \rightarrow B$ by
  $\mathrm{D}f = f \pi_1$. While this example may seem trivial, it
  plays an important role in \cite{Cockett2020Tangent}.
 \end{enumerate}
 \label{ex:10}
\end{Exs}

The axioms for a cartesian differential category express axiomatically
the key properties of the motivating example of $\cat{Smooth\E uc}$.
(i) expresses the linearity of taking derivatives, and (iii) the
compatibility of $\mathrm{D}$ with products; (iv) and (v) are the
nullary and binary chain rules; while (vii) gives symmetry of
second-order partial derivatives. As for (ii) and (vi), \emph{both}
express the linearity of the operation
$(\boldsymbol \nabla f)(\mathbf{x}) \cdot (\thg)$, but in different
ways. We have already discussed $k$-linearity, but in the differential
context, we also have a notion of \emph{$\mathrm{D}$-linearity}. In
the definition, and henceforth, we write $0^m$ for a sequence of $m$
zeroes.

\begin{Defn}
 \label{def:6}
 A map $f \colon A \rightarrow B$ in a cartesian differential
 category is \emph{$\mathrm{D}$-linear} if $\mathrm{D}f(x,v) = fv$
 for all $x,v \colon Z \rightarrow A$. More generally, a map
 $f \colon A_1 \times \dots \times A_n \rightarrow B$ is
 \emph{$\mathrm{D}$-linear in the $i$th variable} if for all suitable
 $v, x_1, \dots, x_n$ we have
 \begin{equation*}
  \mathrm{D}f(x_1, \dots, x_n,0^{i-1}, v, 0^{n-i}) = f(x_1, \dots,
  x_{i-1}, v, x_{i+1}, \dots, x_n)\rlap{ .}
 \end{equation*}
\end{Defn}

In these terms, axiom (vi) above says exactly that $\mathrm{D}f$ is
$\mathrm{D}$-linear in its first argument. In the motivating example,
$\mathrm{D}$-linearity and $k$-linearity coincide, but in the general
case, the former implies the latter, but not vice versa;
see~\cite[Corollary~2.3.4]{Blute2009Cartesian}.

The notion of $\mathrm{D}$-linearity in one variable is conveniently
repackaged using \emph{partial derivatives}, which will be important
later. In terms of the following definition,
$f \colon A_1 \times \dots \times A_n \rightarrow B$ is
$\mathrm{D}$-linear in its $i$th variable just when we have
$\mathrm{D}_i f(x_1, \dots, x_n, v) = f(x_1, \dots, x_{i-1}, v,
x_{i+1}, \dots, x_n)$.

\begin{Defn}
 \label{def:11}
 Given a map $f \colon A_1 \times \dots \times A_n \rightarrow B$ in
 a cartesian differential category and $1 \leqslant i \leqslant n$,
 its \emph{$i$th partial derivative} is the map
 $\mathrm{D}_i f \colon A_1 \times \dots \times A_n \times A_i
 \rightarrow B$ characterised by
 $\mathrm{D}_i f(x_1, \dots, x_n, v) = \mathrm{D}f(x_1, \dots, x_n,
 0^{i-1}, v, 0^{n-i})$.
\end{Defn}

For example, in $\cat{Poly}_k$ the partial derivative
$\mathrm{D}_i f \colon n+1 \rightarrow m$ of a map
$f \colon n \rightarrow m$ is given by
$(\mathrm{D}_if)(x_1, \dots, x_n, v) = (\tfrac{\partial f_1}{\partial
 x_i} v, \dots, \tfrac{\partial f_m}{\partial x_i} v)$. Comparing
this with Examples~\ref{ex:10}\ref{item:2}, we see that in this case
the derivative $\mathrm{D}f$ is the sum of the partial derivatives.
This is true in general, as the first part of the following lemma
shows.

\begin{Lemma}
 \label{lem:5}
 Let $f \colon A_1 \times \dots \times A_n \rightarrow B$ and
 $g \colon B \rightarrow C$ be maps in a cartesian differential
 category.
 \begin{enumerate}[(i)]
 \item We have
  $\mathrm{D}f = \mathrm{D}_1 f + \dots + \mathrm{D}_n f$;
 \item We have
  $\mathrm{D}_i(gf)(x_1, \dots, x_n, v) = \mathrm{D}g\bigl(f(x_1,
  \dots, x_n), \mathrm{D}_if(x_1, \dots, x_n, v)\bigr) $ for all
  suitable $x_1, \dots, x_n, v$.
 \end{enumerate}
\end{Lemma}

\begin{proof}
 Part (i) is~\cite[Lemma~4.5.1]{Blute2009Cartesian}. For (ii), we
 calculate using the chain rule that
 \begin{align*}
  \mathrm{D}_i(gf)(x_1, \dots, x_n, v) &=
  \mathrm{D}(gf)(x_1, \dots, x_n, 0^{i-1}, v, 0^{n-i}) \\
  &= \mathrm{D}g\bigl(f(x_1, \dots, x_n), \mathrm{D}f(x_1, \dots, x_n,
  0^{i-1}, v, 0^{n-i})\bigr) \\& = \mathrm{D}g\bigl(f(x_1, \dots,
  x_n), \mathrm{D}_if(x_1, \dots, x_n, v)\bigr)\text{ .}\qedhere
 \end{align*}
\end{proof}

Finally, we record the definition of cartesian \emph{closed}
differential category. In~\cite{Bucciarelli2010Categorical},
this structure is called a ``differential $\lambda$-category'', and is shown to
admit an interpretation of the differential
$\lambda$-calculus of~\cite{Ehrhard2003The-differential}.

\begin{Defn}
 \label{def:37}
 A cartesian differential category $\A$ is a \emph{cartesian closed
  differential category} if its underlying category is cartesian
 closed, and one of the following two equivalent conditions holds
 (where bar indicates exponential transpose):
 \begin{enumerate}[(i)]
 \item For all $f \colon A \times B \rightarrow C$ in $\A$, we have
  $\mathrm{D}\overline f = \overline{\mathrm{D}_1 f} \colon A \times A
  \rightarrow C^B$;
 \item For all $B,C \in \A$, the counit $\mathrm{ev} \colon C^B \times B
  \rightarrow C$ is $\mathrm{D}$-linear in its first argument.
 \end{enumerate}
\end{Defn}
For the equivalence of these conditions,
see~\cite[Lemma~4.10]{Cockett2019Categorical}.

\section{The Fa\`a di Bruno construction}
\label{sec:faa-di-bruno}

In this section, we describe the striking main result
of~\cite{Cockett2011The-Faa-di-Bruno}. This says that the forgetful
functor $U \colon \CDiff \rightarrow \CLA$ from cartesian differential
categories to cartesian left-$k$-linear categories---with the obvious
strict structure-preserving maps in each case---has a right adjoint
and is comonadic. 
As in~\cite{Cockett2011The-Faa-di-Bruno}, we will denote the value of
this right adjoint at a cartesian left-$k$-linear category $\A$ by
$\faa(\A)$, and call it the \emph{Fa\`a di Bruno category} of $\A$.

The calculations which describe the Fa\`a di Bruno category, and
exhibit its universal property, are so closely aligned to what we need
in this paper that it will be worth going through them thoroughly. In
fact, this will not be pure revision: we give \emph{new} proofs of the
main results of~\cite{Cockett2011The-Faa-di-Bruno} that avoid the term
calculus for cartesian differential categories, and which sidestep
some of the more involved calculations.

\subsection{Objects and morphisms}
\label{sec:faa-di-bruno-1}
The notions of cartesian left-$k$-linear category, and of cartesian
differential category, are \emph{essentially algebraic} in the sense
of Freyd~\cite{Freyd1972Aspects}, and the 
functor $U \colon \CDiff \rightarrow \CLA$ is given by forgetting certain essentially-algebraic
structure. 
It follows by the standard theory~\cite{Gabriel1971Lokal} of
locally finitely presentable categories that $U$ has a left adjoint
$F$, and is monadic.

By contrast, the fact that $U$ has a \emph{right} adjoint $\faa$ does
not follow from any standard theory; however, to discover the values
that such a right adjoint would have to take, we can employ a standard
methodology---namely, that of ``probing'' from suitable free objects
and using adjointness. In this section, we use this approach to find
out what the objects and morphisms of $\faa(\A)$ must be.


First, let $\mathbf{1}$ be the free cartesian left-$k$-linear category
on an object, and $F(\mathbf{1})$ the free cartesian differential
category on that. Objects of $\faa(\A)$ are in bijection with maps
$F(\mathbf{1}) \rightarrow \faa(\A)$ in $\CDiff$, and so with maps
$UF(\mathbf{1}) \rightarrow \A$ in $\CLA$. But since the only
morphisms of $\mathbf{1} \in \CLA$ are ones which must be
$\mathrm{D}$-linear in any cartesian differential category,
$F(\mathbf{1})$ is simply $\mathbf{1}$ with the trivial differential
of Example \ref{ex:10}\ref{item:4}, and $UF(\mathbf{1})$ is again just
$\mathbf{1}$. Therefore \emph{objects of $\faa(\A)$ are the same as
  those of $\A$}.

Now let $\mathbf{2}$ be the free cartesian left-$k$-linear category
on an arrow $f \colon A \rightarrow B$. Arguing as before, arrows of
$\faa(\A)$ are in bijection with maps $UF(\mathbf{2}) \rightarrow \A$
in $\CLA$; to identify these, we must give a presentation of
$F(\mathbf{2})$ \emph{qua} cartesian left-$k$-linear category. Of
course, part of this presentation is the arrow
$f \colon A \rightarrow B$; but we also have
$\mathrm{D}f \colon A^2 \rightarrow B$ and
$\mathrm{D}^2f \colon A^4 \rightarrow B$, and so on. It turns
out\footnote{We will not prove this directly, since this discussion is
  really only by way of motivation.} that the totality of the maps
$\mathrm{D}^n f \colon A^{2^n} \rightarrow B$ generate $F(\mathbf{2})$
as a cartesian left-$k$-linear category; as such, arrows
$A \rightarrow B$ of $\faa(\A)$ can be identified with families of
maps
\begin{equation*}
  (f_0 \colon A \rightarrow B, f_1 \colon A^2 \rightarrow B, \dots,
  f_n \colon A^{2^n} \rightarrow B, \dots)
\end{equation*}
in $\A$ subject to axioms expressing the relations between
$f, \mathrm{D}f, \mathrm{D}^2f, \dots$ in $F(\mathbf{2})$. This is, in
fact, the description of maps of $\faa(\A)$ given
in~\cite{Lemay2018A-tangent}, but \emph{not} the original one
of~\cite{Cockett2011The-Faa-di-Bruno}. To reconstruct the latter, we
must look more closely at the relations holding between the iterated
differentials of a map.


As motivation, we observe that, for the second iterated differential,
we have by axiom (vi) and Lemma~\ref{lem:5} that:
\begin{equation*}
  (\mathrm{D}^2f)(x,r,s,v) = (\mathrm{D}_1\mathrm{D}f)(x,r,s) +
  (\mathrm{D}_2\mathrm{D}f)(x,r,v) = (\mathrm{D}_1\mathrm{D}f)(x,r,s) + \mathrm{D}f(x,v)\text{ ;}
\end{equation*}
this abstracts away the expression of $\mathrm{D}^2f$ in
$\cat{Smooth\E uc}$ in terms of the Jacobian and the Hermitian:
$(\mathrm{D}^2f)(\mathbf{x},\mathbf{r},\mathbf{s},\mathbf{v}) =
\mathbf{r}^\top \cdot (\mathbf{H}f)(\mathbf{x}) \cdot \mathbf{s} +
(\boldsymbol\nabla f)(\mathbf{x}) \cdot \mathbf{v}$. More generally,
we can decompose iterated differentials $\mathrm{D}^n f$ as sums of
\emph{higher-order derivatives}:

\begin{Defn}
  \label{def:12}(\cite[\sec 3.1]{Cockett2011The-Faa-di-Bruno}) For any
  map $f \colon A \rightarrow B$ in a cartesian differential category,
  we define its \emph{$n$th derivative} as the map
  $f^{(n)} = (\mathrm{D}_1)^n f \colon A \times A^n \rightarrow B$.
\end{Defn}

We now give the decomposition of
$\mathrm{D}^n f \colon A^{2^n} \rightarrow B$ in terms of the
$f^{(n)}$'s. In what follows, given a map
$x \colon X \rightarrow A^{2^n}$ and a subset
$I \subseteq [n] = \{1, \dots, n\}$, we write
$x_I \colon X \rightarrow A$ for the projection of $x$ associated to
the characteristic function $\chi_I \in 2^n$. 

\begin{Lemma}
  \label{lem:3}
  Let $f \colon A \rightarrow B$ be a map in a cartesian differential
  category and $n \geqslant 0$.
  \begin{enumerate}[(i)]
  \item $f^{(n)} \colon A \times A^n \rightarrow B$ is symmetric and
    $\mathrm{D}$-linear in its last $n$ variables.
  \item \label{item:7} For any $x \colon X \rightarrow A^{2^n}$, we have
    \begin{equation}\label{eq:4}
      (\mathrm{D}^n f)(x) = \sum_{[n] = A_1 \mid \dots \mid A_k}
      f^{(k)}(x_{\emptyset},x_{A_1}, \dots, x_{A_k})
    \end{equation}
    where the sum is over all (unordered) partitions of $[n]$ into
    non-empty subsets $A_1, \dots, A_k$; in particular, when $n = 0$,
    the unique partition of $[0] = \emptyset$ is the \emph{empty}
    partition with $k = 0$.
  \item \label{item:8} For any $y \colon X \rightarrow A \times A^n$, we have
       \begin{equation}\label{eq:4.1}
    f^{(n)}(y) =  (\mathrm{D}^n f)(y^\circ)
    \end{equation}
    where $y^\circ \colon X \rightarrow A^{2^n}$ is the unique map
    with $y^\circ_\emptyset = \pi_0 y$, 
    $y^\circ_{\lbrace k \rbrace} = \pi_k y$ and
    $y^\circ_I = 0$ for any $I \subseteq [n]$ of cardinality
    $\geqslant 2$.
  \end{enumerate}
\end{Lemma}

\begin{proof}
  (i) is~\cite[Lemma~3.1.1(iii) \&
  Corollary~3.1.2]{Cockett2011The-Faa-di-Bruno}. For (ii), consider
  maps $x_0, \dots, x_n, y_0, \dots, y_n \colon X \rightarrow A$;
  writing $\vec x = x_0, \dots, x_n$ and $\vec y = y_0, \dots, y_n$ we
  have
  \begin{equation}
    \label{eq:6}
    \mathrm{D}(f^{(n)})(\vec x, \vec y) =
    \textstyle \sum_{i=0}^n
    \mathrm{D}_{i+1}(f^{(n)})(\vec x, y_i) =
    f^{(n+1)}(\vec x, y_0) +
    \sum_{i=1}^n f^{(n)}(\vec x[y_i/x_i])
  \end{equation}
  ---where $\vec x[y_i/x_i]$ indicates the substitution of $y_i$ for
  $x_i$ in $\vec x$---by using Lemma~\ref{lem:5} at the first step, and
  the $\mathrm{D}$-linearity of $f^{(n)}$ in its last $n$ variables at
  the second. We now prove~\eqref{eq:4} by induction. The base cases
  $n = 0,1$ are clear. So let $n \geqslant 2$ and assume the result for
  $n-1$. We calculate $\mathrm{D}^n (f)(x)$ to be given by
  \begin{align*}
    &
    \sum_{[n\text-1] = A_1 \mid \dots \mid A_k} \!\!\!\!\!
    \mathrm{D}(f^{(k)})(x_\emptyset,x_{A_1}, \cpdots, x_{A_k}, x_{\{n\}}, x_{A_1
      \cup \{n\}}, \cpdots, x_{A_k \cup \{n\}}) \\
    &= \!\!\!\!\!
    \sum_{[n\text-1] = A_1 \mid \dots \mid A_k} \!\!\!\!\! \bigl(
    f^{(k+1)}(x_{\vec A}, x_{\{n\}}) +
    \textstyle\sum_{i=1}^k f^{(k)}(x_{\vec A}[x_{A_i \cup \{n\}} /
    x_{A_i}])\bigr)
    = \displaystyle\sum_{[n] = A_1 \mid \dots \mid A_k} \!\!\!\!\!
    f^{(k)}(x_{\vec A})
  \end{align*}
  as desired. Here, we write $x_{\vec A}$ for
  $x_\emptyset, x_{A_1}, \dots, x_{A_k}$, and argue as follows. At the
  first step we use the inductive hypothesis and axioms (i) and (iv);
  at the second step, we use~\eqref{eq:6}; and at the third step, we
  use that any partition of $[n]$ is obtained in a unique way from a
  partition of $[n-1]$ either by adding a new singleton part $\{n\}$,
  or by adjoining $n$ to an existing part.

  Finally, for (iii), consider a map
  $y \colon X \rightarrow A \times A^n$. By (ii) we have that
  \begin{align*}
  (\mathrm{D}^n f)(y^\circ) =  \sum_{[n] = A_1 \mid \dots \mid A_k}
    f^{(k)}(y^\circ_\emptyset,y^\circ_{A_1}, \dots, y^\circ_{A_k})\rlap{ ;}
  \end{align*}
  but since $y^{\circ}_{A_i}$ is zero whenever
  $\abs{A_i} \geqslant 2$, and since $f^{(k)}$ is $\mathrm{D}$-linear
   in its last
  $k$ variables, the only term in the sum to the right which is not
  zero is that involving the partition into singletons
  $[n] = \lbrace 1 \rbrace \mid
  \dots \mid \lbrace n \rbrace$. This yields the desired equality:
    \begin{align*}
    (\mathrm{D}^n f)(y^\circ) =
    f^{(n)}(y^\circ_\emptyset,y^\circ_{\lbrace 1 \rbrace}, \dots,
    y^\circ_{\lbrace n \rbrace}) = f^{(n)} (\pi_0y, \pi_1 y, \dots,
    \pi_n y) = f^{(n)} (y)\text{ .} \quad\ \ \  \qquad  \qedhere
      \end{align*}  
\end{proof}
  
It follows that the free cartesian differential category on an arrow
is generated \emph{qua} cartesian left-$k$-linear category by the maps
$f^{(n)} \colon A \times A^n \rightarrow B$ for $n \in \mathbb{N}$. In
fact\footnote{Again, we will not prove this.}, it is \emph{freely}
generated by them, subject to the requirement (mandated by part (i) of
the preceding lemma) that each $f^{(n)}$ is symmetric $k$-linear in
its last $n$ variables. Thus, \emph{maps
  $f \colon A \rightsquigarrow B$ of $\faa(\A)$ are the same as
  $\mathbb{N}$-indexed families
  $f^{(n)} \colon A \times A^n \rightarrow B$ in $\A$, where $f^{(n)}$
  is symmetric $k$-linear in its last $n$~variables}.

\subsection{Composition}
\label{sec:faa-di-bruno-2}

The next step is to understand composition in $\faa(\A)$. As
before, we can determine this by probing $\faa(\A)$, this time by maps from the
free cartesian differential category on a composable pair of arrows.
What this amounts in practice is determining a formula which expresses
the higher-order derivatives of a composite $g \circ f$ in a cartesian
differential category in terms of the derivatives of $g$
and $f$. This formula is the \emph{Fa\`a di Bruno} higher-order chain
rule---whence the nomenclature~$\faa(\A)$.

\begin{Defn}
  \label{def:13}
  Let $f \colon A \rightarrow B$ in a cartesian differential category,
  and suppose that $I = \{n_1 < \dots < n_i\} \subseteq [n]$. We write
  $f^{(I)} \colon A \times A^n \rightarrow B$ for the map
  determined~by
  \begin{equation}\label{eq:5}
    f^{(I)}(x_0, \dots, x_n) = f^{(i)}(x_0, x_{n_1}, \dots, x_{n_i})\rlap{ .}
  \end{equation}
  In particular, we have $f^{(\emptyset)}(x_0, x_1, \dots, x_n) = fx_0$.
\end{Defn}

\begin{Lemma}
  \label{lem:6}(\cite[Corollary~3.2.3]{Cockett2011The-Faa-di-Bruno})
  Let $f \colon A \rightarrow B$ and $g \colon B \rightarrow C$ in a
  cartesian differential category. For each $n \geqslant 0$ we have:
  \begin{equation}\label{eq:8}
    (g \circ f)^{(n)} = \sum_{[n] = A_1 \mid \cdots \mid A_k}
    g^{(k)} \circ (f^{(\emptyset)}, f^{(A_1)}, \dots, f^{(A_k)})\rlap{ .}
  \end{equation}
\end{Lemma}

\begin{proof}
  For each $n \geqslant 0$, define a map
  $f^{[n]} \colon A \times A^n \rightarrow B^{2^n}$ by the rule
  $(f^{[n]})_I = f^{(I)}$. We claim that we have
  \begin{equation}\label{eq:7}
    (g \circ f)^{(n)} = \mathrm{D}^n g \circ f^{[n]} \colon A \times
    A^n \rightarrow C\rlap{ .}
  \end{equation}
  Given this, the desired result will follow immediately
  from~\eqref{eq:4}. We prove~\eqref{eq:7} by induction. The base case
  $n =0$ is trivial; and assuming the result for $n-1$, we verify it
  for $n$ by the following calculation, where we write $\vec v$ for
  $v_1, \dots, v_{n-1}$:
  \begin{align*}
    (g \circ f)^{(n)}(x, \vec v, v_n) &= \mathrm{D_1}\bigl((g\circ
    f)^{(n-1)}\bigr)(x, \vec v, v_n) =
    \mathrm{D_1}\bigl(\mathrm{D}^{n-1}g \circ f^{[n-1]})(x, \vec v,
    v_n)\\
    &= \mathrm{D}^n g\bigl(f^{[n-1]}(x,\vec v), \mathrm{D}_1
    f^{[n-1]}(x,\vec v, v_n)\bigr) = \mathrm{D}^n g(f^{[n]}(x, \vec v,
    v_n))\text{ .}
  \end{align*}
  Here, we use the definition of $(g \circ f)^{(n)}$; induction;
  Lemma~\ref{lem:5}(ii); and the obvious identity
  $(f^{[n-1]}(x,\vec v), \mathrm{D}_1 f^{[n-1]}(x,\vec v, v_n)) =
  f^{[n]}(x, \vec v, v_n)$.
\end{proof}

In a similar manner, we can characterise the identity maps in
$\faa(\A)$ by way of the following lemma, whose proof we leave as an easy
exercise to the reader.

\begin{Lemma}
  \label{lem:4}
  Let $A$ be an object in a cartesian differential category. We have
  that
  \begin{equation}
    \label{eq:10}
    (\mathrm{id}_A)^{(0)} = \mathrm{id}_A\text{, } \qquad 
    (\mathrm{id}_A)^{(1)} = \pi_1 \quad \text{and} \quad
    (\mathrm{id}_A)^{(n)} = 0 \text{ for $n \geqslant 2$.}
  \end{equation}
\end{Lemma}

So far, then, we have shown that $\faa(\A)$ must be the following category.

\begin{Defn}
  \label{def:10}
  Let $\A$ be a cartesian left-$k$-linear category. The \emph{Fa\`a di
    Bruno} category $\faa(\A)$ has:
  \begin{itemize}
  \item \emph{Objects} those of $\A$;
  \item \emph{Morphisms} $f^{(\bullet)} \colon A \rightsquigarrow B$
    are families
    $(f^{(n)} \colon A \times A^n \rightarrow B)_{n \in \mathbb{N}}$
    of maps in $\A$ where each $f^{(n)}$ is symmetric and $k$-linear
    in its last $n$ variables;
  \item \emph{Identity} maps
    $\mathrm{id}^{(\bullet)} \colon A \rightsquigarrow A$ are given by
    the formula~\eqref{eq:10};
  \item \emph{Composition} of
    $f^{(\bullet)} \colon A \rightsquigarrow B$ and
    $g^{(\bullet)} \colon B \rightsquigarrow C$ is given by the
    formula~\eqref{eq:8}.
  \end{itemize}
\end{Defn}

Now by further probing of $\faa(\A)$, we discover that cartesian
left-$k$-linear structure must be given as follows:

\begin{Lemma}
  \label{lem:10}
  The category $\faa(\A)$ is cartesian left-$k$-linear when the
  hom-sets are endowed with the $k$-linear structure inherited from
  $\prod_{n \in \mathbb{N}} \A(A \times A^n, B)$.
\end{Lemma}

\begin{proof}
  Left-$k$-linearity is clear from~\eqref{eq:8}. For the cartesian
  structure, we take the terminal object to be that of $\A$, and the
  binary product of $A,B$ to be given by their product $A \times B$ in
  $\A$ endowed with the projections
  $\pi_0^{(\bullet)}, \pi_1^{(\bullet)}$ specified by
  \begin{equation*}
    (\pi_i)^{(0)} = \pi_i\text{, } \qquad 
    (\pi_i)^{(1)} = \pi_i\pi_1 \quad \text{and} \quad
    (\pi_i)^{(n)} = 0 \text{ for $n \geqslant 2$.} \qedhere
  \end{equation*}
\end{proof}

Note that
$f^{(\bullet)} \colon A_1 \times \dots \times A_k \rightsquigarrow B$
in $\faa(\A)$ is $k$-linear in its $i$th variable just when each
$f^{(n)} \colon (A_1 \times \dots \times A_k)^{n+1} \rightarrow B$ is
jointly $k$-linear in the $n+1$ copies of~$A_i$.

\subsection{Differential structure}
\label{sec:cart-diff-struct}

We now describe the differential operator making $\faa(\A)$ into a
cartesian differential category. Once again, the definition is forced,
and once again we can obtain it by reading off from what happens in a
cartesian differential category.

\begin{Lemma}
  \label{lem:7}
  Let $f \colon A \rightarrow B$ in a cartesian differential category
  and $n \geqslant 0$. We have
  \begin{equation}\label{eq:11}
    (\mathrm{D}f)^{(n)}(x_0, y_0, \dots, x_n, y_n) =
    f^{(n+1)}(\vec x, y_0) + \textstyle\sum_{i=1}^n f^{(n)}(\vec
    x[y_i/x_i])\rlap{ .}
  \end{equation}
\end{Lemma}

\begin{proof}
  By~\eqref{eq:6}, it suffices to prove that
  $(\mathrm{D}f)^{(n)}(x_0, y_0, \dots, x_n, y_n) =
  \mathrm{D}(f^{(n)})(\vec x, \vec y)$. For this, we calculate that
  \begin{align*}
    \mathrm{D}(f^{(n)})(\vec x, \vec y) &= \mathrm{D}(\mathrm{D}^n f \circ
    {\mathrm{id}_A}^{[n]})(\vec x, \vec y) 
    = \mathrm{D}^{n+1} f 
    \bigl(\mathrm{id}_A^{[n]} \vec x,
    \mathrm{D}(\mathrm{id}_A^{[n]})(\vec x, \vec y)\bigr)\\
    & = \mathrm{D}^{n+1} f 
    \bigl(\mathrm{id}_A^{[n]} \vec x, \mathrm{id}_A^{[n]} \vec y\bigr)
    = \mathrm{D}^{n+1} f 
    \circ \mathrm{id}_{A \times A}^{[n]} (\vec x, \vec y) =
    (\mathrm{D}f)^{(n)}(\vec x, \vec y)
  \end{align*}
  using, in turn: \eqref{eq:7}; the chain rule; the
  $\mathrm{D}$-linearity of $\mathrm{id}^{[n]}$; the easy calculation
  from~\eqref{eq:10} that
  $(\mathrm{id}_A^{[n]} \vec x, \mathrm{id}_A^{[n]} \vec y) =
  \mathrm{id}_{A \times A}^{[n]} (\vec x, \vec y)$; and~\eqref{eq:7}
  again.
\end{proof}

This indicates how we must define the differential operator on
$\faa(\A)$; it remains to check that doing so verifies the appropriate
axioms.

\begin{Prop}
  \label{prop:1}
  Let $\A$ be a cartesian left-$k$-linear category. $\faa(\A)$ is a
  cartesian differential category where for
  $f^{(\bullet)} \colon A \rightsquigarrow B$, we define
  $(\mathrm{D}f)^{(\bullet)} \colon A \times A \rightsquigarrow B$
  by~\eqref{eq:11}.
\end{Prop}

\begin{proof}
  We check the seven axioms. For (i), $k$-linearity of $\mathrm{D}$ is
  immediate from~\eqref{eq:11}, and it easy to see that~\eqref{eq:11}
  is also jointly linear in the variables $y_0, \dots, y_n$, as
  required for (ii). Next, (iii) follows from the componentwise nature
  of products in $\faa(\A)$, while (iv) is simply a matter of
  instantiating~\eqref{eq:11} with \eqref{eq:10} and comparing
  with~Lemma~\ref{lem:10}. Leaving (v) aside for the moment, we can
  dispatch (vi) and (vii) by computing
  $(\mathrm{D}\mathrm{D}f)^{(n)}(x_0, y_0, z_0, w_0, \dots, x_n, y_n,
  z_n, w_n)$ to be given by
  \begin{align*}
    &
    (\mathrm{D}f)^{(n+1)}(x_0,y_0,\dots, x_n,y_n, z_0,w_0) +
    \textstyle\sum_{i=1}^n f^{(n)}(x_0, y_0, \dots, x_n, y_n[z_i,w_i/x_i,y_i]) \\
    &=
    \textstyle f^{(n+2)}(\vec x, z_0, y_0) + \sum_i f^{(n+1)}(\vec x[y_i/x_i], z_0) +
    f^{(n+1)}(\vec x, w_0) \\
    &\textstyle + \sum_i f^{(n+1)}(\vec x[z_i/x_i], y_0) + \sum_{i \neq j} f^{(n)}(\vec
    x[z_i/x_i,y_j/x_j]) + \sum_i f^{(n)}(\vec x[w_i/x_i])\rlap{ ;}
  \end{align*}
  clearly, this is unaltered by interchanging the $y$'s and
  $z$'s---yielding (vii)---and reduces to
  $(\mathrm{D}f)^{(n)}(x_0, w_0, \dots, x_n, w_n)$ on zeroing each
  $y$---which gives (vi).

  It remains to prove the chain rule (v): thus, for all
  $f^{(\bullet)} \colon A \rightsquigarrow B$ and
  $g^{(\bullet)} \colon B \rightsquigarrow C$ in $\faa(\A)$ and
  $n \in \mathbb{N}$, we must prove
  $\bigl(\mathrm{D}(g \circ f)\bigr)^{(n)} = \bigl(\mathrm{D}g \circ
  (f\pi_0, \mathrm{D}f)\bigr)^{(n)}$ in $\A$. We have that
  $\bigl(\mathrm{D}g \circ (f\pi_0, \mathrm{D}f)\bigr)^{(n)}$ is given
  by
  \begin{align*}
    &\ \ \ \ \!\!\!\!\!\! \sum_{[n] = A_1 \mid \dots \mid A_k}\!\!\!\!\!\! (\mathrm{D}g)^{(k)}
    \bigl((f\pi_0)^{(\emptyset)}, (\mathrm{D}f)^{(\emptyset)},
    (f\pi_0)^{(A_1)}, (\mathrm{D}f)^{(A_1)}, \cpdots, (f\pi_0)^{(A_k)},
    (\mathrm{D}f)^{(A_k)}\bigr) \\&=\!\!\!\!\!\!
    \sum_{[n] = A_1 \mid \dots \mid A_k}\!\!\!\!\!\! \Bigl(g^{(k+1)}
    \bigl((f\pi_0)^{(\vec A)}, (\mathrm{D}f)^{(\emptyset)}\bigr)
    + \textstyle\sum_{i=1}^n g^{(k)}
    \bigl((f\pi_0)^{(\vec A)}[(\mathrm{D}f)^{(A_i)}/(f\pi_0)^{(A_i)}]\bigr)\Bigr)\text,
  \end{align*}
  where we write $(f\pi_0)^{(\vec A)}$ for
  $f\pi_0^{(\emptyset)}, f\pi_0^{(A_1)}, \dots, f\pi_0^{(A_k)}$. We now
  rewrite terms of the form $(f\pi_0)^{(I)}$ or $(\mathrm{D}f)^{(I)}$
  via the switch isomorphism
  $\sigma \colon A^2 \times (A^2)^n \cong (A \times A^n)^2$. To do so,
  let us write $\dbr{n} = \{1, \dots, n, 0', 1', \dots, n'\}$; now for
  any $I \subseteq [n]$, we write $I^{0'}$ for
  $I \cup \{0'\} \subseteq \dbr{n}$ and, for any $i \in I$ write
  $I^{i'}$ for $I \cup \{i'\} \setminus \{i\} \subseteq \dbr{n}$. Then:
  \begin{equation*}
    (f \pi_0)^{(I)} = f^{(I)}\sigma \qquad \text{and}\qquad (Df)^{(I)}
    = (f^{(I^{0'})} + \textstyle\sum_{i \in
      I} f^{(I^{i'})})\sigma \rlap{ .}
  \end{equation*}

  It follows that
  $\bigl(\mathrm{D}g \circ (f\pi_0, \mathrm{D}f)\bigr)^{(n)}$ is the
  sum
  \begin{align*}
    & \ \ \ \sum_{[n] = A_1 \mid \dots \mid A_k} \Bigl(g^{(k+1)}
    (f^{(\vec A,\{0'\})})
    + \sum_{i=1}^n g^{(k)}
    (f^{(\vec A[A_{i}^{0'} / A_i])}) + \sum_{i=1}^n \sum_{j \in A_i} g^{(k)}
    (f^{(\vec A[A_{i}^{j'} / A_i])})\Bigr)\sigma\\
    & = \sum_{[n]^{0'} = A_1 \mid \dots \mid A_k} g^{(k)}(f^{(\vec
      A)})\sigma +
    \sum_{\substack{1 \leqslant j \leqslant n\\ [n]^{j'} = A_1 \mid \dots \mid A_k}}
    g^{(k)}(f^{(\vec A)})\sigma\rlap{ .}
  \end{align*}
  We thus conclude that
  $\bigl(\mathrm{D}g \circ (f\pi_0, \mathrm{D}f)\bigr)^{(n)}(x_0, y_0,
  \dots, x_n, y_n)$ is given by
  \begin{align*}
    &  \ \ \ \ \sum_{[n+1] = A_1 \mid \dots \mid A_k} g^{(k)}(f^{(\vec
      A)})(\vec x, y_0) +
    \sum_{\substack{1 \leqslant j \leqslant n\\ [n] = A_1 \mid \dots \mid A_k}}
    g^{(k)}(f^{(\vec A)})(\vec x[y_j / x_j])\\
    &= (g \circ f)^{(n+1)}(\vec x, y_0) + \sum_{1 \leqslant j
      \leqslant n} (g \circ f)^{(n)}(\vec x[y_j / x_j]) = \bigl(\mathrm{D}(g
    \circ f)\bigr)^{(n)}(x_0, y_0, \dots, x_n, y_n)
  \end{align*}
  as desired.
\end{proof}

\subsection{Universal property}
\label{sec:high-order-deriv-1}

It remains to show that $\faa(\A)$ is the \emph{cofree} cartesian
differential category on $\A$. To do this, we will first need to
understand higher-order derivatives in $\faa(\A)$. Given a Fa\`a di
Bruno map $f^{(\bullet)} \colon A \rightsquigarrow B$, we may consider
not only the component $f^{(m)} \colon A \times A^m \rightarrow B$ in
$\A$, but also the $m$th order derivative in $\faa(\A)$, which we will
denote by $f^{(m, \bullet)} \colon A \times A^m \rightsquigarrow B$,
with components
$f^{(m,n)} \colon (A \times A^m) \times (A \times A^m)^n \rightarrow
B$. We now find an explicit formula for these components. 

\begin{Not}
  \label{not:5}
We write
$\theta \colon [m] \simeq [n]$ to denote a partial isomorphism between
$[m]$ and $[n]$, comprising subsets $I \subseteq [m]$ and
$J \subseteq [n]$ and an isomorphism $\theta \colon I \rightarrow J$.
If $k$ is the common cardinality of $I$ and $J$, then we define $\abs
\theta$ to be $n+m-k$, and given a family
$(x_{ij})_{0 \leqslant i \leqslant m, 0 \leqslant j \leqslant n}$, 
write $x_{\theta_{(1)}\theta_{(2)}}$ for the list of length
$\abs{\theta}+1$ given by
\begin{equation*}
  x_{\theta_{(1)}\theta_{(2)}} \defeq x_{00}, x_{i_1 \theta(i_1)}, \dots,
  x_{i_k \theta(i_k)}, x_{i'_1 0}, \dots, x_{i'_{m-k}0}, x_{0j'_1},
  \dots, x_{0j'_{n-k}}
\end{equation*}
where $i_1 < \dots < i_k$ enumerates $I$, $i_1' <
\dots < i'_{m-k}$ enumerates $[m] \setminus I$, and $j_1' < \dots <
j'_{n-k}$ enumerates $[n]\setminus J$. 
\end{Not}
For example, if $\theta \colon [3] \simeq [4]$ is the partial
isomorphism with graph $\{(1,2),(3,4)\}$ then we have
\begin{equation*}
  x_{\theta_{(1)}\theta_{(2)}} = x_{00}, x_{12}, x_{34}, x_{20},
  x_{01}, x_{03}\rlap{ .}
\end{equation*}
\begin{Lemma}
  \label{lem:12}
  For  $f^{(\bullet)} \colon A \rightsquigarrow B$ in $\faa(\A)$
  and $\vec x = x_{00}, \dots, x_{m0}, \dots, x_{nm} \colon X
  \rightarrow A$ in $\A$ we
  have that
  \begin{equation*}
    f^{(m,n)}(\vec x) =
    \sum_{\theta \colon [m] \simeq [n]}
    f^{(\abs{\theta})}(x_{\theta_{(1)}\theta_{(2)}})
  \end{equation*}
\end{Lemma}

\begin{proof}
  We proceed by induction on $m$. The base case is clear; so we now
  assume the result for $m-1$ and prove it for $m$. If we write
  ${\vec x}_{\hat m i}$ for $x_{0i}, \dots, x_{m-1,i}$, then we have
  $f^{(m,n)}(\vec x) = f^{(m,n)}({\vec x}_{\hat m 0}, x_{m0}, \dots,
  {\vec x}_{\hat m n}, x_{mn})$ given by
  \begin{align*}
    &f^{(m-1,n+1)}({\vec x}_{\hat m 0}, \cpdots, {\vec x}_{\hat m n},
    x_{m0}, 0^{m-1}) + \sum_{i=1}^n f^{(m-1,n)}\Bigl(({\vec x}_{\hat m 0},
    \cpdots, {\vec x}_{\hat m n})[(x_{mi},0^{m-1})/{\vec x}_{\hat m 
      i}]\Bigr)\\
    &= \sum_{\theta \colon [m-1] \simeq [n]} f^{(\abs{\theta}+1)}( {\vec
      x}_{\theta_{(1)}\theta_{(2)}}, x_{m0})
    + \!\!\!\!
    \sum_{\substack{1 \leqslant i \leqslant n\\ \theta \colon [m-1]
        \simeq [n] \setminus \{i\}}} f^{(\abs{\theta}+1)}(
    {x}_{\theta_{(1)}\theta_{(2)}}, x_{mi})
    \\
    &= \sum_{\theta \colon [m] \simeq [n]} f^{(\abs{\theta})}(
    x_{\theta_{(1)}\theta_{(2)}})
  \end{align*}
  as desired. Here, at the first step, we use that
  $f^{(m, \bullet)} = \mathrm{D}_1 f^{(m-1, \bullet)}$ together
  with~\eqref{eq:11}. At the second step, we use the inductive
  hypothesis: \emph{a priori}, this would yield for the $f^{(m-1,n+1)}$
  term a sum over isomorphisms $\theta \colon I \cong J$ with
  $I \subseteq [m-1]$ and $J \subseteq [n+1]$, but the $m-1$ trailing
  zeroes in the arguments of $f^{(m-1,n+1)}$ mean $n+1$ \emph{cannot}
  be in $J$; similarly, for the $i$th $f^{(m-1,n)}$ term, we
  \emph{cannot} have $i \in J$. Finally, the third step is easiest read
  backwards: the penultimate line is a case split of the final line on
  the cases where $m \notin I$, and where $m \in I$ with
  $\theta(m) = i$.
\end{proof}


We are now in a position to prove cofreeness of $\faa(\A)$. Let
$\varepsilon_\A \colon \faa(\A) \rightarrow \A$ be the functor which
is the identity on objects, and is given on morphisms by
$\varepsilon(f^{(\bullet)}) = f^{(0)}$. Clearly, $\varepsilon_\A$
preserves the $k$-linear structure on the homs, and preserves
cartesian products \emph{strictly}. It is thus a map in $\CLA$.

\begin{Thm}
  \label{thm:1}
  For any $\A \in \CLA$, the map
  $\varepsilon_\A\colon \faa(\A) \rightarrow \A$ exhibits $\faa(\A)$
  as the cofree cartesian differential category on $\A$. That is, for
  any $\B \in \CDiff$ and map $F \colon \B \rightarrow \A$ in $\CLA$,
  there is a unique $\tilde F \colon \B \rightarrow \faa(\A)$ in
  $\CDiff$ with $F = \varepsilon_\A \circ \tilde F$.
\end{Thm}

\begin{proof}
  Given $F \colon \B \rightarrow \A$ as in the statement, we define
  $\tilde F$ to act as $F$ does on objects, and to be given on
  morphisms by $\tilde F(f) = (Ff, F(f^{(1)}), F(f^{(2)}), \dots)$.
  This assignment is functorial by Lemmas~\ref{lem:6} and~\ref{lem:4},
  and is easily seen to be (strict) cartesian left-$k$-linear.
  Furthermore, it preserves the differential by Lemma~\ref{lem:7}; so
  $\tilde F \colon \B \rightarrow \faa(\A)$ is a map in $\CDiff$, and
  clearly $\varepsilon_\A \circ \tilde F = F$.

  It remains to prove unicity of $\tilde F$. If
  $G \colon \B \rightarrow \faa(\A)$ in $\CDiff$ satisfies
  $\varepsilon_\A \circ G = F$, then it must agree with $F$, and hence
  with $\tilde F$ on objects; while on maps, given
  $f \colon X \rightarrow Y$ in $\B$, we have for each
  $n \in \mathbb{N}$ that
  \begin{equation*}
    (Gf)^{(n)} = (Gf)^{(n,0)} = \varepsilon_\A\bigl((Gf)^{(n,
      \bullet)}\bigr) = \varepsilon_\A(G(f^{(n)})) = F(f^{(n)}) =
    (\tilde F f)^{(n)}
  \end{equation*}
  using, in succession: Lemma~\ref{lem:12}; definition of
  $\varepsilon_\A$; that $G$ is a map of cartesian differential
  categories; that $\varepsilon_\A \circ G = F$; and definition of
  $\tilde F$.
\end{proof}

The composite of the cofree differential category functor
$\faa \colon \CLA \rightarrow \CDiff$ with its left adjoint
$U \colon \CDiff \rightarrow \CLA$ yields a comonad on $\CLA$, which
we also denote by $\faa$. By the previous theorem, we easily 
deduce the main results of~\cite{Cockett2011The-Faa-di-Bruno}.

\begin{Cor}
  \label{cor:1}
  The forgetful functor $\CDiff \rightarrow \CLA$ is strictly
  comonadic; that is, the comparison functor
  $\CDiff \rightarrow \cat{Coalg}(\faa)$ is an isomorphism over
  $\CLA$.
\end{Cor}

\begin{proof}
  The forgetful functor $U \colon \CLA \rightarrow \CDiff$ forgets
  essentially-algebraic structure, and so is strictly monadic. In
  particular, it creates all limits, is conservative, and has the
  isomorphism lifting property. Since it has a right adjoint, it also
  creates all colimits, and so by the Beck theorem, is strictly
  comonadic.
\end{proof}

Explicitly, for a cartesian differential category $\B$, its
$\faa$-coalgebra structure is obtained by applying Theorem \ref{thm:1}
to the identity functor $1_\B\colon \B \to \B$. The resulting
$\tilde 1_\B \colon \B \to \faa(\B)$ is the identity on objects, and
is given on maps by $\tilde 1_\B (f) = (f, f^{(1)}, f^{(2)}, \dots)$.
In particular, we re-find the comonad comultiplication
$\delta_\A \colon \faa(\A) \to \faa(\faa(\A))$---constructed in detail
in~\cite[\sec 2.2]{Cockett2011The-Faa-di-Bruno}---as
$\delta_\A = \tilde 1_{\faa(\A)}$. Given this description, we can
exploit Corollary~\ref{cor:1} to obtain an alternative
characterisation of cartesian differential categories.

\begin{Cor}
  \label{cor:2}
  Let $\A$ be a cartesian left-$k$-linear category. To endow $\A$ with
  cartesian differential structure is equally to give, for each
  $n \geqslant 0$, an $n$th-order differential operator
  $(\thg)^{(n)} \colon \A(A,B) \rightarrow \A(A \times A^n, B)$ such
  that:
  \begin{enumerate}[(i),itemsep=0.1\baselineskip]
  \item Each $(\thg)^{(n)}$ is $k$-linear;
  \item Each $f^{(n)}$ is $k$-linear and symmetric in its last $n$
    arguments;
  \item For all binary products $A_0 \times A_1$ we have
    $\pi_i^{(1)} = \pi_i\pi_1$ and
    $\pi_i^{(n)} = 0$ for $n \geqslant 2$.
  \item For all $A \in \A$ we have
    $\mathrm{id}_A^{(1)} = \pi_1$, and $\mathrm{id}_A^{(n)} = 0$ for
    $n \geqslant 2$;
  \item
    $(g f)^{(n)} = \sum_{[n] = A_1 \mid \cdots \mid A_k} g^{(k)}
    (f^{(\emptyset)}, f^{(A_1)}, \cpdots, f^{(A_k)})$ for all
    $f \colon A \rightarrow B$, ${g \colon B \rightarrow C}$;
  \item $f^{(0)} = f$ for all $f \in \A(A,B)$;
  \item
    $(f^{(n)})^{(m)}(\vec x) = \sum_{\theta \colon\! [m] \simeq [n]}
    f^{(\abs{\theta})}(x_{\theta_{(1)}\theta_{(2)}})$ for all
    $f \colon A \rightarrow B$ and $n, m \geqslant 0$.
  \end{enumerate}
\end{Cor}

\begin{proof}
  These are exactly the data of a $\faa$-coalgebra structure
  $D \colon \A \rightarrow \faa(\A)$. (ii), (iv) and (v) express that
  $D$ is a well-defined functor, (i) that it is a map of
  left-$k$-linear categories, and (iii) that it preserves the
  cartesian structure. The counit axiom
  $\varepsilon_\A \circ D = 1$ and the coassociativity axiom
  $\faa(D) \circ D = \delta_\A \circ D$ are conditions (vi) and (vii)
  respectively.
\end{proof}

We conclude this section with a brief remark comparing the above
construction $\faa(\A)$ of the cofree cartesian differential category
with the one given in~\cite{Lemay2018A-tangent}, which we denote by
$\mathsf{D}(\A)$. Since both categories have the same universal
property, they must be isomorphic as cartesian differential
categories; but in fact, the work we have done allows us to construct
the isomorphism explicitly.

As discussed in Section~\ref{sec:faa-di-bruno-1} above,
$\mathsf{D}(\A)$ has the same objects as $\A$, while maps from $A$ to
$B$ are certain $\mathbb{N}$-indexed sequences of maps
$(f_n \colon A^{2^n} \rightarrow B)$ generalising the sequence
$(f, \mathrm{D}f, \mathrm{D}^2f, \dots)$ of iterated differentials of
a map in a cartesian differential category. Since, by contrast, maps
in $\faa(\A)$ generalise sequences of the form
$(f, f^{(1)}, f^{(2)}, \dots)$ in a cartesian differential category,
it is natural to construct the isomorphism
$\faa(\A) \cong \mathsf{D}(\A)$ using Lemma~\ref{lem:3}. In one
direction, we have $\faa(\A) \to \mathsf{D}(\A)$ which is the identity
on objects, and defined on morphisms via the formula of Lemma
\ref{lem:3}\ref{item:7}; while in the other, we have
$\mathsf{D}(\A) \to \faa(\A)$ which is again the identity on objects, and
defined on morphisms now using Lemma~\ref{lem:3}\ref{item:8}.

\section{Differential modalities and Fa\`a di Bruno}
\label{sec:cokleisli}

In this section we do two things. The first is to recall the link
between cartesian differential categories and the \emph{tensor
  differential categories} of~\cite{Blute2006Differential}. As
explained in the introduction, the latter are symmetric monoidal
$k$-linear categories $\V$ with a certain kind of comonad ${!}$ termed
a \emph{differential modality}; the link with cartesian differential
categories is that the co-Kleisli category $\cat{Kl}({!})$ of the
differential modality on a tensor differential category with finite
products is a cartesian differential category.

Many natural examples of cartesian differential categories are either
of the form $\cat{Kl}(!)$, or at least admit a full,
structure-preserving embedding into one of this form. An important
open question is whether \emph{every} cartesian differential category
arises in this way, and our second main result, given in
Section~\ref{sec:appl-an-embedd} below, will answer this in the
positive. The second objective of this section is to take a step
in that direction by proving the claim for cartesian differential
categories of the form $\faa(\A)$.

\subsection{Coalgebra modalities}
\label{sec:diff-categ}
Before recalling the notion of a differential modality, we first
recall some more basic kinds of structure which a comonad on a
symmetric monoidal category may bear.

\begin{Defn}
  \label{def:15}
  Let $(\V, \otimes, I)$ be a symmetric monoidal category and let
  ${!}$ be a comonad on $\V$, with counit $\varepsilon$ and
  comultiplication $\delta$.
  \begin{itemize}[itemsep=0.35\baselineskip]
  \item We call ${!}$ a \emph{coalgebra modality} if it comes endowed
    with maps
    \begin{equation}\label{eq:14}
      e_A \colon {!}A \rightarrow I \qquad \text{and}
      \qquad \Delta_A \colon {!}A \rightarrow {!}A \otimes {!}A\rlap{ ,}
    \end{equation}
    natural in $A$, which are such that each $({!}A, e_A, \Delta_A)$ is a
    cocommutative comonoid, and each $\delta_A$ is a map of comonoids
    $({!}A, e_A, \Delta_A) \rightarrow ({!}{!}A, e_{{!}A},
    \Delta_{{!}A})$.
  \item We call ${!}$ a \emph{monoidal comonad} if it comes endowed with
    maps
    \begin{equation}\label{eq:15}
      m_I \colon I \rightarrow {!}I
      \qquad \text{and} \qquad m_\otimes \colon {!}A \otimes {!}B \rightarrow {!}(A \otimes B)
    \end{equation}
    making ${!}$ into a symmetric monoidal functor, and $\varepsilon$ and
    $\delta$ into monoidal natural transformations; see, for
    example,~\cite[\sec 7]{Moerdijk2002Monads} for the conditions
    involved.
  \end{itemize}
\end{Defn}

If $!$ is a coalgebra modality, then every $!$-coalgebra
$(A, a \colon A \rightarrow {!}A)$ can be made into a cocommutative
comonoid via the maps:
\begin{equation}\label{eq:34}
 A
 \xrightarrow{a} {!}A \xrightarrow{e} I \qquad \text{and} \qquad A \xrightarrow{a} {!}A \xrightarrow{\Delta} {!}A \otimes {!}A
 \xrightarrow{\varepsilon \otimes \varepsilon} A \otimes A\rlap{ ;}
\end{equation}
these constitute the unique comonoid structure on $A$ for which $a \colon A \rightarrow {!}A$ is a comonoid morphism as well
as a ${!}$-coalgebra morphism.
In this way, we obtain a factorisation of the forgetful functor
$\cat{Coalg}({!}) \rightarrow \V$ through the category
$\cat{Cocomon}(\V)$ of cocommutative comonoids in $\V$, and in fact,
making $!$ into a coalgebra modality is \emph{equivalent} to giving
such a factorisation; see~\cite[Theorem~3.12]{Blute2016Derivations}.

On the other hand, if ${!}$ is a monoidal comonad, then we can lift
the symmetric monoidal structure of $\V$ to $\cat{Coalg}({!})$; the
unit is $\hat I = (m_I \colon I \rightarrow {!}I)$ and the 
binary tensor is:
\begin{equation*}
  (A \xrightarrow{a} {!}A) \mathbin{\hat\otimes} (B
  \xrightarrow{b} {!}B) = \bigl(A \otimes B \xrightarrow{m_\otimes (a \otimes b)} {!}(A \otimes B)\bigr)\rlap{ .}
\end{equation*}

If ${!}$ is both a monoidal comonad and a coalgebra modality, then
there are natural compatibilities we can impose between the two
structures. The resulting structure is exactly what is needed to model
the \emph{exponential} modality of linear logic; this explains the
origin of the notation ${!}$ for our comonads.

\begin{Defn}
  \label{def:18}
  Let $(\V, \otimes, I)$ be a symmetric monoidal category and
  $({!}, \varepsilon, \delta)$ a comonad on $\V$. We call ${!}$ a
  \emph{monoidal coalgebra modality}  if it has the structure of a
  monoidal comonad and of a coalgebra modality, in such a way that
  each map of~\eqref{eq:14} is a map of ${!}$-coalgebras and each
  map of~\eqref{eq:15} is a map of $\otimes$-comonoids.
\end{Defn}

Under mild side conditions, the two structures of a monoidal coalgebra
modality determine each other. On the one hand, if ${!}$ is a monoidal
comonad, then it is a monoidal coalgebra modality (in a unique way)
just when the lifted monoidal structure on $\cat{Coalg}({!})$ is
cartesian; see~\cite[Definition~1.17]{Maneggia2004Models}. On the
other hand, if ${!}$ is a coalgebra modality and $\V$ has finite
products, then ${!}$ is a monoidal coalgebra modality (in a unique
way) just when the following \emph{storage maps} are invertible
\begin{equation}\label{eq:3}
  \chi_1 \defeq {!}(1) \xrightarrow{e} I \text{ ,} \  \chi_{AB}
  \defeq {!}(A \times B) \xrightarrow{\Delta} {!}(A
  \times B) \otimes {!}(A \times B) \xrightarrow{{!}\pi_0 \otimes
    {!}\pi_1} {!} A \otimes {!}B\text{ .}
\end{equation}
Indeed, in this situation, the monoidal constraint maps $m_I$ and $m_\otimes$
are found as:
\begin{equation}
  \label{eq:42}
  \begin{gathered}
    \smash{I \xrightarrow{\chi_1^{-1}} {!}(1) \xrightarrow{\delta_1}
      {!}{!}(1) \xrightarrow{{!}\chi_1} {!}I \qquad \text{and}} \\
    {!}A \otimes {!}B \xrightarrow{\chi_{AB}^{-1}}
    {!}(A \times B) \xrightarrow{\delta} {!}{!}(A \times B)
    \xrightarrow{{!}\chi_{AB}} {!}({!}A \otimes {!}B)
    \xrightarrow{\!{!}(\varepsilon \otimes \varepsilon)\!} {!}(A \otimes B)\text{ ;}
  \end{gathered}
\end{equation}
see~\cite[Theorem~3.1.6]{Blute2015Cartesian} and the
references therein.

\subsection{Differential modalities}
\label{sec:diff-modal}
We are now ready for the definition of differential modality. We write
``symmetric monoidal $k$-linear category'' for a category $\V$ which
is symmetric monoidal and $k$-linear, and for which the action on homs
of the tensor product
$\V(A,B) \times \V(C,D) \rightarrow \V(A\otimes C, B\otimes D)$ is
bilinear.
\begin{Defn}
  \label{def:19} Let $(\V, \otimes, I)$ be a symmetric monoidal
  $k$-linear category and $({!}, \varepsilon, \delta)$ a comonad on
  the underlying ordinary category of $\V$.
  \begin{itemize}[itemsep=0.35\baselineskip]
  \item We call ${!}$ a \emph{differential modality} if it is a
    coalgebra modality, and comes endowed with a \emph{deriving
      transformation}: a natural family of maps
    \begin{equation*}
      \mathsf{d}_A \colon {!}A \otimes A \rightarrow {!}A
    \end{equation*}
    rendering commutative the following diagrams, known respectively as
    the \emph{product rule}, the \emph{linear rule}, the \emph{chain
      rule} and the \emph{interchange rule}.
    \begin{gather*}
      \cd[@!C@C-2em]{
        {!}A \otimes A \ar[d]_-{\mathsf{d}} \ar[r]^-{\Delta \otimes 1}
        & {!}A \otimes {!}A \otimes A \ar[d]^-{(1 \otimes \mathsf{d}) + (\mathsf{d} \otimes 1)(1
          \otimes \sigma)}  \\ 
        {!}A
        \ar[r]^-{\Delta} &
        {!}A \otimes {!}A} \quad
      \cd[@!C@C-2em]{
        {{!}A \otimes A} \ar[rr]^-{\mathsf{d}} \ar[dr]_-{e \otimes 1} & &
        {{!}A} \ar[dl]^-{\varepsilon}\\
        & {A}
      }\\
      \cd[@C-0.35em]{
        {!}A \otimes A \ar[d]_-{\Delta \otimes 1}
        \ar[rr]^-{\mathsf{d}} & & {!}A \ar[d]^-{\delta} \\
        {!}A \otimes {!}A \otimes A \ar[r]^-{\delta \otimes
          \mathsf{d}} &
        {!}{!}A \otimes {!}A \ar[r]^-{\mathsf{d}} & {!}{!}A}
      \qquad
      \cd[@C-0.35em]{
        {!}A \otimes A \otimes A \ar[r]^-{1
          \otimes \sigma} \ar[d]_-{\mathsf{d} \otimes 1} &
        {!}A \otimes A \otimes A
        \ar[r]^-{\mathsf{d} \otimes 1} & {!}A
        \otimes A \ar[d]^-{\mathsf{d}} \\
        {!}A \otimes A \ar[rr]^-{\mathsf{d}} && {!}A\rlap{ .}
      }
    \end{gather*}
    We call $\V$ endowed with its differential modality a
    \emph{tensor differential category}.
  \item We call ${!}$ a \emph{monoidal differential modality} if it is a
    monoidal coalgebra modality endowed with a deriving
    transformation\footnote{Note that in~\cite{Fiore2007Differential},
      Fiore requires a monoidal differential modality to satisfy an
      additional axiom relating $m_\otimes$ with $\mathsf{d}$; however,
      as shown in~\cite{Blute2019Differential}, this additional axiom is
      \emph{always} satisfied.}.
  \end{itemize}
\end{Defn}

The above notion of deriving transformation refines that
of~\cite{Blute2006Differential} in two standard ways. Firstly,
it drops the constant rule (``[d.1]'' in \emph{loc.~cit.}) since this
is derivable as in~\cite[Lemma~4.2]{Blute2019Differential}. Secondly,
it adds the interchange rule, which is necessary to ensure that the
following result holds without further side-conditions:

\begin{Prop}
  \label{prop:2} Let
  $\V$ be a symmetric monoidal category $k$-linear category with
  finite (bi)products. For any differential modality ${!}$ on $\V$,
  the co-Kleisli category $\cat{Kl}({!})$ has a structure of cartesian
  differential category. If $\V$ is monoidal closed and $!$ is monoidal, then
  $\cat{Kl}(!)$ is a cartesian closed differential category.
\end{Prop}
\begin{proof}
  The first assertion
  is~\cite[Lemmas~3.2.2~\&~3.2.3]{Blute2009Cartesian}. The claim in
  the final sentence is~\cite[Theorem~4.4.2]{Blute2015Cartesian}.
\end{proof}
While there is no need to recount the proof of this result, we will
need to know how the cartesian differential structure of $\cat{Kl}(!)$
is obtained. The cartesian left-$k$-linear structure is easy: the hom
$\cat{Kl}({!})(A,B)$ inherits $k$-module structure from $\V({!}A, B)$,
and finite products in $\cat{Kl}(!)$ are induced from those of $\V$
along the identity-on-objects right adjoint functor
$\V \rightarrow \cat{Kl}(!)$. As for the differential structure, if
$f \colon {!}A \rightarrow B$ is a map in $\cat{Kl}(!)(A,B)$, then
$\mathrm{D}f \in \cat{Kl}(!)(A \times A, B)$ is the composite
\begin{equation}\label{eq:17}
  {!}(A \times A) \xrightarrow{\ \ \chi\ \ } {!}A \otimes {!}A
  \xrightarrow{\ \ 1
    \otimes \varepsilon \ \ } {!}A \otimes A \xrightarrow{\ \
    \mathsf{d}\ \ } {!}A
  \xrightarrow{\ \ f\ \ } B
\end{equation}
whose first part is the storage map of~\eqref{eq:3}. When ${!}$ is
monoidal and $\V$ is monoidal closed, the exponentials
making $\cat{Kl}({!})$ cartesian closed are given by $B^A \defeq
[{!}A, B]$.

\begin{Defn}
  \label{def:25}
  If ${!}$ is the differential modality of a tensor differential
  category, and $\A$ is a cartesian differential category, then we say
  that $\A$ is \emph{induced} by ${!}$ if $\A \cong \cat{Kl}({!})$ as
  cartesian differential categories.
\end{Defn}

An extremely important source of differential modalities, and hence of
cartesian differential categories, is the following result:

\begin{Prop}
  \label{prop:3}
  Let $\V$ be a symmetric monoidal $k$-linear category with finite
  biproducts, and suppose the forgetful functor
  $\cat{Cocomon}(\V) \rightarrow \V$ has a right adjoint. The
  induced \emph{cofree cocommutative coalgebra} comonad $R$ on $\V$
  can be made into a differential modality which is terminal among
  differential modalities on $\V$.
\end{Prop}

\begin{proof}
  It is well-known that $R$ is a monoidal coalgebra modality; indeed,
  this is the basis for Lafont's semantics for the exponential
  modality of linear logic. The construction of a deriving
  transformation for $\V$ is given (in dual form) in~\cite[\sec
  4]{Blute2016Derivations}, while \sec 6 of \emph{loc.\,cit.}~proves
  its terminality among differential
  modalities\footnote{In~\cite{Blute2016Derivations} there is an
    additional assumption which---in our dual case---amounts to the
    existence of coreflexive equalisers preserved by tensor in each
    variable. However, it is easily checked that this is not necessary
    for the proof of terminality.}.
\end{proof}

\begin{Exs}
  \begin{enumerate}[(i), itemsep=0.35\baselineskip]
  \item \label{item:6}Taking $\V = \kmod^\mathrm{op}$ in the preceding result, we see
    that the free symmetric algebra monad $\mathrm{Sym}$ of~\eqref{eq:13} endows
    $\kmod^\mathrm{op}$ with a differential modality. The induced
    cartesian differential category is exactly the cartesian
    differential category $\cat{Gen\P oly}_k$ of
    Examples~\ref{ex:5}\ref{item:3}, while $\cat{Poly}_k$ is its full
    subcategory on the finitely generated free $k$-modules.
  \item When $\V = \cat{Rel}$, the category of sets and relations, the
    cofree cocommutative comonoid on $X$ is given by the set of finite
    multisets of elements of $X$. So the finite multiset comonad on
    $\cat{Rel}$ is a differential modality; the induced cartesian
    differential category is
      described explicitly in~\cite[\sec 5.1]{Bucciarelli2010Categorical}.
    \item When $\V = \cat{Fin}$, the category of finiteness spaces and
      finitary relations~\cite{Ehrhard2005Finiteness}, the cofree
      cocommutative comonoid is again given by the set of finite
      multisets with finiteness structure as defined
      in~\cite{Mellies2018An-explicit}. The induced cartesian
      differential category is described in~\cite[\sec
      5.2]{Bucciarelli2010Categorical}.
  \item When $\V = \kmod$ for $k$ an algebraically closed field of
    characteristic zero, the cofree cocommutative coalgebra on a
    $k$-vector space $A$ is given by $\bigoplus_{x \in A} \mathrm{Sym}(A)$, 
    where $\mathrm{Sym}(A)$ is the symmetric algebra on $A$ as in~\eqref{eq:13}.
    This follows from results of~\cite{Sweedler1969Hopf}, and is spelt
    out in~\cite{Murfet2015On-Sweedlers}. In this case, the
    differential modality structure, and the induced cartesian
    differential category, are discussed in~\cite{Clift2020Cofree}. \label{item:5}
  \end{enumerate}
  \label{ex:7}
\end{Exs}

Note that Proposition~\ref{prop:3} produces \emph{monoidal}
differential modalities which, in the case of (ii), (iii) and (iv),
reside on monoidal \emph{closed} categories. Thus, by
Proposition~\ref{prop:2}, the co-Kleisli categories of these latter
examples are cartesian \emph{closed} differential categories. There
are also important differential modalities which are not monoidal, for
example on the category of $C^\infty$-rings; see~\cite[\sec
3]{Blute2006Differential} and \cite{Cruttwell2021Integral}.

Lastly, it may be worth mentioning that monoidal differential
modalities have an alternative axiomatisation as monoidal coalgebra
modalities equipped with a \emph{codereliction}; this is a natural
transformation ${\eta\colon A \to {!}A}$ satisfying certain identities
which correspond to evaluating the differential at zero
\cite{Blute2019Differential,Blute2006Differential,
  Ehrhard2018An-introduction, Fiore2007Differential}. These identities
involve the canonical maps
\begin{equation}
\begin{gathered}
  u_A \defeq I \xrightarrow{m_I} {!}I \xrightarrow{{!}0} {!}A \qquad
  \text{and} \\ \nabla_A \defeq {!} A \otimes {!} A
  \xrightarrow{\delta \otimes \delta} {!}{!}A \otimes {!}{!}A
  \xrightarrow{m_\otimes} {!}({!}A \otimes {!}A)
  \xrightarrow{{!}(\varepsilon \otimes e + e \otimes \varepsilon)} !A
\end{gathered}\label{eq:45}
\end{equation}
definable in any monoidal coalgebra modality on a symmetric monoidal
$k$-linear category, which together with $\Delta$ and $e$ endow each object ${!}A$ with
\emph{bialgebra} structure.
These same bialgebra maps are involved in the bijective correspondence
between deriving transformations and coderelictions, due
to~\cite[Theorem 4]{Blute2019Differential}. Indeed, the deriving
transformation corresponding to a codereliction
$\eta\colon A \to {!}A$ is given by:
\begin{equation*}
  {!}A \otimes A \xrightarrow{1 \otimes \eta} {!}A \otimes {!}A
  \xrightarrow{\nabla} {!}A\rlap{ ;}
\end{equation*}
while the codereliction of a deriving transformation $\mathsf{d}\colon {!}A
\otimes A \to {!}A$ is given by
\begin{equation*}
  A \xrightarrow{u} {!}A \otimes A \xrightarrow{\mathsf{d}} {!}A\rlap{ .}
\end{equation*}

While the formulation in terms of a codereliction is more common in
the literature on differential linear logic, for the purposes of the
present paper it will be deriving transformations which are the focus;
the algebra structure maps $u$ and $\nabla$ and codereliction $\eta$
will play no subsequent role.
 
\subsection{\texorpdfstring{$\faa(\kmodw)$}{Faa(kModw)} as a co-Kleisli construction}
\label{sec:faakmodw}

In this section, we show that, for the primordial left-$k$-linear
category $\kmodw$, its cofree cartesian differential category
$\faa(\kmodw)$ is induced by a particular (monoidal)
differential modality $Q$ on $\kmod$. To obtain $Q$, we could work
backwards from $\faa(\kmodw)$ using the results
of~\cite{Blute2015Cartesian}, but it will be more illuminating to
describe it directly.

In fact, we have already seen the formula for $Q$ in
Examples~\ref{ex:10}\ref{item:5}; there, $k$ was an algebraically
closed field of characteristic zero, and the formula
$\bigoplus_{x \in A} \mathrm{Sym}(A)$ in question gave the \emph{terminal} differential
modality on $\kmod$. For more general $k$, it turns out that this
formula still describes a differential modality $Q$ on $\kmod$, but the
universal property is different: it is the \emph{initial} monoidal
differential modality.

For the purposes in this paper, we will not actually require this
universal property, and so we reserve the proof of $Q$'s initiality
for a follow-up paper---where it will be considered in a more general
context---and content ourselves here with giving the explicit
formulae. Note that these extend the ones given in~\cite{Clift2020Cofree}
for $k$ an algebraically closed field of characteristic~zero.


\begin{Defn}
  \label{def:16}
  The \emph{initial monoidal differential modality} $Q$ on $\kmod$ is
  given as follows.
  \begin{itemize}
  \item On objects, we have $QA = \textstyle \bigoplus_{x \in A} \mathrm{Sym}(A)$.
      We will write $\spn{x_0, \dots, x_n} \in QA$ for the image of the
      pure tensor $x_1 \otimes \cdots \otimes x_n \in A^{\otimes n}$
      under the composite
      \begin{equation*}
        A^{\otimes n} \twoheadrightarrow A^{\otimes n} / \mathfrak{S}_n
        \xrightarrow{\iota_n} S(A) \xrightarrow{\iota_{x_0}} \textstyle\bigoplus_{x \in A} \mathrm{Sym}(A)
      \end{equation*}
      of the quotient map and two coproduct injections. In particular, when
      $n = 0$, we write $\spn{x_0}$ for the image of $1 \in A^{\otimes 0}$
      under the displayed composite. Note
      $x_0, \dots, x_n \mapsto \spn{x_0, \dots, x_n}$ is symmetric
      multilinear in $x_1, \dots, x_n$ but \emph{not} $x_0$.\vskip0.25\baselineskip
    \item On a map $f \colon A \rightarrow B$, we determine
      $Qf \colon QA \rightarrow QB$ by
      \begin{equation*}
        \spn{x_0, \dots, x_n} \mapsto \spn{f(x_0), \dots, f(x_n)}\rlap{ .}
      \end{equation*}
    \item For the comonad structure, the counit
      $\varepsilon_A \colon QA \rightarrow A$ is determined by
      \begin{equation*}
        \spn{x_0} \mapsto x_0\text{ ,} \qquad \spn{x_0, x_1} \mapsto x_1
        \qquad\text{and} \qquad 
        \spn{x_0, \dots, x_n} \mapsto 0 \text{ if $n \geqslant 2$.}
      \end{equation*}
      and the comultiplication $\delta_A \colon QA \rightarrow QQA$ is
      determined by
      \begin{equation*}
        \spn{x_0, \dots, x_n} \mapsto \sum_{[n] = A_1 \mid \cdots \mid
          A_k} \spn{ \spn{x_0}, \spn{x_{A_1}}, \dots, \spn{x_{A_k}} }
      \end{equation*}
      for $n \geqslant 0$. Here, as in Section~\ref{sec:faa-di-bruno},
      the sum is over unordered partitions of $[n]$; and we
      write $\spn{x_{I}}$ for $\spn{x_0, x_{i_1}, \dots, x_{i_k}}$.
      Note in particular that $\delta_A(\spn{x_0}) = \spn{\spn{x_0}}$.

      \vskip0.25\baselineskip
    \item For the coalgebra modality structure, the comonoid counit
      $e_A \colon QA \rightarrow k$ is determined~by
      \begin{equation*}
        \spn{x_0} \mapsto 1 \qquad \text{and} \qquad
        \spn{x_0, \dots, x_n} \mapsto 0 \text{ if $n \geqslant 1$,}
      \end{equation*}
      while the comultiplication map
      $\Delta_A \colon QA \rightarrow QA \otimes QA$ is determined by
      \begin{equation*}
        \spn{x_0,\dots, x_n} \mapsto \textstyle\sum_{I \subseteq [n]}
        \spn{x_I} \otimes \spn{x_{[n] \setminus I}}\rlap{ .}
      \end{equation*}
    \item For the monoidal structure, the nullary constraint
      $m_I \colon k \rightarrow Qk$ is determined by
      $1 \mapsto \spn{1}$, while the binary constraint
      $m_\otimes \colon QA \otimes QB \rightarrow Q(A \otimes B)$ is
      determined as follows (using the conventions of
      Notation~\ref{not:5} above):
      \begin{equation*}
        \spn{x_0, \dots, x_n} \otimes \spn{y_0, \dots, y_m} \mapsto 
        \sum_{\theta \colon [n] \simeq [m]}
        \spn{x_{\theta_{(1)}} \otimes y_{\theta_{(2)}}}\rlap{ .}
    \end{equation*}
    \item Finally, the deriving transformation
      $\mathsf{d}_A \colon QA \otimes A \rightarrow QA$ is determined by
      \begin{equation*}
        \spn{x_0,\dots, x_n} \otimes y \mapsto \spn{x_0, \dots, x_n, y}\rlap{ .}
      \end{equation*}
    \end{itemize}
\end{Defn}

The reader should have no trouble checking the axioms showing that
$(Q, \varepsilon, \delta)$ as described above is a comonad; that the
maps $(e, \Delta)$ endow it with the structure of a coalgebra
modality; and that $\mathsf{d}$ satisfies the deriving transformation
axioms. It is then an interesting exercise to obtain the given form of
the monoidal structure maps by first showing that the storage
maps~\eqref{eq:3} for $Q$ are invertible, and then deriving $m_I$ and
$m_\otimes$ via the formulae~\eqref{eq:42}.
For the sake of completeness, we also note that:
\begin{itemize}
\item The bialgebra maps~\eqref{eq:45} have the unit 
  $u_A\colon k \rightarrow QA$ determined by
  $1 \mapsto \spn{0,1}$
  and the multiplication $\nabla\colon QA \otimes QA \rightarrow QA$ 
  determined by
  \begin{equation*}
    \spn{x_0,\dots, x_n} \otimes \spn{y_0,\dots, y_m} \mapsto \spn{x_0 + y_0,x_1, \dots, x_n, y_1, \dots, y_m}\rlap{ .}
  \end{equation*}
\item The codereliction map $\eta_A \colon A \rightarrow QA$
  is given by $x \mapsto \spn{0,x}$.
\end{itemize}

\begin{Prop}
  \label{prop:4}
  The cartesian differential category $\faa(\kmodw)$ is induced by the
  initial monoidal differential modality $Q$ on $\kmod$.
\end{Prop}

\begin{proof}
  Clearly, objects of $\cat{Kl}(Q)$ and $\faa(\kmodw)$ are the same.
  On maps, since
  \begin{equation*}
    QA = \textstyle\bigoplus_{a \in A} SA =
    \bigoplus_{a \in A}(\bigoplus_{n \in \mathbb{N}}A^{\otimes n} /
    \mathfrak{S}_n) \cong \bigoplus_{n \in \mathbb{N}}(\bigoplus_{a \in A}A^{\otimes n} /
    \mathfrak{S}_n)\rlap{ ,}
  \end{equation*}
  we have a bijection between maps $A \rightarrow B$ in $\cat{Kl}(Q)$
  and in $\faa(\kmodw)$ by sending the $k$-linear map
  $f \colon QA \rightarrow B$ to the family of functions
  $f^{(\bullet)} \colon A \times A^n \rightarrow B$ with
  \begin{equation}\label{eq:16}
    f^{(n)}(x_0, \dots, x_n) = f(\spn{x_0, \dots, x_n})\rlap{ .}
  \end{equation}
  It is clear from the formula for
  $\varepsilon_A \colon QA \rightarrow A$ that identity maps correspond
  under this bijection. As for composition, we can read off from the
  formulae for $\delta_A$ and $Qf$ that the co-Kleisli composite of
  $f \colon QA \rightarrow B$ and $g \colon QB \rightarrow C$ is given
  by
  \begin{equation*}
    \spn{x_0, \dots, x_n} \quad \mapsto \quad \sum_{[n] = A_1 \mid \cdots \mid
      A_k} g\spn{ \,f\spn{x_0}, f\spn{x_{A_1}}, \dots,
      f\spn{x_{A_k}}\,}\rlap{ .}
  \end{equation*}
  Transforming this via the formula~\eqref{eq:16} and comparing
  with~\eqref{eq:8}, we conclude that this co-Kleisli composite
  corresponds to the Fa\'a di Bruno composite
  $(g \circ f)^{(\bullet)}$. So we have an isomorphism of categories
  $\cat{Kl}(Q) \cong \faa(\kmodw)$ which it is an easy exercise to
  check is an isomorphism of cartesian left-$k$-linear categories.

  Finally, we compare the differentials $\mathrm{D}$. For
  $\cat{Kl}(Q)$, this is computed via formula~\eqref{eq:17}; taking the
  image of a basis element of $Q(A \oplus A)$ under each of the maps in
  this composite in succession yields:
  \begin{align*}
    \spn{(x_0, y_0), \dots, (x_n, y_n)} & {} \mapsto
    \textstyle\sum_{I \subseteq [n]}
    \spn{x_I} \otimes \spn{y_{[n] \setminus I}}\\
    & {} \mapsto
    \spn{x_0, \dots, x_n} \otimes y_0 + \textstyle\sum_{i=1}^n
    \spn{x_0, \dots, x_{i-1}, x_{i+1}, \dots, x_n}
    \otimes y_i \\
    & {} \mapsto
    \spn{x_0, \dots, x_n, y_0} + \textstyle\sum_{i=1}^n
    \spn{x_0, \dots, x_{i-1}, y_i, x_{i+1}, \dots x_n}
  \end{align*}
  Transforming this via~\eqref{eq:16} and comparing with~\eqref{eq:11},
  we conclude that $\cat{Kl}(Q) \cong \faa(\kmodw)$ as cartesian
  differential categories.
\end{proof}
We remark in passing that, since $\kmod$ is symmetric monoidal closed,
and $Q$ is a monoidal differential modality, 
$\faa(\kmodw) \cong \cat{Kl}(Q)$ is a cartesian \emph{closed}
differential category, with the exponential $B^A$ of
$A,B \in \faa(\kmodw)$ given by the $k$-module of Fa\`a di Bruno maps
$A \rightsquigarrow B$.

\subsection{\texorpdfstring{$\faa(\A)$}{Faa(A)} as a co-Kleisli construction}
\label{sec:presheaves-left-k}

We now show that, for a general left-$k$-linear category $\A$,
there is a full structure-preserving embedding of $\faa(\A)$ into a
cartesian differential category induced by a (monoidal) differential modality.

\begin{Defn}
  \label{def:14}
  Let $\A$ be a left-$k$-linear category. We write
  $\cat{Psh}_\ell(\A)$ for the category $[\A^\mathrm{op}, \kmod]$ of
  $\kmod$-valued presheaves on the underlying ordinary category
  of~$\A$ (the nomenclature will be explained in
  Section~\ref{sec:presh-skew-mono} below). The left-$k$-linearity of $\A$ endows each representable
  $\A(\thg, A) \colon \A^\mathrm{op} \rightarrow \cat{Set}$
  with a lifting through the forgetful functor
  $\kmod \rightarrow \cat{Set}$, and we write
  $yA \in \cat{Psh}_\ell(\A)$ for the object so~obtained.
\end{Defn}

We view $\cat{Psh}_\ell(\A)$ as a symmetric monoidal $k$-linear
category, where both the monoidal structure and the $k$-linear
structure on the homs is given componentwise. With respect to this
structure, we have a monoidal differential modality on
$\cat{Psh}_\ell(\A)$ which is induced by postcomposition with the
initial monoidal differential modality $Q$ on $\kmod$. We call this
the \emph{pointwise initial monoidal differential modality}, and
denote it by abuse of notation by $Q$.

\begin{Defn}
  \label{def:23}
  Let $\A$ be a cartesian left-$k$-linear category. We write
  $\cat{Kl}_\A(Q)$ for the full subcategory of the co-Kleisli
  category of the pointwise initial monoidal differential modality $Q$
  on $\cat{Psh}_\ell(\A)$ on those objects of the form $yA$.
\end{Defn}

Note that we have $y(A \times B) \cong yA \times yB$, so that
$\cat{Kl}_\A(Q)$ is closed under finite products in the full
co-Kleisli category, and so is itself a cartesian differential
category. Our objective is to show it is isomorphic to
$\faa(\A)$; the key to which is a characterisation of objects of the
form $Q(yA) \in \cat{Psh}_\ell(\A)$. Towards this, we give:

\begin{Defn}
  \label{def:24}
  Let $\A$ be a cartesian left-$k$-linear category, and let
  $X \in \cat{Psh}_\ell(\A)$.
  \begin{enumerate}[(i),itemsep=0.35\baselineskip]
  \item We say that $x \in XA$ is \emph{$k$-linear} if each of the
    following functions is $k$-linear:
    \begin{equation}\label{eq:18}
      \begin{aligned}
        \A(B,A) & \rightarrow XB \\
        f & \mapsto Xf(x)\rlap{ .}
      \end{aligned}
    \end{equation}
    In other words, writing $x \cdot f$ for $Xf(x)$, we have
    $x \cdot (\lambda f+\mu g) = \lambda (x\cdot f) + \mu (x \cdot g)$.
  \item We say that $x \in X(A_1 \times \dots \times A_n)$ is
    \emph{$k$-linear in the $i$th variable} if it is $k$-linear as an
    element of the presheaf
    $X(A_1 \times \dots \times A_{i-1} \times (\thg) \times A_{i+1}
    \times \dots \times A_n)$.
  \item We say that $x \in X(A \times B^n)$ is \emph{symmetric in the
      last $n$ variables} if it is fixed by $X(1 \times \sigma)$ for all
    permutations $\sigma \colon B^n \rightarrow B^n$ of the product
    factors.
  \item A \emph{Fa\`a di Bruno sequence of $X$ at stage $A$} is a family
    of elements
    \begin{equation*}
      \bigl(x^{(n)} \in X(A \times A^n) : n \in \mathbb{N}\bigr)
    \end{equation*}
    such that each $x^{(n)}$ is symmetric and multilinear in its last $n$
    variables.
  \end{enumerate}
\end{Defn}

\begin{Lemma}
  \label{lem:13}
  Let $\A$ be a cartesian $k$-linear category, and let $A \in \A$. The
  object $Q(yA) \in \cat{Psh}_\ell(\A)$ classifies Fa\`a di Bruno sequences at
  stage $A$. More precisely, the following family of elements is a
  Fa\`a di Bruno sequence
  \begin{equation}\label{eq:20}
    \spn{\pi_0, \dots, \pi_n} \in Q\bigl(\A(A \times A^n,
    A)\bigr) = Q(yA)(A \times A^n)\rlap{ ,}
  \end{equation}
  and for any $X \in \cat{Psh}_\ell(\A)$ and Fa\`a di Bruno sequence $x^{(\bullet)}$
  of $X$ at stage $A$, there is a unique
  $\xi \colon Q(yA) \rightarrow X$ in $\cat{Psh}_\ell(\A)$ such that
  $\xi(\spn{\pi_0, \dots, \pi_{n}}) = x^{(n)}$ for each
  $n$.
\end{Lemma}

\begin{proof}
  We first verify that the elements~\eqref{eq:20} constitute a Fa\`a
  di Bruno sequence. For any sequence of maps
  $(f_0, \dots, f_n) \colon B \rightarrow A^n$ in $\A$, we have that
  \begin{equation}\label{eq:21}
    \spn{\pi_0,\dots, \pi_{n}} \cdot (f_0, \dots, f_n) =
    \spn{f_0,\dots, f_n} \in Q (yA)(B)\rlap{ ;}
  \end{equation}
  Since for each $f_0$ the assignment
  $f_1, \dots, f_n \mapsto \spn{f_0, \dots, f_n}$ is symmetric
  multilinear, we see that $\spn{\pi_0, \dots, \pi_{n}}$ is
  symmetric multilinear in its last $n$ variables, as desired.

  We now show universality of~\eqref{eq:20}. Given $x^{(\bullet)}$ a
  Fa\`a di Bruno sequence of $X \in \cat{Psh}_\ell(\A)$ at stage $A$, we define
  $\xi \colon Q(yA) \rightarrow X$ in $\cat{Psh}_\ell(\A)$ to have components
  \begin{align*}
    Q\A(B, A) & \rightarrow XB \\
    \spn{f_0, \dots, f_n} & \mapsto x^{(n)} \cdot (f_0, f_1, \dots, f_n)\rlap{ .}
  \end{align*}
  These are well-defined linear maps because of
  the symmetric multilinearity of $x$ in its last $n$ variables, and
  $\xi(\spn{\pi_0, \dots, \pi_n}) = x^{(n)} \cdot (\pi_0, \pi_1,
  \dots, \pi_n) = x^{(n)} \in X(A \times A^n)$. Moreover, if
  $\gamma \colon Q(yA) \rightarrow X$ in $\cat{Psh}_\ell(\A)$ satisfies
  $\gamma(\spn{\pi_0, \dots, \pi_n}) = x^{(n)}$ for each $n$,
  then we have
  $\gamma(\spn{f_0, \dots,  f_n}) = \gamma(\spn{\pi_0, \dots, \pi_n}) \cdot (f_0, \dots, f_n) = x^{(n)} \cdot (f_0, f_1, \dots, f_n)$
  by~\eqref{eq:21} and naturality of $\gamma$, so that $\gamma = \xi$
  as required.
\end{proof}

Given this result, it is now straightforward to prove:

\begin{Prop}
  \label{prop:5}
  Let $\A$ be a left-$k$-linear category. We have an isomorphism of cartesian differential
  categories $\faa(\A) \cong \cat{Kl}_\A(Q)$, so that $\faa(\A)$ admits a full structure-preserving embedding
  into the cartesian differential category induced by the pointwise
  initial monoidal differential modality $Q$
  on $\cat{Psh}_\ell(\A)$.
\end{Prop}

\begin{proof}
  Clearly, objects of $\faa(\A)$ and $\cat{Kl}_\A(Q)$ are the same.
  Now maps from $A$ to $B$ in $\cat{Kl}_\A(Q)$ are maps
  $f \colon Q yA \rightarrow yB$ in $\cat{Psh}_\ell(\A)$; by
  Lemma~\ref{lem:13}, these correspond to Fa\`a di Bruno sequences of
  $yB$ at stage $A$, whose data is precisely that of a map
  $f^{(\bullet)} \colon A \rightsquigarrow B$ in $\faa(\A)$. Observe that
  this $f^{(\bullet)}$ is characterised by 
  \begin{equation*}
    f^{(n)}(x_0, \dots, x_n) = f(\spn{x_0, \dots, x_n}) \in \A(B,A)
  \end{equation*}
  for all $X \in \A$ and $x_0, \dots, x_n \colon X \rightarrow A$;
  noting the formal similarity to~\eqref{eq:16}, we can thus conclude
  by transcribing the remainder of the proof of
  Proposition~\ref{prop:4}.
\end{proof}

\section{Enrichment over skew monoidal categories}
\label{sec:skew-enrichment}

We have just seen how to use the initial monoidal differential
modality $Q$ on $\kmod$ to construct cofree cartesian differential
categories. This suggests that the notion of cartesian differential
category is somehow controlled by the comonad $Q$; and the main result
of the paper, to be proved in the next section, will show that this is
indeed the case. There, we will see that that cartesian differential
categories arise as categories \emph{enriched} over $\kmod$ for a
monoidal structure which is not the usual one, but rather a certain
``warping'' of it controlled by the comonad $Q$.

As explained in the introduction, a slight complication is that this
warping is no longer a monoidal structure in the usual sense, but
rather a \emph{skew monoidal} structure in the sense of
Szlach\'anyi~\cite{Szlachanyi2012Skew-monoidal}. Since skew monoidal
structures, and categories enriched over them, are likely to be
unfamiliar to many, we devote this section to developing the necessary
notions. We note that Sections~\ref{sec:skew-mono-categ}
and~\ref{sec:enrichm-skew-mono} are revision from~\cite{
  Campbell2018Skew-enriched,Street2013Skew-closed,Szlachanyi2012Skew-monoidal},
but Sections~\ref{sec:change-base}--\ref{sec:finite-products} are
novel (though straightforward).

\subsection{Skew monoidal categories}
\label{sec:skew-mono-categ}
A skew monoidal category generalises a monoidal category by dropping
the requirement that the associativity and unitality constraint maps
be invertible. Of course, it then matters how we choose to orient
these maps, and ``skew monoidal'' refers to the following particular choice.

\begin{Defn}
  \label{def:1}\cite{Szlachanyi2012Skew-monoidal}
  A \emph{skew monoidal category} is a category $\V$ endowed
  with a tensor product $\otimes \colon \V \times \V
  \rightarrow \V$ and unit object $I \in \V$, together with natural
  families
  \begin{equation*}
    \alpha_{ABC} \colon (A \otimes B) \otimes C \rightarrow A \otimes
    (B \otimes C) \qquad \lambda_A \colon I \otimes A \rightarrow A
    \qquad \rho_A \colon A \rightarrow A \otimes I
  \end{equation*}
  of maps satisfying the Mac Lane associativity pentagon, the condition
  $\lambda_I \circ \rho_I = 1_I$, and the three unit conditions (where we
  omit tensor symbols for compactness):
  \begin{equation*}
    \cd{
      {A B} \ar[r]^-{1} \ar@{<-}[d]_{\lambda_A 1} &
      {A B} \ar@{<-}[d]^{\lambda_{A B}} &
      {A B} \ar[r]^-{1} \ar[d]_{\rho_A 1} &
      {A B} \ar@{<-}[d]^{1 \lambda_B} &
      {A B} \ar[r]^-{1} \ar[d]_{\rho_{AB}} &
      {A B} \ar@{<-}[d]^{1 \rho_B} \\
      {(I A) B} \ar[r]_-{\alpha_{IAB}} &
      {I (A B)} &
      {(A I) B} \ar[r]_-{\alpha_{AIB}} &
      {A (I B)} &
      {(A B) I} \ar[r]_-{\alpha_{ABI}} &
      {A (B I)}\rlap{ .}
    }
  \end{equation*}
  A skew monoidal category is said to be \emph{left closed} if each
  functor $(\thg) \otimes A \colon \V \rightarrow \V$ has a right
  adjoint $[A, \thg]$.
\end{Defn}
More precisely, we have just defined a \emph{left-skew} monoidal
category; a \emph{right-skew} monoidal category reverses the
directions of all three maps, but our convention will always be that
``skew'' means ``left-skew''.

\begin{Ex}
  \label{ex:4}
  Let $(M, +, 0)$ be a monoid. There is a left closed skew monoidal
  structure on $\cat{Set}$ with unit $1 = \{\ast\}$, tensor product
  $A \times^M B = A \times M \times B$, and constraint
  maps 
  given as follows (where juxtaposition denotes cartesian product):
  \begin{align*}
    \lambda \colon 1MB &\rightarrow B & \rho \colon A & \mapsto AM1 &
    \alpha \colon (AMB)MC & \rightarrow AM(BMC) \\
    (\ast, m, b) & \mapsto b & a & \mapsto (a,0,\ast) & ((a,m,b),n,c) &
    \mapsto (a, m+n, (b, n, c))\rlap{ .}
  \end{align*}
  The associated internal hom is given by
  $B \Rightarrow^M C = C^{M \times B}$.
\end{Ex}



This example is an instance of the following important construction,
which builds skew monoidal structures from monoidal ones:


\begin{Defn}
  \label{def:3}(\cite[Proposition~7.2]{Szlachanyi2012Skew-monoidal}).
  Let $(\V, \otimes, I)$ be a monoidal category endowed with a
  monoidal comonad $({!}, \varepsilon, \delta, m_\otimes, m_I)$. The
  \emph{fusion operator}~\cite[\sec 2.6]{Bruguieres2011Hopf} of ${!}$,
  which we use repeatedly in what follows, is the natural
  transformation $H$ with components
  \begin{equation}\label{eq:39}
    H \defeq {!}A \otimes {!}B \xrightarrow{1 \otimes \delta} {!}A
    \otimes {!}{!}B \xrightarrow{m_\otimes} {!}(A \otimes {!}B)\rlap{ .}
  \end{equation}
  The
  \emph{skew-warping} of $\otimes$ with respect to ${!}$ is the skew
  monoidal structure on $\V$ with unit $I$, with tensor
  $A \otimes^{!} B = A \otimes {!}B$ and with constraint cells
  \begin{gather*}
    (A \!\otimes\! {!}B) \!\otimes\! {!}C
    \xrightarrow{\alpha} A \!\otimes\! ({!}B \!\otimes\! {!}C)
    \xrightarrow{1 \otimes H} A \!\otimes\! {!}(B \!\otimes\! {!}C)\\
    I \otimes {!}A \xrightarrow{\lambda} {!}A
    \xrightarrow{\varepsilon} A \quad \text{and} \quad 
    A \xrightarrow{\rho} A \otimes I \xrightarrow{1
      \otimes m_I} A \otimes {!}I \rlap{ .}
  \end{gather*}
  If the monoidal structure $(\otimes, I)$ is left closed, then so
  too is $(\otimes^{!}, I)$, with the associated internal hom given by
  $[B, C]^{!} = [{!}B, C]$.
  We will often write\footnote{Note that we will \emph{never} use $\V^!$ to denote the co-Kleisli category of $!$.} $\V^{!}$ to denote $\V$ endowed with
  its warped skew monoidal structure $(\otimes^!, I)$.
\end{Defn}

For example, if $(M,+,0)$ is a monoid, then $M \times (\thg)$ is a
monoidal comonad on $\cat{Set}$, where $\varepsilon$ and $\delta$
involve projection and duplication, and $m_\otimes$ and $m_I$ involve
the monoid structure of $M$. Instantiating Definition~\ref{def:3} at
this monoidal comonad recovers the skew monoidal structure of
Example~\ref{ex:4}.

The following further example provides a first indication of the relevance of
skew monoidal structure to cartesian differential categories.
\begin{Ex}
  \label{ex:2}
  Let $\V$ be a monoidal category with all copowers
  $X \cdot I \defeq \sum_{x \in X} I$ of the unit $I$, whose tensor
  product preserves these copowers in the second variable. We have a monoidal
  adjunction
  \begin{equation}\label{eq:2}
    \cd[@C+2em]{
      {\V} \ar@<-4.5pt>[r]_-{\V(I, \thg)} \ar@{<-}@<4.5pt>[r]^-{(\thg)
        \cdot I} \ar@{}[r]|-{\bot} &
      {\cat{Set}} 
    }
  \end{equation}
  inducing a monoidal comonad $K$ on $\V$; in fact, it is the
  \emph{initial} monoidal comonad on $\V$. The warped monoidal
  structure $\otimes^{K}$ is characterised by the fact that maps
  $A \otimes KB \rightarrow C$ are the same as functions
  $\V(I, B) \rightarrow \V(A,C)$. In particular, when $\V = \kmod$, maps $A \otimes^K B \rightarrow C$ are
  exactly \emph{left-$k$-linear} maps $A \times B \rightarrow C$.
\end{Ex}

\subsection{Enrichment in a skew monoidal category}
\label{sec:enrichm-skew-mono}

We now turn to the notion of a category enriched in a skew monoidal
category. Our definition follows Street~\cite[\sec
10]{Street2013Skew-closed}, and mimics exactly the notion of
enrichment over a genuine monoidal category. We remark
that~\cite{Campbell2018Skew-enriched} gives a subtler notion of
\emph{skew-enriched category}, involving further extra data beyond the
obvious. While there are good reasons to require these extra data, for
our purposes it will prove unnecessary; though see Remark~\ref{rk:1}
below.

\begin{Defn}
  \label{def:2}
  Let $(\V, \otimes, I)$ be a skew monoidal category. A
  \emph{$\V$-enriched category} $\A$ comprises the data of a set
  $\mathrm{ob}(\A)$ of objects; hom-objects $\A(A,B) \in \V$ for each
  $A,B \in \mathrm{ob}(\A)$; and composition and identity morphisms in $\V$
  \begin{equation*}
    m_{ABC} \colon \A(B,C) \otimes \A(A,B) \rightarrow \A(A,C) \qquad
    \text{and} \qquad i_A \colon I \rightarrow \A(A,A)\rlap{ ;}
  \end{equation*}
  for all $A,B,C \in \mathrm{ob}(\A)$. These data are required to
  satisfy the axioms expressed by the commutativity of the diagrams
  \begin{equation*}
    \cd[@-1em@C-1em]{
      \sh{l}{2em}{\bigl(\A(C,D)\otimes\A(B,C)\bigr)\otimes\A(A,B)}
      \ar[r]^-{m \otimes 1} \ar[dd]_{\alpha} &
      \sh{r}{1em}{\A(B,D)\otimes\A(A,B)} \ar[dr]^-{m} \\
      & & \A(A,D)\\
      \sh{l}{2em}{\A(C,D)\otimes\bigl(\A(B,C)\otimes\A(A,B)\bigr)}
      \ar[r]^-{1 \otimes m} &
      \sh{r}{1em}{\A(C,D)\otimes\A(A,C)} \ar[ur]^-{m}
    }
  \end{equation*}
  \begin{equation*}
    \cd[@!C@C-6em]{
      {I\otimes\A(A,B)} \ar[rr]^-{i \otimes 1} \ar[dr]_-{\lambda} & &
      {\A(B,B)\otimes\A(A,B)} \ar[dl]^-{m} \\ &
      {\A(A,B)}
    } \qquad 
    \cd[@!C@C-2.6em]{
      {\A(A,B)\otimes I} \ar[r]^-{1  \otimes i} \ar@{<-}[d]_-{\rho} & 
      {\A(A,B)\otimes\A(A,A)} \ar[d]^-{m} \\ 
      {\A(A,B)} \ar[r]^-{1} & {\A(A,B)}\rlap{ .}
    }
  \end{equation*}
  The \emph{underlying ordinary category} $\A_0$ has the same
  objects as $\A$, hom-sets $\A_0(A,B) = \V(I, \A(A,B))$, and composition of
  $f \in \A_0(A,B)$ and $g \in \A_0(B,C)$ given by
  \begin{equation}\label{eq:32}
    I \xrightarrow{\rho_I} I \otimes I \xrightarrow{g \otimes f}
    \A(B,C) \otimes \A(A,B) \xrightarrow{m} \A(A,C)\rlap{ .}
  \end{equation}
  We may speak of a \emph{map} in  $\A$ to
  mean a map in the underlying category $\A_0$.
\end{Defn}

\begin{Ex}
  \label{ex:3}
  If the monoidal category $\V$ admits an initial monoidal comonad $K$
  as in Example~\ref{ex:2}, then we can consider categories enriched
  in the skew-warping $\V^{K}$. Such a category $\A$
  involves a set of objects, hom-$\V$-objects $\A(A,B)$, identities
  $I \rightarrow \A(A,A)$ and composition maps of the form
  \begin{equation*}
    \A(B,C) \otimes^K \A(A,B) \rightarrow \A(A,C)\rlap{ .}
  \end{equation*}
  In light of Example~\ref{ex:2}, it follows that 
  when $\V = \kmod$ we recapture exactly the notion of
  \emph{left-$k$-linear category}; this was observed in passing, and
  without any details being given, in~\cite[\sec 5.1]{Blute2015Cartesian}.

  In general, to give a $\V^K$-category is to give what
  we might call a \emph{left-$\V$-enriched category}: this is an
  ordinary category $\A_0$ together with, for each $B \in \A_0$, a
  lifting of the hom-functor
  $\A_0(\thg, B) \colon \A_0^\mathrm{op} \rightarrow \cat{Set}$
  through $\V(I, \thg)$:
 \begin{equation*}
 \cd[@+1em@C+0.5em]{
 & \V \ar[d]^-{\V(I,\thg)} \\
 \A_0^\mathrm{op} \ar[ur]^-{\A(\thg, B)}
 \ar[r]_-{\A_0(\thg, B)} & \cat{Set}\rlap{ .}
 }
 \end{equation*}
\end{Ex}

\subsection{Change of enrichment base}
\label{sec:change-base}
In classical enriched category theory, change of base allows us to
turn a $\V$-enriched category into a $\W$-enriched one via a monoidal
functor $\V \rightarrow \W$. This works just as well in the skew
context. 
\begin{Defn}
  \label{def:28}
  A \emph{monoidal functor}
  $(F, m_I, m_\otimes) \colon \V \rightarrow \W$ between skew monoidal
  categories comprises a functor $F \colon \V \rightarrow \W$ together
  with a map $m_I \colon I \rightarrow FI$ and a natural family of
  maps $m_\otimes \colon FA \otimes FB \rightarrow F(A \otimes B)$,
  rendering commutative each of the diagrams:
  \begin{gather*}
    \cd{
      {I \otimes FA} \ar[r]^-{m_I \otimes 1} \ar[d]_{\lambda} &
      {FI \otimes FA} \ar[d]^{m_\otimes} &
      {FA} \ar[r]^-{\rho} \ar[d]_{F\rho} &
      {FA \otimes I} \ar[d]^{1 \otimes m_I} 
      \\
      {FA} \ar@{<-}[r]^-{F\lambda} &
      {F(I \otimes A)} & F(A \otimes I) \ar@{<-}[r]^-{m_\otimes} &
      {FA \otimes FI}
    }\\
    \cd{
      (FA \otimes FB) \otimes FC \ar[r]^-{\alpha} \ar[d]_-{m_\otimes
        \otimes 1} & FA \otimes (FB \otimes FC) \ar[r]^-{m_\otimes} &
      FA \otimes F(B \otimes C) \ar[d]^-{m_\otimes} \\
      F(A \otimes B) \otimes FC \ar[r]^-{m_\otimes} & F((A \otimes B)
      \otimes C) \ar[r]^-{F\alpha} & F(A \otimes (B \otimes C))\rlap{ .}
    }
  \end{gather*}

  If $\A$ is a $\V$-enriched category, then its \emph{base change}
  $F_\ast \A$ is the $\W$-enriched category with the same objects,
  hom objects $(F_\ast \A)(A,B) = F(\A(A,B))$, and identities and
  composition given by
  \begin{gather*}
    I \xrightarrow{m_I} FI \xrightarrow{Fi}F\A(A,A) \quad \text{ and}
    \\
    F\A(B,C)
    \otimes F\A(A,B) \xrightarrow{m_\otimes} F(\A(B,C) \otimes
    \A(A,B)) \xrightarrow{Fm} F\A(A,C)\rlap{ .}
  \end{gather*}
\end{Defn}

\begin{Exs}
  \label{ex:18}
  \begin{enumerate}[(i),itemsep=0.25\baselineskip]
  \item For any skew monoidal category $\V$ the functor
    ${\V(I, \thg) \colon \V \rightarrow \cat{Set}}$ is monoidal, where
    $m_I \colon 1 \rightarrow \V(I,I)$ picks out the identity, and
    where the map
    $m_\otimes \colon \V(I,A) \times \V(I,B) \rightarrow \V(I,A
    \otimes B)$ takes $f,g$ to the composite
    \begin{equation*}
      I \xrightarrow{\rho_I} I \otimes I \xrightarrow{f \otimes g} A
      \otimes B\rlap{ .}
    \end{equation*}
    Base change along $\V(I,\thg)$ sends a $\V$-category $\A$ to its
    underlying category $\A_0$.
  \item If $\V$ is a monoidal category and $P,Q$ are
    monoidal comonads on $\V$, then each map of monoidal comonads
    $\gamma \colon P
    \rightarrow Q$ gives a monoidal functor $\mathrm{id} \colon \V^Q \rightarrow \V^P$
    with nullary constraint $m_I$ the identity, and binary
    constraints given by
    \begin{equation*}
      \mathrm{id}(A) \otimes^P \mathrm{id}(B) = A \otimes PB \xrightarrow{1 \otimes \gamma_B} A
      \otimes QB = \mathrm{id}(A \otimes^Q B)\rlap{ .}
    \end{equation*}
    This induces a base change operation
    $\A \mapsto \gamma^\ast(\A)$ from $\V^Q$- to
    $\V^P$-categories; its only effect is to turn the
    composition maps $\A(B,C) \otimes Q\A(A,B) \rightarrow \A(A,C)$
    into ones $\A(B,C) \otimes Q\A(A,B)
    \rightarrow \A(A,C)$ by precomposing with $1 \otimes \gamma$.

    For example, if $\V$ supports the initial monoidal comonad $K$,
    then the unique map to the terminal monoidal comonad
    $\mathrm{id}$, namely
    $\varepsilon \colon K \rightarrow \mathrm{id}$, gives a change
    of base functor from $\V$-enriched categories to
    left-$\V$-enriched categories. In particular, this is how
    $k$-linear categories can be viewed as left-$k$-linear.
  \end{enumerate}
\end{Exs}


\subsection{\texorpdfstring{$\V$}{V}-linear maps in an enriched category}
\label{sec:v-linear-maps}

Since left-$k$-linear categories can be seen as categories enriched over
a skew monoidal base, it is reasonable to ask if there is an analogue
for a general skew monoidal enrichment of the notion of
\emph{$k$-linear map}. The answer is ``yes'', but some care is needed:
in full generality, being ``linear'' may be structure on a map, rather
than a property of it.

\begin{Defn}
  \label{def:26}
  Let $\V$ be a skew monoidal category, $\A$ a $\V$-enriched category
  and $A,B \in \A$. A \emph{$\V$-linear map}
  $f \colon A \rightarrow_\ell B$ comprises maps
  $f_X \colon \A(X,A) \rightarrow \A(X,B)$ in $\V$ which are
  $\V$-natural, in the sense of rendering commutative each diagram
  \begin{equation}\label{eq:30}
    \cd{
      \A(Y, A) \otimes \A(X,Y) \ar[d]_-{m} \ar[r]^-{f_Y \otimes 1} & \A(Y,B) \otimes
      \A(X,Y) \ar[d]^-{m} \\
      \A(X,A) \ar[r]^-{f_X} & \A(X,B)\rlap{ .}
    }
  \end{equation}
  The \emph{underlying map} $f_0 \colon A \rightarrow B$ of $f$ is
  given by the composite
  \begin{equation*}
    I \xrightarrow{i} \A(A,A) \xrightarrow{f_A} \A(A,B)\rlap{ .}
  \end{equation*}
  The objects and $\V$-linear maps form a category $\A_\ell$, and the
  assignment $f \mapsto f_0$ is the action on maps of an
  identity-on-objects functor $\A_\ell \rightarrow \A_0$.
\end{Defn}

When $\V$ is genuinely monoidal, a weak form of the Yoneda lemma shows
that $\V$-linear maps are in bijection with maps of the underlying
category, i.e., $\A_\ell \rightarrow \A_0$ is an isomorphism. However,
in the skew case, it may not even be true that a $\V$-linear map $f$
is determined by its underlying map $f_0$. The following is the
diagram which would usually be drawn to prove this statement:
\begin{equation}\label{eq:25}
  \cd{
    I \otimes \A(X,A) \ar@/^2em/[rr]^-{f_0 \otimes 1} \ar[dr]_-{\lambda} \ar[r]^-{i \otimes 1} & \A(A, A) \otimes \A(X,A) \ar[d]_-{m} \ar[r]^-{f_A \otimes 1} & \A(A,B) \otimes
    \A(X,A) \ar[d]^-{m} \\ {} & 
    \A(X,A) \ar[r]^-{f_X} & \A(X,B) \rlap{ ;}
  }
\end{equation}
but in the skew context, $\lambda$ need not be invertible, so
invalidating the formula
$f_X = m \circ (f_0 \otimes 1) \circ \lambda^{-1}$ which would
determine $f_X$ from $f_0$. 

\begin{Ex}
  \label{ex:16}
  Let $\V$ be the monoidal category of $\mathbb{N}$-graded $k$-modules with its
  usual tensor product
  $(A \otimes B)_n = \sum_{n = p+q}A_p \otimes B_q$, and let $\V^K$ be
  its skew-warping for the initial monoidal comonad. A category $\A$
  enriched over
  $\V^K$ has graded $k$-modules of maps
  $\A(A,B)$, identities $\mathrm{id}_A \in \A(A,A)_0$, and
  composition given by left-$k$-linear maps $\A(B,C)_n \times
  \A(A,B)_0 \rightarrow \A(A,C)_n$. The leading example is the
  category $\A$ whose objects are graded $k$-modules, and for which
  $\A(A,B)_n$ is the set of functions $A_0 \rightarrow B_n$. In this
  case, a map in the underlying category is simply a function $A_0
  \rightarrow B_0$, while a $\V^K$-linear map can be
  calculated to be a genuine map of graded $k$-modules $A \rightarrow
  B$; clearly, the former does not determine the latter.
\end{Ex}

This leads us to make the following definition, which captures the
situation in which $\V$-linearity is, in fact, a property rather than
a structure:
\begin{Defn}
  \label{def:27}
  A skew monoidal category $(\V, \otimes, I)$ is called \emph{left
    covering} if each map $\lambda_A \colon I \otimes A \rightarrow A$ is
  an epimorphism.
\end{Defn}
In this case, the standard Yoneda lemma argument
via~\eqref{eq:25} now proves:
\begin{Lemma}
  \label{lem:9}
  If $\V$ is a left covering skew monoidal category, and $\A$ is a
  $\V$-category, then to give a $\V$-linear map $f \colon A
  \rightarrow_\ell B$ is equally to give a map $f_0 \colon A \rightarrow
  B$ in (the underlying category of) $\A$ for
  which each of the following factorisations exists:
\begin{equation}\label{eq:26}
  \cd{
    I \otimes \A(X,A) \ar[r]^-{f_0 \otimes 1} \ar@{->>}[d]_-{\lambda} & \A(A,B) \otimes
    \A(X,A) \ar[d]^-{m} \\ \A(X,A) \ar@{-->}[r]^-{f_X} & \A(X,B)\rlap{ .}
  }
\end{equation}
In particular, when $\V$ is left covering, $\A_\ell \rightarrow \A_0$
is a faithful functor.
\end{Lemma}
\begin{Ex}
  \label{ex:15}
  If $!$ is a monoidal comonad on the monoidal category $\V$, then the
  skew-warping $\V^!$ is left covering just when each counit map
  $\varepsilon_X \colon !X \rightarrow X$ is epimorphic.
  In particular, if $\V$ admits the initial monoidal comonad $K$, then
  $\V^K$ is left covering precisely when the unit object $I$ is a
  generator for $\V$.

  In this case, if $\A$ is a $\V^K$-category, then a map
  $f \in \A_0(A,B)$ is $\V^K$-linear precisely when each composition
  function $f \circ (\thg) \colon \A_0(X,A) \rightarrow \A_0(X,B)$ is
  the image under the faithful functor
  $\V(I,\thg) \colon \V \rightarrow \cat{Set}$ (faithful since $I$ is
  a generator) of a map $\A(X,A) \rightarrow \A(X,B)$ in $\V$. In
  particular, when $\V = \kmod$, a $\V^K$-linear map is precisely a
  \emph{$k$-linear map} in the sense of Definition~\ref{def:5}.
\end{Ex}

\subsection{Finite products in an enriched category}
\label{sec:finite-products}

Since we are interested in cartesian differential categories, we will
need to understand the notion of \emph{finite product} in the
skew-enriched context. We conclude this section by discussing
this.
\begin{Defn}
  \label{def:9}
  Let $\V$ be a skew monoidal category and $\A$ a $\V$-enriched
  category.
  \begin{enumerate}[(i),itemsep=0.25\baselineskip]
  \item An object $1 \in \A$ is \emph{terminal} for $\A$ if
    $\A(X,1)$ is terminal in $\V$ for all $X \in \A$.
  \item A \emph{binary product} of $A,B \in \A$ comprises an object $A
    \times B \in \A$ and a span of
    $\V$-linear maps $\pi_0 \colon A
    \leftarrow_\ell A \times B \rightarrow_\ell B \colon \pi_1$, each of
    whose components
    $\A(X, A) \leftarrow \A(X, A \times B) \rightarrow
    \A(X, B)$ constitutes a product diagram in $\V$.
  \item $\A$ is \emph{cartesian} when it has a terminal object and all
    binary products.
  \end{enumerate}
\end{Defn}
From Lemma~\ref{lem:9} it follows immediately that:
\begin{Lemma}
  \label{lem:11}
  If $\V$ is a left covering skew monoidal category, then a
  $\V$-category $\A$ is cartesian just when its underlying
  ordinary category has finite products, and all binary product
  projections are $\V$-linear.
\end{Lemma}

\begin{Ex}
  \label{ex:9}
  A left-$k$-linear category $\A$ is cartesian \emph{qua}
  $\kmod^K$-category precisely when it is cartesian
  left-$k$-linear.
\end{Ex}

%
%
%

\section{Cartesian differential categories as enriched
  categories}
\label{sec:cart-diff-categ-2}

In this section, we give the first main result of this paper,
exhibiting cartesian differential categories as cartesian
$\kmod^Q$-enriched categories, where $Q$ is the initial monoidal
differential modality of Definition~\ref{def:16}.

\subsection{Characterising \texorpdfstring{$\kmod^Q$}{kModQ}-categories}
\label{sec:kmodq-categories}
Before proving the main theorem, we identify general
$\kmod^Q$-categories; these turn out to be a variant of cartesian
differential categories which do away with the need for finite
products. An example of this notion would be the co-Kleisli category
of the differential modality on a tensor differential category
\emph{without} finite products.

We first record an explicit description of the fusion
map~\eqref{eq:39} for $Q$.

\begin{Lemma}
  \label{lem:1}
  The fusion map $H \colon QX \otimes QY \rightarrow Q(X \otimes QY)$
  for the initial monoidal differential modality $Q$ on $\kmod$ has
  action determined by
  \begin{equation*}
    \spn{x_0, \dots, x_m} \otimes \spn{y_0, \dots, y_n} \quad \mapsto
    \quad \sum_{\substack{[n] = A_1 \mid \dots \mid A_k \\ \theta \colon [m] \simeq [k]}} \spn{
      x_{\theta_{(1)}} \otimes
      \spn{y_{A_{\theta_{(2)}}}}}\rlap{ ,}
  \end{equation*}
  with the conventions of
  Notation~\ref{not:5}, and with $A_0 \defeq \emptyset$
  so that $\spn{y_{A_0}} = \spn{y_0}$.
\end{Lemma}
\begin{proof}
  This can simply be read off from Definition~\ref{def:16}.
\end{proof}
\begin{Prop}
  \label{prop:8}
  To give a $\kmod^Q$-enriched category $\A$ is equally to give
  a collection of objects; a $k$-module $\A(A,B)$ of maps between each
  pair of objects; identity elements $\mathrm{id}_A \in \A(A,A)$; and
  composition functions
  \begin{equation}
    \label{eq:19}
    \begin{aligned}
      \A(B,C) \times \A(A,B) \times \A(A,B)^n & \rightarrow \A(A,C) \\
      (g, f_0, \dots, f_n) & \mapsto g^{(n)}(f_0, \dots, f_n)
    \end{aligned}
  \end{equation}
  for each $A,B,C \in \A$ and each $n \geqslant 0$, subject to the
  following axioms:
  \begin{enumerate}[(i)]
  \item Each~\eqref{eq:19} is $k$-linear in $g$, and symmetric
    $k$-linear in $f_1, \dots, f_n$;
  \item We have $g^{(0)}(\mathrm{id_A}) = g$;
  \item We have $\mathrm{id}_B^{(0)}(f) = f$,
    $\mathrm{id}_B^{(1)}(f_0, f_1) = f_1$ and
    $\mathrm{id}_B^{(n)}(f_0, \dots, f_n) = 0$ for all $n \geqslant 2$;
  \item For all $f_0, \dots, f_n \colon A \rightarrow B$, $g_0, \dots, g_m
    \colon B \rightarrow C$ and $h \colon C \rightarrow D$ we have
    \begin{equation*}
      \displaystyle\bigl(h^{(m)}(g_0, \dots, g_m)\bigr)^{(n)}(\vec f)
      = \sum_{\substack{[n] = A_1 \mid \dots \mid A_k\\
          \theta \colon [m] \simeq [k]}}
      h^{(\abs{\theta})}\bigl(g_{\theta_{(1)}}^{(A_{\theta_{(2)}})}(\vec f)\bigr) \rlap{ ,}
    \end{equation*}
    where we define
    $g_i^{(A_{j})}(\vec f)$ as in~\eqref{eq:5}.
  \end{enumerate}
\end{Prop}
   
\begin{proof}
  A $\kmod^Q$-category $\A$ has
  objects $A,B,\dots$; homs $\A(A,B) \in \kmod$; identity maps
  $i \colon k \rightarrow \A(A,A)$ which pick out elements
  $\mathrm{id}_A \in \A(A,A)$; and composition maps
  $m \colon \A(B,C) \otimes Q\A(A,B) \rightarrow \A(A,C)$, which, on writing their action as:
  \begin{equation*}
    g \otimes \spn{f_0, \dots, f_n}  \qquad \mapsto \qquad g^{(n)}(f_0,
    \dots, f_n)\rlap{ ,}
  \end{equation*}
  correspond to families of maps~\eqref{eq:19} satisfying
  axiom (i) in the statement. Now the right identity axiom for $\A$ requires commutativity of:
  \begin{equation*}
    \cd{
      \A(A,B) \ar[r]^-{1 \otimes m_I}  \ar[d]_-{1} &
      \A(A,B) \otimes QI
      \ar[d]^-{1 \otimes Qi} \\
      \A(A,B)  & \A(A,B) \otimes Q\A(A,B)\rlap{ .} \ar[l]_-{m}
    }
  \end{equation*}
  Chasing $g \in \A(A,B)$ around the long side we get $g \mapsto g \otimes
  \spn{\mathrm{id}_A} \mapsto g^{(0)}(\mathrm{id}_A)$, 
  so that commutativity is exactly axiom (ii).  
  Next, the left identity axiom requires commutativity in:
  \begin{equation*}
    \cd{
      Q\A(A,B) \ar[dr]^-{\varepsilon} \ar[r]^-{i \otimes 1} & \A(B,B) \otimes
      Q\A(A,B) \ar[d]^-{m} \\
      & \A(A,B)\rlap{ .}
    }
  \end{equation*}
  The upper composite takes $\spn{f_0, \dots, f_n}$ to
  $\mathrm{id}_B^{(n)}(f_0, \dots, f_n)$, and so comparing with the
  formula for $\varepsilon$, commutativity of this diagram is exactly axiom (iii).
  Finally, the associativity axiom requires commutativity in:
  \begin{equation*}
    \cd[@C-0.8em]{
      \A(C,D) \otimes Q\A(B,C) \otimes Q\A(A,B) \ar[d]_-{1 \otimes H} \ar[rr]^-{m \otimes 1} & &
      \sh{l}{3em}\A(B,D) \otimes Q\A(A,B) \ar[d]_-{m} \\
      \A(C,D) \otimes Q(\A(B,C) \otimes Q\A(A,B)) \ar[r]^-{1 \otimes Qm} &
      \sh{r}{1.5em}\A(C,D) \otimes Q\A(A,C) \ar[r]^-{m} &
      \A(A,D)\rlap{ .}
    }
  \end{equation*}
  Chasing a generating element $h \otimes \spn{g_0, \dots, g_m}
  \otimes \spn{f_0, \dots, 
        f_n}$ around the top composite yields $(h^{(m)}(g_0,
      \dots, g_m))^{(n)}(\vec f)$. On the other hand, chasing this
      generator around the lower composite yields in succession (using
      Lemma~\ref{lem:1} at the first step):
      \begin{align*}
        & \ \ \ \ h \otimes \spn{g_0, \dots, g_m} \otimes
        \spn{f_0, \dots, f_n} \\
      & \mapsto \sum_{\substack{[n] = A_1 \mid \dots \mid A_k \\ \theta \colon [m] \simeq [k]}}  h \otimes \spn{
          g_{\theta_{(1)}} \otimes
          \spn{f_{A_{\theta_{(2)}}}}}      \\
        & \mapsto \sum_{\substack{[n] = A_1 \mid \dots \mid A_k \\ \theta \colon [m] \simeq [k]}}  h \otimes \spn{
          g_{\theta_{(1)}}^{(A_{\theta_{(2)}})}(\vec f)}
      \mapsto \sum_{\substack{[n] = A_1 \mid \dots \mid A_k \\ \theta
          \colon [m] \simeq [k]}}  h^{(\abs \theta)} \bigl(
      g_{\theta_{(1)}}^{(A_{\theta_{(2)}})}(\vec f)\bigr)\text{, }
    \end{align*}
    so that the associativity axiom is equivalent
    to axiom (iv) in the statement.
\end{proof}

\subsection{Characterising cartesian \texorpdfstring{$\kmod^Q$}{kModQ}-categories}
\label{sec:kmodq-categ-with}

We now consider what it means for a $\kmod^Q$-category to have finite
products; this will bridge the gap with cartesian differential
categories. Note first that, since the counit maps
$\varepsilon_X \colon QX \rightarrow X$ of Definition~\ref{def:16} are
visibly epimorphic, $\kmod^Q$ is a left covering skew monoidal
category. Thus, by Lemma~\ref{lem:11}, a $\kmod^Q$-category is
cartesian just when its underlying ordinary category has finite
products, and product projections are $\kmod^Q$-linear. The following
result characterises this notion of linearity.

\begin{Prop}
  \label{prop:9}
  A map $g \colon A \rightarrow B$ of a $\kmod^Q$-category $\A$ is
  $\kmod^Q$-linear just when, for all $X \in \A$ and $f_0, f_1, \ldots
  \in \A(X,A)$ we have:
  \begin{equation}
    g^{(1)}(f_0, f_1) = g^{(0)}(f_1) \qquad \text{and} \qquad g^{(n)}(f_0,
    \dots, f_n) = 0 \text{ for all $n \geqslant 2$.}\label{eq:27}
  \end{equation}
\end{Prop}
\begin{proof}
  By Lemma~\ref{lem:9}, $g$ is $\kmod^Q$-linear just when, for each
  $X \in \A$ there is a factorisation in $\kmod$ of the form
  \begin{equation*}
    \cd{
      Q\A(X,A) \ar[r]^-{g \otimes 1} \ar@{->>}[d]_-{\varepsilon} & \A(A,B) \otimes
      \A(X,A) \ar[d]^-{m} \\ \A(X,A) \ar@{-->}[r]^-{g_X} & \A(X,B)\rlap{ .}
    }
  \end{equation*}
  Evaluating both ways around the diagram at $\spn{f_0} \in Q\A(X,A)$, we
  must have $g_X(f_0) = g^{(0)}(f_0)$; evaluating at $\spn{f_0, f_1}$,
  we must have $g_X(f_1) = g^{(1)}(f_0, f_1)$; and evaluating at
  $\spn{f_0, \dots, f_n}$ for $n \geqslant 2$, we must have $0 =
  g^{(n)}(f_0, \dots, f_n)$. This shows the necessity
  of~\eqref{eq:27}; the sufficiency follows on noting that $g_X$
  defined thus is $k$-linear, since $g_X(f) =
  g^{(1)}(0,f)$ and $g^{(1)}$ is $k$-linear in its second argument.
\end{proof}

We are now ready to prove the first main result of the paper.

\begin{Thm}
  \label{thm:3}
  To give a cartesian differential category is equally to give a
  cartesian $\kmod^Q$-category; under this correspondence, the
  $\mathrm{D}$-linear maps correspond to the $\kmod^Q$-linear ones.
\end{Thm}
\begin{proof}
  Consider first a cartesian differential category $\A$, presented as
  in Corollary~\ref{cor:2}. The corresponding $\kmod^Q$-category
  $\underline \A$
  will have the same objects, hom-modules, and identity maps; while its
  composition functions~\eqref{eq:19} are defined using the higher-order
  derivatives and composition in $\A$. We obtain the various axioms of
  Proposition~\ref{prop:8} as follows:
  \begin{enumerate}[(i)]
  \item This follows from axioms (i) and (ii) of Corollary~\ref{cor:2};
  \item This follows from the category axioms for $\A$;
  \item This follows from axiom (iv) of Corollary~\ref{cor:2};
  \item This follows from the combination of axioms (v) and (vii) of Corollary~\ref{cor:2}.
  \end{enumerate}
  To see that $\underline \A$ is cartesian, note that its underlying
  ordinary category---which is $\A$---admits finite products, and that
  the product projection maps $\pi_i$ are $\kmod^Q$-linear by axiom
  (iii) of Corollary~\ref{cor:2}.

  Suppose conversely that $\underline \A$ is a cartesian
  $\kmod^Q$-category. By changing base along the unique monoidal
  comonad morphism $\gamma \colon K \rightarrow Q$, we see that
  $\underline \A$ has an underlying left-$k$-linear category $\A$,
  with the same objects, hom-objects and identities, and with
  composition law $g \circ f = g^{(0)}(f)$. Moreover, the finite
  products of $\underline \A$ yield finite products in $\A$, which is
  thus \emph{cartesian} left-$k$-linear.

  For each $n \geqslant 0$ and $A,B \in \A$, we now define the
  higher-order derivative
  \begin{equation}
    \label{eq:29}
    \begin{aligned}
      (\thg)^{(n)} \colon \A(A,B) & \rightarrow \A(A \times A^n, B) \\
      f & \mapsto f^{(n)}(\pi_0, \pi_1, \dots, \pi_n)
    \end{aligned}
  \end{equation}
  on $\A$. We claim that these satisfy the axioms of
  Corollary~\ref{cor:2}, so yielding a cartesian differential
  structure on $\A$. Observe first that, given
  $f_0, \dots, f_n \colon A \rightarrow B$ and
  $g \colon B \rightarrow C$, we can form
  $f = (f_0, \dots, f_n) \colon A \rightarrow B \times B^n$ and now
  have 
  \begin{equation}
    \label{eq:28}
    \begin{aligned}
      g^{(n)} \circ f &= (g^{(n)}(\pi_0, \pi_1, \dots, \pi_n))^{(0)}(f)
      \\ &= g^{(n)}(\pi_0^{(0)}(f), \dots, \pi_n^{(0)}(f)) =
      g^{(n)}(f_0, \dots, f_n)\rlap{ ,}
    \end{aligned}
  \end{equation}
  recapturing the composition operation in $\underline \A$. In
  verifying the axioms of Corollary~\ref{cor:2}, we first translate
  them under~\eqref{eq:28} to axioms on the composition
  law~\eqref{eq:19} in $\underline{\A}$. On doing so, we find that:
  \begin{enumerate}[(i)]
  \item[(i)--(ii)] follow from condition (i) of
    Proposition~\ref{prop:8};
  \addtocounter{enumi}{2}
  \item follows from the $\kmod^Q$-linearity of product projections in
  $\underline\A$;
\item follows from axiom (iii) of
  Proposition~\ref{prop:8};
\item follows from axiom (iv) of Proposition~\ref{prop:8} on taking $m
  = 0$;
\item follows from axiom (ii) of Proposition~\ref{prop:8};
\item  follows from axiom (iv) of Proposition~\ref{prop:8} on
  taking $(g_0, \dots, g_m) = (\pi_0, \dots, \pi_m)$.
\end{enumerate}

We have thus shown that there are assignments
$\A \mapsto \underline \A$ and $\underline \A \mapsto \A$ between
cartesian differential categories and cartesian $\kmod^Q$-categories,
and it is now easy to see that these are mutually inverse: in one
direction, we use~\eqref{eq:28}, and in the other, the fact that
$(\pi_0, \dots, \pi_n) = \mathrm{id} \colon A \times A^n \rightarrow A
\times A^n$.

Finally, Proposition~\ref{prop:9} shows that,
under the correspondence just given, the $\kmod^Q$-linear maps correspond
to the $\mathrm{D}$-linear ones.
\end{proof}

It is probably worth recording in as concrete a form as possible the
two directions of our main correspondence. On the one hand, for a
cartesian differential category $\A$, the corresponding
$\kmod^Q$-category has the same objects, the same hom-$k$-modules, and
the same identities; and has composition laws
\begin{align*}
  \A(B,C) \otimes Q\A(A,B) & \rightarrow \A(A,C) \\
  g \otimes \spn{f_0, \dots, f_n} & \mapsto g^{(n)}(f_0, \dots, f_n)\rlap{ ,}
\end{align*}
where $g^{(n)}$ denotes the $n$th derivative of
Definition~\ref{def:12}. On the other hand, for a cartesian
$\kmod^Q$-category $\A$, the corresponding cartesian differential
category has the same objects, the same hom-$k$-modules and the same
identities; while its composition and its differential
operator are defined from the $\kmod^Q$-composition maps $m \colon \A(B, C) \otimes
Q\A(A,B) \rightarrow \A(A,C)$ via
\begin{equation*}
  g \circ f = m(g \otimes \spn{f}) \qquad \text{and} \qquad
  \mathrm{D}f = m(f \otimes \spn{\pi_0, \pi_1})\rlap{ .}
\end{equation*}

\section{Presheaves over a skew monoidal base}
\label{sec:second-interl-enrich}

In the rest of the paper, we will make use of our first main theorem
to prove our second one: that every small cartesian differential
category has a full structure-preserving embedding into one induced by
a monoidal differential modality on a symmetric monoidal closed
$k$-linear category. The embedding in question will be the Yoneda
embedding into an enriched presheaf category, and so in this section
we develop the appropriate notions in the skew context. We note that
the definitions in this section draw largely
on~\cite{Street2013Skew-closed} with some novelties (the notions of
\emph{tight} $\V$-functor and the identification of the $\V$-linear
presheaf maps); while the lemmas and propositions are all new.

\subsection{Enriched functors}
\label{sec:enriched-functors}
So far, we have discussed categories enriched over a skew monoidal
category in isolation. We will now need to discuss also
\emph{functors} between enriched categories. The obvious definition is
the following one:
\begin{Defn}
  \label{def:30}
  Let $\V$ be a skew monoidal category, and let $\A, \B$ be
  $\V$-categories. A \emph{$\V$-functor} $F \colon \A \rightarrow \B$
  comprises an assignment $A \mapsto FA$ from objects of $\A$ to
  those of $\B$, together with maps $F_{A,B} \colon \A(A,B)
  \rightarrow \B(FA,FB)$ in $\V$ for all $A,B \in \A$, which render
  commutative all  diagrams of the following form:
  \begin{equation*}
    \cd{
      I \ar[r]^-{i} \ar@{=}[d]_-{}& \A(A,B) \ar[d]_-{F_{AB}} &
      {\A(B,C) \otimes \A(A,B)} \ar[r]^-{m} \ar[d]_{F_{BC} \otimes F_{AB}} &
      {\A(A,C)} \ar[d]^{F_{AC}} \\
      I \ar[r]^-{i} & \B(FA,FB) &
      {\B(FB,FC) \otimes \B(FA,FB)} \ar[r]^-{m} &
      {\B(FA,FC)}\rlap{ .}
    }
  \end{equation*}
\end{Defn}
However, for various purposes, this definition is insufficient. For
example, we would like to say that a $\V$-functor preserves finite
products just when it sends product cones to product cones. However, a
$\V$-functor as defined above does not even induce a mapping on product
cones, because it has no action on $\V$-linear maps. This motivates:
\begin{Defn}
  \label{def:31}
  Let $\V$ be a skew monoidal category, and let $\A, \B$ be
  $\V$-categories. A \emph{tight $\V$-functor} $F \colon \A
  \rightarrow_t \B$ is a $\V$-functor $F \colon \A \rightarrow \B$
  together with an ordinary functor $F_\ell \colon \A_\ell \rightarrow
  \B_\ell$ on categories of linear maps, such that for each $f \colon
  B \rightarrow_\ell C$ in $\A_\ell$ and $A \in \A$, the
  following diagram commutes:
  \begin{equation*}
    \cd{
      {\A(A,B)} \ar[r]^-{F_{AB}} \ar[d]_{f_A} &
      {\B(FA,FB)} \ar[d]^{(F_\ell f)_{FA}} \\
      {\A(A,C)} \ar[r]_-{F_{AC}} &
      {\B(FA,FC)}\rlap{ .}
    }
  \end{equation*}
\end{Defn}
\begin{Rk}
  \label{rk:1}
  This definition is based on ideas
  of~\cite{Campbell2018Skew-enriched}. If $\V$ is a skew monoidal
  category, then Campbell in \emph{loc.\,cit.}~defines a
  \emph{skew-enriched $\V$-category} to comprise an ordinary category
  $\A_\ell$, together with a functor
  $\A(\thg, \thg) \colon \A_\ell^\mathrm{op} \times \A_\ell
  \rightarrow \V$ and maps $I \rightarrow \A(A,A)$ and
  $\A(B,C) \otimes \A(A,B) \rightarrow \A(A,C)$ which are
  natural in $A,B, C \in \A_\ell$, and satisfy the axioms of
  Definition~\ref{def:2}. If $\A$ and $\B$ are $\V$-categories in our sense,
  then they become skew-enriched $\V$-categories in Campbell's sense on
  taking $\A_\ell$ and $\B_\ell$ to be the categories of $\V$-linear maps; on doing
  so, the skew-enriched $\V$-functors between them, in Campbell's
  sense, are precisely our tight $\V$-functors.
\end{Rk}
While tightness is, in general, extra structure on a $\V$-functor, in
the left covering case which is of primary interest to us, it is a
mere property.

\begin{Lemma}
  \label{lem:2}
  Let $\V$ be a left covering skew monoidal category, and let $\A,\B$
  be $\V$-categories. To give a tight $\V$-functor $F \colon \A
  \rightarrow_t \B$ is equally to give a $\V$-functor $F \colon \A
  \rightarrow \B$ whose underlying functor $F_0 \colon \A_0
  \rightarrow \B_0$ preserves $\V$-linear maps.
\end{Lemma}
\begin{proof}
  A straightforward argument using Lemma~\ref{lem:9}.
\end{proof}

\begin{Ex}
  \label{ex:21}
  Let $\A$ and $\B$ be left-$k$-linear categories. A $\kmod^K$-functor
  between them is an ordinary functor $F \colon \A \rightarrow \B$
  whose action on homs preserves the $k$-module structure. Such a
  functor is tight just when it preserves $k$-linear maps.
\end{Ex}

Using the notion of tightness, we can now describe what it means for a
functor to preserve cartesian structure.
\begin{Defn}
  \label{def:32}
  Let $\V$ be a skew monoidal category, and let $\A, \B$ be cartesian
  $\V$-categories. A tight $\V$-functor $F \colon \A \rightarrow_t \B$
  is said to be \emph{cartesian} if it sends terminal objects to
  terminal objects, and $F_\ell$ sends binary product cones
  $A \leftarrow_\ell A \times B \rightarrow_\ell B$ to binary product
  cones.
\end{Defn}

\begin{Ex}
  \label{ex:23}
  Let $\A$ and $\B$ be cartesian differential categories, seen as
  cartesian $\kmod^Q$-categories. It is a straightforward calculation
  to see that a $\kmod^Q$-functor from $\A$ to $\B$ is an ordinary
  functor $F \colon \A \rightarrow \B$ which preserves addition on the
  homs, and for which each diagram of the following form commutes in $\B$:
  \begin{equation*}
    \cd{
      F(A \times A) \ar[d]_-{(F\pi_1, F\pi_2)}
      \ar[r]^-{F(\mathrm{D}f)} &
      FB\rlap{ .} \\
      FA \times FA \ar[ur]_-{\mathrm{D}(Ff)} 
    }
  \end{equation*}
  Such a $\kmod^Q$-functor is tight precisely when it preserves
  $\mathrm{D}$-linearity of maps; and it is cartesian when, in
  addition, it preserves finite products in the usual sense. In fact,
  preservation of $\mathrm{D}$-linearity \emph{automatically} implies
  preservation of finite products by the argument
  of~\cite[Lemma~1.3.2]{Blute2009Cartesian} (so in this sense finite
  products in $\kmod^Q$-categories are \emph{absolute} limits).
\end{Ex}

\subsection{Enriched presheaves}
\label{sec:presh-skew-mono}

We now describe the notion of \emph{presheaf} on a category enriched
over a skew monoidal category.

\begin{Defn}
  \label{def:17}(\cite[\sec 5]{Street2013Skew-closed})
  Let $\V$ be a skew monoidal category and $\A$ a $\V$-enriched
  category. A \emph{presheaf} $X$ on $\A$ is given by objects $XA \in
  \V$ for each $A \in \A$, and maps
  $m \colon XB \otimes \A(A,B) \rightarrow XA$ for each $A,B \in \A$, rendering
 commutative each~diagram
  \begin{equation*}
    \cd[@-1em@C-1em]{
      \sh{l}{0.2em}{\bigl(XC\otimes\A(B,C)\bigr)\otimes\A(A,B)}
      \ar[r]^-{m \otimes 1} \ar[dd]_{\alpha} &
      \sh{r}{1em}{XB\otimes\A(A,B)} \ar[dr]^-{m} \\
      & & XA\\
      \sh{l}{0.2em}{XC\otimes\bigl(\A(B,C)\otimes\A(A,B)\bigr)} \ar[r]^-{1
        \otimes m} &
      \sh{r}{1em}{XC\otimes\A(A,C)} \ar[ur]^-{m}
    } \ \ \      \cd[@!C@C-2.8em]{
      {XA\otimes I} \ar[r]^-{1 \otimes i} \ar@{<-}[d]_-{\rho} & 
      {XA\otimes\A(A,A)} \ar[d]^-{m} \\ 
      {XA} \ar[r]^-{1} & {XA}\rlap{ .}
    }
  \end{equation*}
  If $X,Y$ are presheaves on $\A$, then a \emph{$\V$-linear presheaf
    map} $f \colon X \rightarrow Y$ comprises families of maps
  $f_A \colon XA \rightarrow YA$ which commute with the $\A$-action,
  in the sense of rendering commutative each square
  \begin{equation*}
    \cd{
      {XB \otimes \A(A,B)} \ar[r]^-{m} \ar[d]_{f_B \otimes 1} &
      {XA} \ar[d]^{f_A} \\
      {YB \otimes \A(A,B)} \ar[r]_-{m} &
      {YA}\rlap{ .}
    }
  \end{equation*}
  We write $\cat{Psh}_{\ell}(\A)$ for the category of presheaves on
  $\A$ and $\V$-linear presheaf maps.
\end{Defn}

\begin{Ex}
  \label{ex:17}
  If $\A$ is a $\V$-category, then for each $A \in \A$, we have a
  presheaf $yA \in \cat{Psh}_\ell(\A)$ with components
  $yA(B) = \A(B,A)$ and action given by composition in $\A$. Moreover,
  $\V$-linear presheaf maps
  $yA \rightarrow yB$ are \emph{precisely}
  $\V$-linear maps $A \rightarrow_\ell B$ in $\A$.
\end{Ex}

\begin{Ex}
  \label{ex:19}
  Let $\V$ be a monoidal category which supports the initial monoidal
  comonad $K$, and let $\A$ be a $\V^K$-enriched category. In this
  case, $\cat{Psh}_\ell(\A)$ is the ordinary functor category
  $[\A_0^\mathrm{op}, \V]$. In particular, if $\A$ is a
  left-$k$-linear category, then $\cat{Psh}_\ell(\A)$ is
  $[\A_0^\mathrm{op}, \kmod]$---so justifying the notation of
  Definition~\ref{def:14} above.
\end{Ex}

\subsection{Enriched presheaves as an enriched category}
\label{sec:v-categ-presh}

Example~\ref{ex:17} offers a partial justification of the nomenclature
``$\V$-linear'' for the maps of $\cat{Psh}_\ell(\A)$. We will now justify
it more fully by showing that such maps are in fact the $\V$-linear
maps of a $\V$-category of presheaves $\cat{Psh}(\A)$. Towards this,
we make the following key definition.




\begin{Defn}
  \label{def:20}
  Let $\V$ be a skew monoidal category and $\A$ a $\V$-category.
  Given a presheaf $X$ on $\A$ and $V \in \V$, we define $V \ast X$
  to be the presheaf with components $(V \ast X)(A) = V
  \otimes XA$ and action maps
  \begin{equation*}
    (V \otimes XB) \otimes \A(A,B) \xrightarrow{\alpha} V \otimes \bigl(XB
    \otimes \A(A,B)\bigr) \xrightarrow{1 \otimes m} V \otimes XA\rlap{ .}
  \end{equation*}
  This construction is easily seen to underlie a functor $\ast \colon
  \V \times \cat{Psh}_\ell(\A) \rightarrow \cat{Psh}_\ell(\A)$.
\end{Defn}
\begin{Defn}
  \label{def:29}
  If $X,Y$ are presheaves on $\A$, then a \emph{presheaf hom} from $X$
  to $Y$ is a representation for the functor
  $\cat{Psh}_\ell(\A)(\thg \ast X, Y) \colon \V^\mathrm{op}
  \rightarrow \cat{Set}$, comprising an object $\phom{X,Y} \in \V$ and
  hom-set isomorphisms
  \begin{equation*}
    \cat{Psh}_\ell(\A)(V \ast X, Y) \cong \V(V, \dbr{X,Y})\rlap{ ,}
  \end{equation*}
  natural in $V \in \V$. We write
  $\mathsf{ev} \colon \phom{X,Y} \ast X \rightarrow Y$ for the counit
  of this representation.
\end{Defn}

\begin{Ex}
  \label{ex:12}(Yoneda lemma). For any presheaf $X$ on $\A$ and any
  $A \in \A$, there is a $\V$-linear map
  $m \colon XA \ast yA \rightarrow X$ whose components are
  given by the action of $\A$ on $X$. This map exhibits $XA$ as
  $\phom{yA, X}$; for indeed, if
  $\gamma \colon V \ast yA \rightarrow X$ is any other $\V$-linear
  presheaf map, then the composite
\begin{equation*}
  V \xrightarrow{\rho} V \otimes I \xrightarrow{1 \otimes i} V \otimes
  \A(A,A) \xrightarrow{\gamma} XA
\end{equation*}
is the unique factorisation of $\gamma$ through $m$.
\end{Ex}
We will be interested in the situation where \emph{all} presheaf homs
$\phom{X,Y}$ on a given $\A$ exist. This will certainly be the case if
$\A$ is small, and $\V$ is left closed and complete; see~\cite[\sec
5]{Street2013Skew-closed} for the construction. Note that in this
case, the assignment $X, Y \mapsto \phom{X,Y}$ becomes a functor
$\cat{Psh}_\ell(\A)^\mathrm{op} \times \cat{Psh}_\ell(\A) \rightarrow
\V$ such that the bijections $\cat{Psh}_\ell(\A)(V \ast X, Y)
\cong \V(V, \phom{X,Y})$ are natural in $X$ and $Y$ as well as $V$.

\begin{Defn}
  \label{def:21}
  Let $\V$ be a skew monoidal category and $\A$ a $\V$-category. Suppose
  for all presheaves $X,Y$ on $\A$ that the presheaf hom $\phom{X,Y}$
  exists. We define the $\V$-category $\cat{Psh}(\A)$ to have
  presheaves on $\A$ as objects, and hom-objects the $\phom{X,Y}$'s.
  The identities $I \rightarrow \phom{X,X}$ are induced by universality of $\phom{X,X}$
    applied to the $\V$-linear presheaf map $I \ast X \rightarrow X$ with components
    \begin{equation}\label{eq:31}
      \lambda_{XA} \colon I \otimes XA \rightarrow XA\rlap{ ;}
    \end{equation}
  while composition $\phom{Y,Z} \otimes \phom{X,Y} \rightarrow
    \phom{X,Z}$ is induced by universality of $\phom{X,Z}$ applied to
    the $\V$-linear presheaf map $(\phom{Y,Z} \otimes \phom{X,Y}) \ast X \rightarrow
    Z$ with components
    \begin{equation*}
      (\phom{Y,Z} \otimes \phom{X,Y}) \otimes XA \xrightarrow{\alpha}
      \phom{Y,Z} \otimes (\phom{X,Y} \otimes XA) \xrightarrow{1 \otimes \mathsf{ev}}
      \phom{Y,Z} \otimes YA \xrightarrow{\mathsf{ev}} ZA\rlap{ .}
    \end{equation*}
\end{Defn}

\begin{Ex}
  \label{ex:13}
  If $\A$ is a left-$k$-linear category, then $\cat{Psh}(\A)$ is the
  left-$k$-linear category whose objects are functors
  $\A_0^\mathrm{op} \rightarrow \kmod$, and whose morphisms
  $X \rightarrow Y$ are families of functions (not necessarily
  $k$-linear) $XA \rightarrow YA$ satisfying the usual naturality
  condition. The $k$-module structure on the homs is given
  componentwise.
\end{Ex}

We now justify the name ``$\V$-linear'' for the morphisms of
$\cat{Psh}_\ell(\A)$.

\begin{Prop}
  \label{prop:10}
  Let $\A$ be a $\V$-category for which the functor $\V$-category
  $\cat{Psh}(\A)$ exists. The $\V$-linear maps in $\cat{Psh}(\A)$
  correspond bijectively with maps of $\cat{Psh}_\ell(\A)$.
\end{Prop}
\begin{proof}
  Let $X,Y$ be presheaves on $\A$. A $\V$-linear map
  $f \colon X \rightarrow_\ell Y$ of $\cat{Psh}(\A)$ comprises a
  family of maps $f_Z \colon \phom{Z,X} \rightarrow \phom{Z,Y}$ in $\V$
  satisfying the $\V$-naturality condition of Definition~\ref{def:26}.
  When $Z = yA$, we may by the Yoneda lemma take it that $\phom{yA, X}
  = XA$ and $\phom{yA,Y} = YA$, so that $f_{yA} \colon XA \rightarrow
  YA$ for each $A \in \A$. We may likewise take it that $\phom{yB, yA}
  = \A(B,A)$ so that instantiating~\eqref{eq:30} as to the left below
  yields the commuting diagram as to the right
  \begin{equation*}
      \cd[@C-0.6em]{
      \phom{yA, X} \otimes \phom{yB,yA} \ar[d]_-{m} \ar[r]^-{f_{yA}
        \otimes 1} & \phom{yA, Y} \otimes \phom{yB,yA}
      \ar[d]^-{m} &
      XA \otimes \A(B,A) \ar[d]_-{m} \ar[r]^-{f_{yA}
        \otimes 1} & YA \otimes \A(B,A)
      \ar[d]^-{m} 
      \\
      \phom{yB, X} \ar[r]^-{f_{yB}} & \phom{yB,Y} &
      XB \ar[r]^-{f_{yB}} & XA\rlap{ .}
    }
  \end{equation*}
  In this way, we obtain a map $f_{y\thg} \colon X \rightarrow Y$
  of $\cat{Psh}_\ell(\A)$.

  On the other hand, given a map $g \colon X \rightarrow Y$ of
  $\cat{Psh}_\ell(\A)$, we can for each $Z \in \cat{Psh}(\A)$ apply
  the functor $\phom{Z, \thg} \colon \cat{Psh}_\ell \rightarrow \V$ to
  obtain a map $g_Z \colon \phom{Z,X} \rightarrow \phom{Z,Y}$ in $\V$.
  By construction, $g_Z$ is unique such that the following square
  commutes:
  \begin{equation*}
    \cd{
      {\phom{Z,X} \ast Z} \ar[r]^-{g_Z \ast 1} \ar[d]_{\mathsf{ev}} &
      {\phom{Z,Y} \ast Z} \ar[d]^{\mathsf{ev}} \\
      {X} \ar[r]_-{g} &
      {Y}\rlap{ ,}
    }
  \end{equation*}
  and it is easy to prove from this that the family
  of maps $g_Z$ constitute a $\V$-linear map $\tilde g \colon X
  \rightarrow_\ell Y$ of $\cat{Psh}(\A)$.
  Finally, a standard Yoneda argument shows these two processes $f
  \mapsto f_{y\thg}$ and $g \mapsto \tilde g$ to be mutually inverse.
\end{proof}

The preceding result identifies the category of $\V$-linear maps in
$\cat{Psh}(\A)$; the next one identifies the underlying ordinary
category $\cat{Psh}_0(\A)$. In the case where $\V$ is genuinely
monoidal, these two categories coincide, but not in general.

\begin{Prop}
  \label{prop:11}
  Let $\A$ be a $\V$-enriched category for which the presheaf category
  $\cat{Psh}(\A)$ exists. The underlying ordinary category
  $\cat{Psh}_0(\A)$ is isomorphic to the co-Kleisli category of the comonad on
  $\cat{Psh}_\ell(\A)$ defined as follows:
  \begin{itemize}
  \item The underlying endofunctor is $I \ast (\thg) \colon
    \cat{Psh}_\ell(\A) \rightarrow \cat{Psh}_\ell(\A)$.
  \item The counit at $X$ is the map $I \ast X \rightarrow_\ell
    X$ with components~\eqref{eq:31};
  \item The comultiplication at $X$ is the map $I \ast X \rightarrow_\ell I
    \ast (I \ast X)$ with components
    \begin{equation*}
      I \otimes XA \xrightarrow{\rho_I \otimes 1} (I \otimes I)
      \otimes XA \xrightarrow{\alpha} I \otimes (I \otimes XA)\rlap{ .}
    \end{equation*}
  \end{itemize}
\end{Prop}
\begin{proof}
  Clearly, objects of $\cat{Psh}_0(\A)$ are presheaves on $\A$. Morphisms
  $X \rightarrow Y$ are, by definition, maps $f \colon I \rightarrow \phom{X,Y}$
  in $\V$, which by the universal property of $\phom{X,Y}$ are equally
  maps $\bar f \colon I \ast X \rightarrow Y$ in $\cat{Psh}_\ell(\A)$. The
  identity map at $X$ is $i \colon I \rightarrow \phom{X,X}$,
  which by definition corresponds to the map $I \ast Xf
  \rightarrow X$ with components~\eqref{eq:31}. Finally,
  composition of $\bar{f} \colon I \ast X \rightarrow Y$ and
  $\bar g \colon
  I \ast Y \rightarrow Z$ proceeds via the formula~\eqref{eq:32},
  which transposes to become the composite
  \begin{equation*}
    I \ast X \xrightarrow{\rho_I \ast 1} (I \otimes I) \ast X
    \xrightarrow{(f \otimes g) \ast 1} (\phom{Y,Z} \otimes \phom{X,Y})
    \ast X \xrightarrow{m \ast 1} \phom{X,Z} \ast X
    \xrightarrow{\mathsf{ev}} Z
  \end{equation*}
  in $\cat{Psh}_\ell(\A)$. By the definition of $m$ in
  $\cat{Psh}(\A)$, this is equally the composite
  \begin{equation*}
    I \ast X \xrightarrow{\rho_I \ast 1} (I \otimes I) \ast X
    \xrightarrow{\alpha} I \ast (I \ast X) \xrightarrow{f \ast (g \ast
      1)} \phom{Y,Z} \ast (\phom{X,Y} \ast
     X) \xrightarrow{\mathsf{ev}} \phom{Y,Z} \ast Y
    \xrightarrow{\mathsf{ev}} Z
  \end{equation*}
  which is, in turn, the composite
  \begin{equation*}
    I \ast X \xrightarrow{\rho_I \ast 1} (I \otimes I) \ast X
    \xrightarrow{\alpha} I \ast (I \ast X) \xrightarrow{1 \ast \bar g}
    I \ast Y \xrightarrow{\bar f} Z\rlap{ ,}
  \end{equation*}
  which is precisely the co-Kleisli composite of $\bar f$ and $\bar
  g$, as desired.
\end{proof}
\begin{Ex}
  \label{ex:20}
  Let $\A$ be a left-$k$-linear category. We saw in
  Example~\ref{ex:13} that $\cat{Psh}_0(\A)$ is the category whose
  objects are functors $\A_0^\mathrm{op} \rightarrow \kmod$ and whose
  maps are natural transformations with not-necessarily-linear
  components. By the preceding result, we can identify this
  with the co-Kleisli category of the comonad $[1, K]$ on
  $[\A_0^\mathrm{op}, \kmod]$ (i.e., the comonad which applies $K$ on
  $\kmod$ pointwise).
\end{Ex}

\subsection{The Yoneda embedding}
\label{sec:yoneda-embedding}
As in classical enriched category theory, we have a Yoneda embedding
into the category of presheaves.

\begin{Defn}
  \label{def:33}
  Let $\V$ be a skew monoidal category, and $\A$ a $\V$-category for
  which $\cat{Psh}(\A)$ exists. The
  \emph{Yoneda embedding} $y \colon \A \rightarrow \cat{Psh}(\A)$ is
  the $\V$-functor which on objects, sends $A$ to $yA$; and acts on
  homs via the maps $\A(B,C) \rightarrow \phom{yB, yC}$ induced by
  universality of $\phom{yB,yC}$ applied to the presheaf map
  $\A(B,C) \ast yB \rightarrow yC$ whose components are given by
  composition in $\A$.
\end{Defn}

By Example~\ref{ex:17}, $\V$-linear maps $A \rightarrow_\ell B$ in
$\A$ are the same as maps $yA \rightarrow yB$ of $\cat{Psh}_\ell(\A)$;
and by Proposition~\ref{prop:10}, these correspond bijectively with
$\V$-linear maps in $\cat{Psh}(\A)$. By way of this assignment on
linear maps, we can make the Yoneda embedding into a \emph{tight}
$\V$-functor $\A \rightarrow_t \cat{Psh}(\A)$. The next result shows
that this behaves as expected with respect to cartesian structure.

\begin{Lemma}
  \label{lem:14}
  Let $\V$ be skew monoidal with finite products, and let $\A$ be
  a $\V$-category for which $\cat{Psh}(\A)$ exists. $\cat{Psh}(\A)$ is
  cartesian, and the  Yoneda embedding
  $\A \rightarrow_t \cat{Psh}(\A)$ is cartesian and fully faithful
  (i.e., an isomorphism on homs).
\end{Lemma}
\begin{proof}
  First note that $\cat{Psh}_\ell(\A)$ has finite products. The
  terminal presheaf of $\cat{Psh}_\ell(\A)$ has all its components
  terminal in $\V$, and the unique possible action maps. The binary
  product of $X,Y \in \cat{Psh}_\ell(\A)$ has components
  $(X \times Y)A = XA \times YA$ and action maps
  $\bigl(m(\pi_0 \otimes 1), m(\pi_1 \otimes 1)\bigr) \colon (XB
  \times YB) \otimes \A(A,B) \rightarrow XA \times YA$, while the
  projection maps $X \leftarrow X \times Y \rightarrow Y$ are given
  componentwise by those in $\V$.

  Now, the span
  $\pi_0 \colon X \leftarrow X \times Y \rightarrow Y \colon \pi_1$ in
  $\cat{Psh}_\ell(\A)$ corresponds under Proposition~\ref{prop:10} to
  a span of $\V$-linear maps
  ${\tilde{\pi}_0 \colon X \leftarrow_\ell X \times Y \rightarrow_\ell
    Y \colon \tilde{\pi}_1}$ of $\cat{Psh}(\A)$, whose components are
  the spans
  $\phom{Z, \pi_0} \colon \phom{Z,X} \leftarrow \phom{Z,X \times Y}
  \rightarrow \phom{Z, Y} \colon \phom{Z, \pi_1}$ in $\V$. Each such
  span is the image of a product span under a right adjoint functor
  $\phom{Z, \thg}$, and so is itself a product. So
  ${\tilde \pi}_0, {\tilde \pi}_1$ exhibit $X \times Y$ as a product
  of $X,Y$ in $\cat{Psh}(\A)$. Similarly, the terminal object of
  $\cat{Psh}_\ell(\A)$ is also a terminal presheaf in $\cat{Psh}(\A)$.

  Now it follows from the Yoneda lemma, Example~\ref{ex:12}, that
  $y \colon \A \rightarrow_t \cat{Psh}(\A)$ is fully faithful. To see
  that it is cartesian, we must show that a product span
  $A \leftarrow_\ell A \times B \rightarrow_\ell B$ in $\V$ induces a
  product span $yA \leftarrow_\ell y(A \times B) \rightarrow_\ell yB$
  in $\cat{Psh}(\A)$. By the preceding part of the argument, it will
  suffice to show that the corresponding span
  $yA \leftarrow y(A \times B) \rightarrow yB$ in $\cat{Psh}_\ell(\A)$
  is limiting: which is so since its components are the product
  diagrams $\A(X, A) \leftarrow \A(X,A \times B) \rightarrow \A(X,B)$
  in $\V$.
\end{proof}

\section{An embedding theorem for cartesian
  differential categories}
\label{sec:appl-an-embedd}

In this section, we prove our second main theorem: that every
small cartesian differential category has a full structure-preserving
embedding into one induced by a differential modality. In fact, we
will do better: our embedding will always be into a cartesian
differential category induced by a \emph{monoidal} differential
modality on a monoidal \emph{closed} category---so that our embedding
is into the cartesian closed
differential category associated to a model of intuitionistic
differential linear logic.


As explained at the start of the previous section, the basic strategy
will be to embed a small cartesian differential category $\A$ into its
$\kmod^Q$-enriched presheaf category $\cat{Psh}(\A)$. The
$\kmod^Q$-enrichment of $\cat{Psh}(\A)$ corresponds to a cartesian
differential structure on the underlying category $\cat{Psh}_0(\A)$,
and by Proposition~\ref{prop:11}, this latter category is the
co-Kleisli category of the comonad $I \ast (\thg)$ on
$\cat{Psh}_\ell(\A)$. We will show that this comonad
underlies a monoidal differential modality on $\cat{Psh}_\ell(\A)$
which induces the cartesian differential structure of
$\cat{Psh}_0(\A)$; this yields our result.

\subsection{Presheaves over a skew-warped base}
\label{sec:skew-warp-presh}

Most of the hard work will be in showing that $\cat{Psh}_\ell(\A)$
bears the appropriate structure: firstly, symmetric monoidal closed
$k$-linear structure, and secondly, a monoidal differential modality
whose underlying comonad agrees with $I \ast (\thg)$. In obtaining
these, it will be convenient to work more generally: thus, for the
rest of this section, we suppose that $\V$ is a symmetric monoidal
category, that $!$ is a symmetric monoidal comonad on $\V$, and that
$\A$ is a $\V^!$-category.

A key observation is that presheaves on the $\V^{!}$-category $\A$ can
be identified with presheaves on an associated $\V$-category, so
allowing us to make use of results from classical enriched category
theory. To obtain this associated $\V$-category, we will change base
along the composite symmetric monoidal functor
\begin{equation}\label{eq:43}
  ! = (\V, \otimes^!, I) \xrightarrow{C} (\cat{Coalg}(!), \hat\otimes,
  \hat I) \xrightarrow{U} (\V, \otimes, I)\rlap{ ,}
\end{equation}
where here $U$ is the strict monoidal forgetful functor, and
$C$ is the cofree functor, made monoidal via the structure
maps
\begin{equation*}
  m_I \colon (I, m_I)
  \rightarrow ({!}I, \delta_{I}) \ \ \text{and} \ \ m_\otimes \defeq H \colon ({!}A,
  \delta_A) \mathbin{\hat \otimes}
  ({!}B, \delta_B) \rightarrow \bigl({!}(A \otimes {!}B),
  \delta_{A \otimes {!}B}\bigr)\rlap{ ;}
\end{equation*}
it is a routine diagram chase with the axioms for a symmetric monoidal
comonad to show that these maps do indeed provide monoidal structure.


Now, base change along the composite functor ${!}$ of~\eqref{eq:43}
associates to the $\V^{!}$-enriched category $\A$ a 
$\V$-enriched category $!_\ast (\A)$, with the same objects as $\A$,
hom-objects ${!}\A(A,B)$, and identities and composition given by the composites
\begin{equation}\label{eq:35}
  i^{\sharp} \defeq I \xrightarrow{m_I} !I \xrightarrow{!i} !\A(A,A) \qquad \text{and}
\end{equation}
\begin{equation}\label{eq:36}
  m^{\sharp} \defeq {!}\A(B,C) \otimes {!}\A(A,B) \xrightarrow{H} 
  {!}(\A(B,C) \otimes {!}\A(A,B)) \xrightarrow{{!}m} {!}\A(A,C)\rlap{ .}
\end{equation}


We now show that the category of $\V^!$-presheaves on $\A$ is equally
well the category of $\V$-presheaves on $!_\ast (\A)$. This will allow
us to study $\V^!$-presheaves via the classical theory of presheaves
over a symmetric monoidal enrichment base.
\begin{Lemma}
  \label{lem:19}
  There is an equality of categories $\cat{Psh}_\ell(\A) = 
  \cat{Psh}_\ell(!_\ast (\A))$.
\end{Lemma}
\begin{proof}
  The basic data of an $\A$-presheaf and a $!_\ast\A$-presheaf are the
  same: a family of components $XA \in \V$ and action maps
  $m \colon XB \otimes {!} \A(A,B) \rightarrow XA$. Moreover, by
  unfolding the definitions~\eqref{eq:35} and~\eqref{eq:36}, we see
  that the axioms coincide, too.
\end{proof}

\subsection{Lifting modalities to presheaves}
\label{sec:lift-mono-coalg}

We now exploit the preceding lemma to show that various kinds of
structure can be lifted from $\V$ to $\cat{Psh}_\ell(\A)$. Note first
that $\V^{\mathrm{ob}\,\A}$, the cartesian product of $\ob\, \A$
copies of $\V$, bears all the same structure as $\V$ does, defined
pointwise; over the next three propositions, we will show that this
structure can be lifted along the obvious forgetful functor
$\cat{Psh}_\ell(\A) \rightarrow \V^{\mathrm{ob}\,\A}$. In each case,
we will appeal to general results which allow us to do this without
needing to check coherence by hand; however, after the fact, we also
give an explicit description of the lifted structure so obtained.


\begin{Prop}
\label{prop:14}
If ${!}$ is a monoidal coalgebra modality, then the pointwise
symmetric monoidal structure of $\V^{\mathrm{ob}\,\A}$ lifts to a
symmetric monoidal structure on $\cat{Psh}_\ell(\A)$.  If $\V$ is
complete and monoidal closed, and $\A$ is small, then the monoidal
structure on $\cat{Psh}_\ell(\A)$ is closed; while 
if $\V$ is
symmetric monoidal $k$-linear, then so is $\cat{Psh}_\ell(\A)$.
\end{Prop}
\begin{proof}
  Since $\cat{Psh}_\ell(\A) = \cat{Psh}_\ell(!_\ast(\A))$, it suffices
  to prove the claim for the latter category. Since ${!}$ is a
  monoidal coalgebra modality, each hom $!\A(A,B)$ of the
  $\V$-category $!_\ast(\A)$ is a cocommutative comonoid, and the
  identity and composition maps~\eqref{eq:35} and~\eqref{eq:36} are maps of
  cocommutative comonoids. As explained in~\cite[\sec
  5]{Day1970On-closed}, this implies that the pointwise monoidal
  structure on $\V^{\mathrm{ob}\,\A}$ lifts to
  $\cat{Psh}_\ell(!_\ast(\A))$; and that, under the stated further
  hypotheses, this lifted monoidal structure is \emph{closed}.

  Suppose now that $\V$ is symmetric monoidal $k$-linear. In this
  case, $\cat{Psh}_\ell(\A)$ becomes $k$-linear on endowing each
  hom-set with the pointwise $k$-module structure; note that this
  structure is preserved by pre- and post-composition because it is so
  in $\V$. Moreover, the tensor product of $\cat{Psh}_\ell(\A)$
  preserves in each variable the $k$-module structure on the homs
  because the same is true in $\V$.
\end{proof}
\begin{Rk}
  \label{rk:3}
  We can be quite explicit about the monoidal structure on
  $\cat{Psh}_\ell(\A)$. If $X, Y$ are presheaves on $\A$, then their
  componentwise tensor $XA \otimes YA$ becomes a presheaf via the
  structure maps
  \begin{equation}
    \label{eq:22}
    \begin{gathered}
      (XB \otimes YB) \otimes {!}\A(A,B) \xrightarrow{1 \otimes
        \Delta}
      (XB \otimes YB) \otimes \bigl({!}\A(A,B) \otimes {!}\A(A,B)\bigr)\\
      \xrightarrow{\cong} (XB \otimes {!}\A(A,B)) \otimes (YB \otimes
      {!}\A(A,B)) \xrightarrow{m \otimes m} XA \otimes YA\rlap{ ,}
    \end{gathered}
  \end{equation}
  where the unnamed isomorphism uses the associativity and symmetry
  maps in $\V$.
  The unit for this tensor is the presheaf constant at
  $I$, with structure maps
  \begin{equation}
    I \otimes {!}\A(A,B) \xrightarrow{1 \otimes e} I \otimes I
    \xrightarrow{\rho} I\rlap{ .}\label{eq:24}
  \end{equation}
  As for the internal hom of presheaves
  $[Y,Z]$, this has components given by the following hom-objects of the
  $\V$-category $\cat{Psh}(!_\ast(\A))$:
  \begin{equation*}
    [Y,Z]A = \cat{Psh}({!}_\ast(\A))({!}\A(\thg, A) \otimes Y, Z)
  \end{equation*}
  where ${!}\A(\thg, A)$ is the representable presheaf on
  $A \in {!_\ast}(\A)$ and the $\otimes$ is the tensor product of
  presheaves just defined. Recognising the right-hand side as the
  $\V$-presheaf hom $\phom{{!}\A(\thg, A) \otimes Y, Z}$,
  we obtain the structure map
  ${[Y,Z]B \otimes {!}\A(A,B) \rightarrow [Y,Z]A}$ by
  transposing the $\V$-linear presheaf map 
  $([Y,Z]B \otimes {!}\A(A,B)) \ast ({!}\A(\thg, A) \otimes Y)
  \rightarrow Z$ with $C$-component
  \begin{equation}\label{eq:44}
    [Y,Z]B \otimes {!}\A(A,B) \otimes {!}\A(C, A) \otimes
    YC \xrightarrow{1 \otimes m^{\sharp} \otimes 1} [Y,Z]B \otimes {!}\A(C,B) \otimes
    YC \xrightarrow{\mathsf{ev}} ZC\text{ .}
  \end{equation}
\end{Rk}

\begin{Prop}
\label{prop:15}
If ${!}$ is a monoidal coalgebra modality, then the pointwise monoidal
coalgebra modality ${!}^{\mathrm{ob}\,\A}$ on $\V^{\mathrm{ob}\,\A}$ lifts to a monoidal
coalgebra modality on $\cat{Psh}_\ell(\A)$. 
\end{Prop}
\begin{proof}
  We will make use of the following construction. Given a
  monoidal adjunction $F \dashv G \colon \M \rightarrow \N$ between
  (genuine) monoidal categories and a $\N$-category $\C$, we obtain an
  adjunction $\tilde F \dashv \tilde G \colon
  \cat{Psh}_\ell(F_\ast(\C)) \rightarrow \cat{Psh}_\ell(\C)$
  on presheaves as follows. $\tilde F$ acts on a $\C$-presheaf $X$ by applying $F$
  componentwise, and endowing the result with the action maps
  \begin{equation*}
    FXB \otimes F\C(A,B) \xrightarrow{m_\otimes} F(XB \otimes
    \C(A,B)) \xrightarrow{Fm} FXA\rlap{ .}
  \end{equation*}
  On the other hand, $\tilde G$ acts on a $F_\ast(\C)$-presheaf $Y$ by
  applying $G$ componentwise, and endowing the result with the action
  maps
  \begin{equation*}
    GYB \otimes \C(A,B) \xrightarrow{1 \otimes \eta} GYB \otimes
    GF\C(A,B) \xrightarrow{m_\otimes} G(YB \otimes F\C(A,B))
    \xrightarrow{Gm} GYA\rlap{ ,}
  \end{equation*}
  where $\eta$ is the unit of $F \dashv G$.
  
  We now apply this construction to the cofree-forgetful monoidal adjunction
  $U \dashv C \colon (\V, \otimes, I) \rightarrow (\cat{Coalg}(!),
  \otimes, I)$ and the $\cat{Coalg}(!)$-category $C_\ast(\A)$. This
  yields an adjunction on presheaves as to the left below, which lifts
  the pointwise adjunction as to the right:
  \begin{equation}\label{eq:37}
    \cd{
      {\cat{Psh}_\ell(C_\ast(\A))} \ar@<-4.5pt>[r]_-{\tilde U}
      \ar@{<-}@<4.5pt>[r]^-{\tilde C} \ar@{}[r]|-{\top} &
      {\cat{Psh}_\ell(!_\ast(\A))} &
      {\cat{Coalg}(!)^{\mathrm{ob}\,\A}} \ar@<-4.5pt>[r]_-{U^{\mathrm{ob}\,\A}}
      \ar@{<-}@<4.5pt>[r]^-{C^{\mathrm{ob}\,\A}} \ar@{}[r]|-{\top} &
      {\V^{\mathrm{ob}\,\A}} \rlap{ .}
    }
  \end{equation}
  

  Now, we have already seen that the $\V$-category
  $\cat{Psh}_\ell(!_\ast(\A))$ bears a symmetric monoidal structure
  lifting that of $\V^{\mathrm{ob}\,\A}$, since every hom has a
  cocommutative comonoid structure which is respected by composition.
  Since ${!}$ is a monoidal coalgebra modality, the same properties
  hold for the $(\cat{Coalg}(!), \hat \otimes, \hat I)$-category $C_\ast(\A)$,
  and so we can also lift the pointwise monoidal structure of
  $\cat{Coalg}(!)^{\mathrm{ob}\,\A}$ to $\cat{Psh}(C_\ast(\A))$.

  In fact, the monoidal structure of the entire adjunction to the
  right in~\eqref{eq:37} lifts to the adjunction to the left. Indeed,
  the \emph{strict} monoidal structure of
  $U \colon (\cat{Coalg}(!), \hat \otimes, \hat I) \rightarrow (\V, \otimes, I)$
  clearly lifts to
  $\tilde U \colon \cat{Psh}_\ell(C_\ast(\A)) \rightarrow
  \cat{Psh}_\ell({!}_\ast(\A))$, since both monoidal structures lift
  that of $\V^{\mathrm{ob}\,\A}$; and moreover,
  by~\cite[Theorem~2.2]{Kelly1974Doctrinal}, any monoidal adjunction
  is uniquely determined by the underlying adjunction, and a strong
  monoidal structure on the left adjoint.

  It follows that $\tilde U \tilde C$ is a \emph{monoidal} comonad on
  $\cat{Psh}_\ell(!_\ast(\A)) = \cat{Psh}_\ell(\A)$ which lifts the
  pointwise monoidal comonad $!^{\mathrm{ob}\,\A}$ on
  $\V^{\mathrm{ob}\,\A}$. It remains to show this monoidal comonad is
  in fact a monoidal coalgebra modality.
  By~\cite[Theorem~3]{Benton1995A-mixed}, it suffices for this to show
  that the lifted monoidal structure on~$\cat{Psh}_\ell(C_\ast(\A))$
  is \emph{cartesian}---which follows immediately from the fact that
  the monoidal structure of $\cat{Coalg}(!)$ is itself cartesian
  (cf.~\cite[Example~5.2]{Day1970On-closed}). 
\end{proof}

\begin{Rk}
  \label{rk:4}
  Again, we can be quite explicit about the induced monoidal coalgebra
  modality on $\cat{Psh}_\ell(\A)$. The underlying functor acts on a
  presheaf $X \in \cat{Psh}_\ell(\A)$ by applying ${!}$ to all of its
  components, and equipping it with the following action maps, where
  $H$ is the fusion operator of~\eqref{eq:39}:
  \begin{equation}\label{eq:38}
    !XB \otimes {!}\A(A,B)
    \xrightarrow{H} 
    !(XB \otimes {!}\A(A,B))
    \xrightarrow{!m}
    {!}XA\rlap{ .}
  \end{equation}
  The remaining data $(\varepsilon, \delta, e, \Delta, m_I, m_\otimes)$
  are all given componentwise by the corresponding data for ${!}$ on
  $\V$.
\end{Rk}

  \begin{Prop}
  \label{prop:16}
  If $\V$ is a $k$-linear symmetric monoidal category, and
  ${!}$ is a monoidal differential modality on $\V$, then the pointwise
  monoidal differential modality $!^{\mathrm{ob}\,\A}$ on
  $\V^{\mathrm{ob}\,\A}$ lifts to a monoidal differential modality on
  $\cat{Psh}_\ell(\A)$.
\end{Prop}
\begin{proof}
  It suffices to show that, for each presheaf $X$ on $\A$, the family
  of maps $\mathsf{d}_{XA} \colon {!}XA \otimes XA \rightarrow {!}XA$
  are $\V^{!}$-linear in $A$, so comprising the components of a
  presheaf map $\mathsf{d}_X \colon {!}X \otimes X \rightarrow {!}X$.
  Once we have this, the fact that $\mathsf{d}$ is a natural
  transformation, and the axioms for a deriving transformation, follow
  componentwise from the corresponding facts in $\V$. The desired
  $\V^{!}$-linearity amounts to showing that, for all $A,B \in \A$,
  the following square commutes:
  \begin{equation*}
    \cd[@C+1em]{
      {{!}XB \otimes XB \otimes \A(A,B)} \ar[r]^-{m_{{!}X \otimes X}}
      \ar[d]_{\mathsf{d} \otimes 1} &
      {{!}XA \otimes XA} \ar[d]^{\mathsf{d}} \\
      {{!}XB \otimes \A(A,B)} \ar[r]^-{m_{{!}X}} &
      {{!}XA}
    }
  \end{equation*}
  where the top and bottom edges are the action maps of the presheaves
  ${!}X \otimes X$ and ${!}X$ respectively. Expanding these out, this
  is equally to show that the outside of the following diagram
  commutes, where to avoid an unwieldy presentation, we are writing
  $\mathrm{A}$, $\mathrm{B}$ and $\mathrm{H}$ as ciphers for $XA$,
  $XB$ and $\A(A,B)$:
\begin{equation*}
  \cd[@!C@C-7em]{
      & {!}\XB \ots \XB \ots {!}\AB \ar[rrr]^-{11\Delta}
      \ar[dl]_-{\mathsf{d}1} \ar[dr]^-{11 \delta} & & &
      {!}\XB \ots \XB \ots {!}\AB \ots {!}\AB \ar[rrr]^-{1 \sigma 1}
      \ar[dr]^-{11\delta\delta}&&&
      {!}\XB \ots {!}\AB \ots \XB \ots {!}\AB \ar[rrr]^-{1\delta11} \ar[dr]^-{1\delta1\delta}&&&
      {!}\XB \ots {!}{!}\AB \ots \XB \ots {!}\AB \ar[dr]^-{m_\otimes
        11} \\
      {!}\XB \ots {!}\AB \ar[dr]^-{1\delta} & &
      \sh{l}{0.6em}{!}\XB \ots \XB \ots {!}{!}\AB \ar[rrr]^-{11 \Delta}
      \ar[dl]^-{\mathsf{d}1}&&&
      \sh{l}{0.9em}{!}\XB \ots \XB \ots {!}{!}\AB \ots {!}{!}\AB
      \ar[rrr]^-{1\sigma1} &&&
      \sh{l}{0.8em}{!\XB} \ots {!}{!}\AB \ots \XB \ots {!}{!}\AB
      \ar[rrr]^-{m_{\otimes} 1 \varepsilon} &&&
      \sh{l}{0.1em}{!}(\XB \ots {!}\AB) \ots \XB \ots {!}\AB
      \ar[dl]^-{\mathsf{d}}
      \ar[dr]^-{{!}m \otimes m}\\
      & {!}\XB \ots {!}{!}\AB \ar[rrrrrrrrr]^-{m_\otimes} & & & &&& &&&
      {!}(\XB \ots {!}\AB) \ar[dr]^-{{!}m} &&
      {!}\XA \ots \XA\text{ .} \ar[dl]_-{\mathsf{d}} \\
      & & & & & & & & & & & {!}\XA
    }
  \end{equation*}
  Each of the small regions commutes easily using the axioms for a
  monoidal comonad plus naturality of $\mathsf{d}$. The large region
  is the so-called \emph{monoidal axiom}, which is shown
  in~\cite[Theorem~6.12]{Blute2019Differential} to hold for any
  monoidal differential modality.
\end{proof}

\subsection{The embedding theorem}
\label{sec:embedding-theorem}

We are finally ready for our second main result.

\begin{Thm}
  \label{thm:4}
  Every small cartesian differential category $\A$ admits a full
  structure-preserving embedding into the cartesian closed
  differential category induced by a monoidal differential modality on
  a symmetric monoidal closed $k$-linear category.
\end{Thm}
\begin{proof}
  Viewing $\A$ as a cartesian $\kmod^Q$-category, we can form the
  category of presheaves $\cat{Psh}(\A)$ and the  Yoneda
  embedding $y \colon \A \rightarrow_t \cat{Psh}(\A)$. By
  Lemma~\ref{lem:14} and Example~\ref{ex:23}, this corresponds to a
  full structure-preserving embedding of cartesian differential
  categories; so it suffices to show that the cartesian differential
  structure of $\cat{Psh}(\A)$ arises in the desired manner.

  The category on which this cartesian differential structure resides
  is the underlying ordinary category $\cat{Psh}_0(\A)$, which by
  Proposition~\ref{prop:11} is the co-Kleisli category of the comonad
  $I \ast (\thg)$ on $\cat{Psh}_\ell(\A)$. This comonad acts on a
  presheaf $X$ by sending it to the presheaf with components
  $(I \ast X)(A) = I \otimes^Q XA = I \otimes QXA$ and action
  \begin{equation*}
    I \otimes QXB \otimes Q\A(A,B) \xrightarrow{1 \otimes H} I \otimes Q(XB
    \otimes Q\A(A,B)) \xrightarrow{1 \otimes m} I \otimes QXA\rlap{ ;}
  \end{equation*}
  comparing with~\eqref{eq:38}, we see that upon transporting along
  the (invertible) left unit constraints of $\V$ we obtain precisely
  the lifted comonad on $\cat{Psh}_\ell(\A)$ of
  Proposition~\ref{prop:15}. Now by Propositions~\ref{prop:14},
  \ref{prop:15} and \ref{prop:16}, this lifted comonad is a monoidal
  differential modality on a symmetric monoidal closed $k$-linear
  category, and so its co-Kleisli category---which is
  $\cat{Psh}_0(\A)$---bears cartesian closed differential structure.

  All that remains is to show that the cartesian differential structure
  on $\cat{Psh}_0(\A)$ coming from this monoidal differential modality
  coincides with the cartesian differential structure coming from the
  $\kmod^Q$-enrichment. Clearly, it suffices to check that the
  differential operators coincide.

  Now, a map of $\cat{Psh}_0(\A)$ from $X$ to $Y$ is a $\V$-linear
  presheaf map $f \colon QX \rightarrow Y$ with components
  $f_A \colon QXA \rightarrow YA$ in $\kmod$. By 
  Proposition~\ref{prop:4}, each such component corresponds to a map
  $f_A^{(\bullet)} \colon XA \rightsquigarrow YA$ in $\faa(\kmodw)$
  upon defining
  \begin{equation*}
    f_A^{(n)}(x_0, \dots, x_n) \defeq f_A(\spn{x_0, \dots, x_n}) \rlap{ .}
  \end{equation*}
  %
  %

  Because the monoidal differential modality on $\cat{Psh}_\ell(\A)$
  is given pointwise by the initial differential modality $Q$ on
  $\kmod$, Proposition~\ref{prop:4} also tells us that
  that the differential
  $\mathrm{D}f \colon X \times X \rightarrow Y$ in $\cat{Psh}_0(\A)$
  associated to this monoidal differential modality
  sends an element $\spn{(x_0, y_0), \dots, (x_n, y_n)}$ of $Q(XA \times
  XA)$ to the element
  \begin{equation*}
    f^{(n+1)}(x_0, x_1, \dots, x_n, y_0) + \textstyle \sum_{i =1}^n f^{(n)}(x_0, \dots, x_{i-1}, y_i, x_{i+1}, \dots, x_n) \in YA\rlap{ .}
  \end{equation*}


  On the other hand, to compute $\mathrm{D}f$ via the
  $\kmod^Q$-enrichment of $\cat{Psh}(\A)$, we view $f$ as an element
  of the presheaf hom $\phom{X,Y} = \cat{Psh}_\ell(\A)(QX,Y)$, and
  apply the composition map
  $\phom{X,Y} \otimes Q\phom{X \times X, X} \rightarrow \phom{X \times
    X, Y}$ to this $f$ and to
  $\spn{p_0, p_1} \in Q\phom{X \times X, X}$. Here
  $p_0, p_1 \in \phom{X \times X, X} = \cat{Psh}_\ell(Q(X \times X), X)$ are the projection maps of the
  product in $\cat{Psh}(\A)$, with  components
\begin{equation}
     (p_i)_A \defeq Q(XA \times XA) \xrightarrow{\varepsilon} XA\times XA
     \xrightarrow{\pi_i} XA\rlap{ .}\label{eq:61}
   \end{equation}

   By definition of composition in $\cat{Psh}(\A)$, to compute
   $\mathrm{D}f$ in this way is equally well to partially evaluate the following map at $f$ and
   $\spn{p_0, p_1}$ in its first two arguments:
  \begin{equation*}
    \phom{X,Y} \otimes Q\phom{X^2,X} \otimes Q(XA)^2 \xrightarrow{1 \otimes H} 
    \phom{X,Y} \otimes Q(\phom{X^2,X} \otimes Q(XA)^2) 
    \xrightarrow{\mathsf{ev}
      (1 \otimes Q\mathsf{ev})} YA\text{.}
  \end{equation*}
  %
  %
  %
  We thus find the action of $\mathrm{D}f$ on
  an element $\spn{w_0, \dots, w_n} = \spn{(x_0, y_0), \dots, (x_n,
    y_n)}$ of $Q(XA
  \times XA)$ to be given as follows:
  \begin{align*}
    & \ \ \ \ \ \ \ f \otimes \spn{p_0, p_1} \otimes
    \spn{w_0, \dots, w_n} 
    \ \ \ \ \ \mapsto \sum_{\substack{[n] = A_1 \mid \dots \mid A_k \\ \theta \colon [1] \simeq [k]}} f \otimes \spn{
      p_{\theta_{(1)}} \otimes
      \spn{w_{A_{\theta_{(2)}}}}} \\
    & \mapsto\  \sum_{\substack{[n] = A_1 \mid \dots \mid A_k \\ \theta \colon [1] \simeq [k]}} f\Bigl(\spn{\,
      p_{\theta_{(1)}}(
      \spn{w_{A_{\theta_{(2)}}}})\,}\Bigr) 
    \ = \sum_{\theta \colon [1] \simeq [n]} f\Bigl(\spn{\,
      \pi_{\theta_{(1)}}(w_{\theta_{(2)}})\,}\Bigr) \\
    & = \ f^{(n+1)}(x_0, x_1, \dots, x_n, y_0) + \textstyle \sum_{1 \leqslant i \leqslant
      n} f^{(n)}(x_0, \dots, x_{i-1}, y_i, x_{i+1}, \dots, x_n)\rlap{ .}
  \end{align*}
  Here, we use Lemma~\ref{lem:1} at the
  first step, and for the first equality, use that $p_i =
  \pi_i \varepsilon$ and the description of $\varepsilon$ to see that
  each $A_i$ must be a singleton. This proves that the two
  differential operators on $\cat{Psh}_0(\A)$ coincide, as desired.
\end{proof}

\subsection{An explicit description of the embedding}
\label{sec:an-expl-descr}

We now give as explicit a description as possible of the structures
involved in the embedding theorem for cartesian differential
categories. We begin by describing the presheaves themselves.

\begin{Defn}
  \label{def:43}
  Let $\A$ be a cartesian differential category. A \emph{differential
    presheaf} on $\A$ is a (not necessarily additive) functor
  $X \colon \A^\mathrm{op} \rightarrow \kmod$ equipped with operators
  $\mathrm{D} \colon XA \rightarrow X(A \times A)$ such that:
  \begin{enumerate}[(i)]
  \item Each $\mathrm{D}$ is $k$-linear;
  \item Each $\mathrm{D}\xi \in X(A \times A)$ is
    $k$-linear in its second argument (as in Definition~\ref{def:24});
  \item $\mathrm{D}(\xi \cdot f) = \mathrm{D}(\xi) \cdot
    (f\pi_0,\mathrm{D}f) \in X(A \times A)$ for
    all $f \colon A \rightarrow B$ and
    $\xi \in XB$.
  \item $\mathrm{D}(\mathrm{D}\xi) \cdot (x,r,0,v) = \mathrm{D}(\xi) \cdot
    (x,v)$ for all
    $x,r,v \colon Z \rightarrow A$, $\xi \in XA$;
  \item
    $\mathrm{D}(\mathrm{D}\xi) \cdot (x,r,s,0) =
    \mathrm{D}(\mathrm{D}\xi) \cdot (x,s,r,0)$ for all
    $x,r,s \colon Z \rightarrow A$, $\xi \in XA$.
  \end{enumerate}
\end{Defn}
Here, as previously, we write $\xi \cdot f$ for $Xf(\xi)$.
\begin{Prop}
  \label{prop:22}
  Let $\A$ be a cartesian differential category. The
  $\kmod^Q$-enriched presheaves on $\A$ are exactly the differential
  presheaves.
\end{Prop}
\begin{proof}
  To give a $\kmod^Q$-enriched presheaf $X$ on $\A$ is to give
  $k$-modules $XA$ for each $A \in \A$, together with action maps
  $ XB \otimes Q\A(A,B) \rightarrow XA$ for each $A,B \in \A$ obeying
  unit and associativity axioms. If we notate the action maps as
  $\xi \otimes \spn{f_0, \dots, f_n} \mapsto \xi^{(n)}(f_0, \dots,
  f_n)$ then by transcribing the proof of Proposition~\ref{prop:8}, we
  see that to give these is equally to give functions
\begin{equation}
  \label{eq:60}
    \begin{aligned}
      XB \times \A(A,B) \times \A(A,B)^n & \rightarrow XA \\
      (\xi, f_0, \dots, f_n) & \mapsto \xi^{(n)}(f_0, \dots, f_n)
    \end{aligned}
  \end{equation}
  for each $A,B \in \A$ and $n \geqslant 0$, satisfying the evident
  analogues of axioms (i)--(iv) of Proposition~\ref{prop:8}. Now by
  arguing as in the proof of Theorem~\ref{thm:3}, this is in turn
  equivalent to giving a functor $X \colon \A^\mathrm{op} \rightarrow
  \kmod$ together with higher-order derivative operators
  \begin{equation*}
    (\thg)^{(n)} \colon XA \rightarrow X(A \times A^n)
  \end{equation*}
  satisfying the analogues of axioms (i)--(ii) and (iv)--(vii) of
  Corollary~\ref{cor:2}. Finally, by transcribing the argument of
  Corollary~\ref{cor:2} itself, we see that giving these higher-order
  derivative operators is equivalent to giving the first-order
  differential operators $\mathrm{D} \colon XA \rightarrow X(A \times
  A)$ satisfying the axioms (i)--(v) above.
\end{proof}


Just as before, we can be quite concrete about the correspondence
between differential presheaves and $\kmod^Q$-presheaves on the
cartesian differential category $\A$. On the one hand, given a
differential presheaf $X$ on $\A$, the corresponding $\kmod^Q$-presheaf
has the same components, and action maps
\begin{align*}
  XB \otimes Q\A(A,B) & \rightarrow XA \\
  \xi \otimes \spn{f_0, \dots, f_n} & \mapsto \xi^{(n)} \cdot (f_0, \dots, f_n)\rlap{ ,}
\end{align*}
where $\xi^{(n)}$ denotes the $n$th derivative of $\xi$ defined in the
same manner as in 
Definition~\ref{def:12}. On the other hand, for a 
$\kmod^Q$-presheaf $X$, the corresponding  differential presheaf
has the same components, and action maps and differential
defined from the $\kmod^Q$-action maps $m \colon XB \otimes
Q\A(A,B) \rightarrow XA$ via
\begin{equation*}
  \xi \cdot f = m(\xi \otimes \spn{f}) \qquad \text{and} \qquad
  \mathrm{D}\xi = m(\xi \otimes \spn{\pi_0, \pi_1})\rlap{ .}
\end{equation*}

We now describe the category of differential presheaves and linear
maps on a cartesian differential category $\A$, along with its
symmetric monoidal closed $k$-linear structure and its differential modality.

\begin{Defn}
  \label{def:45}
  Let $\A$ be a cartesian differential category.
  \begin{itemize}[itemsep=0.25\baselineskip]
  \item   A \emph{linear map} $\alpha \colon X \rightarrow_\ell Y$ of
  differential presheaves on $\A$ is a
    natural transformation ${\alpha \colon X \Rightarrow Y \colon
    \A^\mathrm{op} \rightarrow \kmod}$ which preserves the
  differential, i.e.,
  \begin{equation*}
    \alpha_{A \times A}(\mathrm{D} \xi)
    = \mathrm{D}(\alpha_A(\xi)) \qquad \text{for all $A \in \A$ and $\xi \in XA$.}
  \end{equation*}
  We write $\cat{D\P sh}_\ell(\A)$ for the $k$-linear category of
  differential presheaves on $\A$ and linear maps, where the
  $k$-module structure on the homs is given componentwise.
\item   The \emph{pointwise tensor product} $X \otimes Y$ of differential presheaves $X$
  and $Y$ has underlying functor $X \otimes Y \colon \A^\mathrm{op}
  \rightarrow \kmod$ with values $(X \otimes Y)A = XA \otimes XY$ and
  $(X \otimes Y)f = Xf \otimes Yf$, and differential operator given by
  \begin{equation}
    \label{eq:9}
    \begin{aligned}
      \mathrm{D} \colon XA \otimes YA & \rightarrow X(A \times A)
      \otimes Y(A \times A) \\
      \xi \otimes \upsilon & \mapsto \mathrm{D}\xi \otimes (\upsilon
      \cdot \pi_0) + (\xi \cdot \pi_0) \otimes \mathrm{D}\upsilon\rlap{
        .}
    \end{aligned}
  \end{equation}
  \item The \emph{pointwise unit} is the differential presheaf whose underlying functor
  $I \colon \A^\mathrm{op} \rightarrow \kmod$ is constant at $k \in
  \kmod$, and whose differential operator $\mathrm{D}  \colon I(A) \rightarrow
  I(A\times A)$ is everywhere zero.
\item For a differential presheaf $X$, the differential presheaf
    $QX$ has underlying functor $Q \circ X \colon \A^\mathrm{op}
    \rightarrow \kmod$, where $Q$ is the initial differential modality
    on $\kmod$, and differential operator $\mathrm{D} \colon QX(A) \rightarrow QX(A \times A)$ given by
    \begin{equation*}
      \spn{\xi_0, \dots, \xi_n}  \mapsto \spn{\xi_0 \pi_0, \dots,
        \xi_n \pi_0, \mathrm{D}\xi_0} + \sum_{1 \leqslant i
        \leqslant n} \spn{\xi_0 \pi_0, \dots, \xi_{i-1} \pi_0,
        \mathrm{D}\xi_{i}, \xi_{i+1}\pi_0, \dots, \xi_n \pi_0}\rlap{ .}
    \end{equation*}
  \item The \emph{pointwise differential
    modality} $Q$ on $\cat{D\P sh}_\ell(\A)$ has the action on objects
  just described, and all its remaining data given pointwise by the
  corresponding data for the initial
    differential modality on $\kmod$.

  \item The \emph{pointwise internal hom} $[Y,Z]$ of differential
    presheaves $Y,Z$ is the functor
    $[Y,Z] \colon \A^\mathrm{op} \rightarrow \kmod$ with values 
    $[Y,Z]A = \cat{D\P sh}_\ell(Q\A(\thg, A) \otimes Y, Z)$ and
    $[Y,Z]f = (\thg) \circ (\tilde f \otimes 1)$; here, if $f \colon A \rightarrow
    B$ in $\A$ then $\tilde f \colon Q\A(\thg, A) \rightarrow_\ell
    Q\A(\thg, B)$ is the linear map with components
    \begin{align*}
      \tilde f \colon Q\A(C, A) & \rightarrow Q\A(C,B) \\
      \spn{g_0, \dots, g_n} & \mapsto \sum_{[n] = A_1 \mid \cdots \mid
      A_k} \spn{f^{(\emptyset)}(\vec g), f^{(A_1)}(\vec g), \dots,
      f^{(A_k)}(\vec g)}\rlap{ .}
  \end{align*}
  Its differential operator is given by $(\thg) \circ
  (\chi_A \otimes 1) \colon [Y,Z](A) \rightarrow [Y,Z](A \times A)$, where $\chi_A \colon Q\A(\thg, A \times
  A)
  \rightarrow_\ell Q\A(\thg, A)$ is the linear map with components
  \begin{align*}
      \chi_A \colon Q\A(C, A \times A) & \rightarrow Q\A(C,A) \\
      \spn{(f_0,g_0), \cpdots, (f_n,g_n)} & \mapsto \spn{f_0, \cpdots,
        f_n, g_0} + \textstyle\sum_{i=1}^n \spn{f_0, \cpdots, f_{i-1},g_i,f_{i+1},
      \cpdots, f_n}\rlap{ .}
  \end{align*}

\end{itemize}
\end{Defn}
\begin{Prop}
  \label{prop:23}
  Let $\A$ be a cartesian differential category. The category
  $\cat{Psh}_\ell(\A)$ of presheaves on $\A$ \emph{qua}
  $\kmod^Q$-category is isomorphic to $\cat{D\P sh}_\ell(\A)$. Under
  this isomorphism, the symmetric monoidal closed $k$-linear structure and
  differential modality on $\cat{Psh}_\ell(\A)$ of
  Propositions~\ref{prop:14} and Proposition~\ref{prop:16} correspond
  to the pointwise symmetric monoidal closed structure and differential
  modality on $\cat{D\P sh}_\ell(\A)$.
\end{Prop}
\begin{proof}
  The non-trivial points are verifying the descriptions of the tensor
  product, tensor unit, action of $Q$, and internal hom for
  differential presheaves. For the first of these, the tensor product
  of $X,Y \in \cat{Psh}_\ell(\A)$ has $(X \otimes Y)A = XA \otimes YA$
  with action maps
  $m \colon XB \otimes YB \otimes Q\A(A,B) \rightarrow XA \otimes YA$
  given by the composite in~\eqref{eq:22}. Tracing an element of the
  form $\xi \otimes \nu \otimes \spn{f}$ through this composite we get
  \begin{equation*}
    \xi \otimes \nu \otimes \spn{f} \ \mapsto\  \xi \otimes \nu \otimes
    \spn{f} \otimes \spn{f} \ \mapsto\  \xi \otimes \nu \otimes
    \spn{f} \otimes \spn{f} \ \mapsto\  (\xi \cdot f) \otimes (\nu \cdot f)
  \end{equation*}
  so that the corresponding differential presheaf satisfies $(X
  \otimes Y)f = Xf \otimes Yf$; while tracing through $\xi \otimes \nu \otimes
  \spn{\pi_0, \pi_1}$, we get
  \begin{multline*}
    \xi \otimes \nu \otimes \spn{\pi_0, \pi_1} \ \mapsto \
    \xi \otimes \nu \otimes (\spn{\pi_0, \pi_1} \otimes \spn{\pi_0} +
    \spn{\pi_0} \otimes \spn{\pi_0, \pi_1}) \\ \ \mapsto \
    \xi \otimes \spn{\pi_0, \pi_1} \otimes \nu \otimes \spn{\pi_0} +
    \xi \otimes \spn{\pi_0} \otimes \nu \otimes \spn{\pi_0, \pi_1} \
    \mapsto \
    \mathrm{D}\xi \otimes \nu \pi_0 + \xi \pi_0 \otimes \mathrm{D}\nu 
  \end{multline*}
  so that the corresponding differential presheaf has differential
  operator~\eqref{eq:9}.

  Similarly, the unit presheaf in $\cat{Psh}_\ell(\A)$ is constant at
  $k$, and its action maps~\eqref{eq:24} act on elements
  $1 \otimes \spn{f}$ and $1 \otimes \spn{\pi_0, \pi_1}$ via
  $1 \otimes \spn{f} \mapsto 1$ and
  $1 \otimes \spn{\pi_0, \pi_1} \mapsto 0$, so that the corresponding
  differential presheaf is constant at $I$, with zero
  differential.

  Next, for any $X \in \cat{Psh}_\ell(\A)$, the corresponding $QX \in
  \cat{Psh}_\ell(\A)$ has components $QXA$, and action maps $QXB
  \otimes Q\A(A,B) \rightarrow QXA$ given as
  in~\eqref{eq:38}. Tracing an element of the form $\spn{\xi_0, \dots,
    \xi_n} \otimes \spn{f}$ through this composite, we get
  \begin{align*}
    \spn{\xi_0, \dots,
      \xi_n} \otimes \spn{f} \ \mapsto \!\!\!\smash{\sum_{\substack{[0] = A_1 \mid \dots \mid A_k \\ \theta \colon [n] \simeq [k]}}} \!\!\!\spn{
      \xi_{\theta_{(1)}} \otimes
      \spn{f_{A_{\theta_{(2)}}}}} &= \spn{\xi_0 \otimes \spn{f},
      \dots, \xi_n \otimes \spn{f}} \\ &\mapsto \spn{\xi_0 \cdot
      f, \dots, \xi_0 \cdot f}\rlap{ ,}
  \end{align*}
  so that the corresponding differential presheaf satisfies $(QX)f =
  Q(Xf)$. On the other hand, tracing through $\spn{\xi_0, \dots,
    \xi_n} \otimes \spn{\pi_0, \pi_1}$ yields
\begin{align*}
    \ \ & \spn{\xi_0, \cpdots,
      \xi_n} \otimes \spn{\pi_0, \pi_1} \ \mapsto \!\!\!\sum_{\substack{[1] = A_1 \mid \dots \mid A_k \\ \theta \colon [n] \simeq [k]}} \!\!\!\spn{
      \xi_{\theta_{(1)}} \otimes
      \spn{\pi_{A_{\theta_{(2)}}}}} = \!\!\!\sum_{\theta \colon [n] \simeq [1]} \!\!\!\spn{
      \xi_{\theta_{(1)}} \otimes
      \spn{\pi_{\theta_{(2)}}}}\\
    &= \spn{\xi_0 \otimes \spn{\pi_0}, \dots, \xi_n \otimes
      \spn{\pi_0}, \xi_0 \otimes \spn{\pi_0, \pi_1}} \\ & \qquad \ 
    + \sum_{1
      \leqslant i \leqslant n} \spn{\xi_0 \otimes \spn{\pi_0}, \dots,
      \xi_{i-1} \otimes \spn{\pi_0}, \xi_i \otimes \spn{\pi_0, \pi_1},
    \xi_{i+1} \otimes \spn{\pi_0}, \dots, \xi_n \otimes \spn{\pi_0}}
  \\ &
  \mapsto \spn{\xi_0 \cdot \pi_0, \dots,
        \xi_n \cdot \pi_0, \mathrm{D}\xi_0} + \sum_{1 \leqslant i
        \leqslant n}\! \spn{\xi_0 \cdot \pi_0, \dots, \xi_{i-1} \cdot \pi_0,
        \mathrm{D}\xi_{i}, \xi_{i+1} \cdot \pi_0, \dots, \xi_n \cdot \pi_0}
    \end{align*}
    so that the corresponding differential presheaf $QX$ has the 
    differential operator of Definition~\ref{def:45}, as desired.

    Finally, we consider the internal hom $[Y,Z]$ in
    $\cat{Psh}_\ell(\A)$. By Remark~\ref{rk:3}, this has components
    $[Y,Z]A = \phom{Q\A(\thg, A) \otimes Y Z} = \cat{D\P
      sh}_\ell(Q\A(\thg, A) \otimes Y, Z)$ with the action maps
    $[Y,Z]B \otimes Q\A(A,B) \rightarrow [Y,Z]A$ obtained by
    transposing the composites in~\eqref{eq:44}. In particular, this
    means that for any $f \colon A \rightarrow B$ in $\A$, the
    reindexing map $[Y,Z]f$ of the corresponding differential presheaf
    sends $\alpha \in [Y,Z]B$ to the element of $[Y,Z]A$ whose
    components are obtained by partially evaluating~\eqref{eq:44} at
    $\alpha$ and $\spn{f}$ in its first two parameters. We thus find the value of
    the linear map $Q\A(\thg, A) \otimes Y \rightarrow Z$ so induced
    at
    $\spn{g_0, \dots, g_n} \otimes \gamma \in {!}\A(C,A) \otimes YC$
    to be given by
    \begin{align*}
      \alpha \otimes \spn{f} \otimes \spn{g_0, \dots, g_n} \otimes
      \gamma &\mapsto \!\!\!\sum_{[n] = A_1 \mid \cdots \mid
      A_k} \alpha \otimes \spn{f \otimes \spn{g_0}, f \otimes
      \spn{g_{A_1}}, \dots, f \otimes \spn{g_{A_k}}} \otimes \gamma \\
      & \mapsto \!\!\!\sum_{[n] = A_1 \mid \cdots \mid
      A_k} \alpha \otimes \spn{f^{(\emptyset)}(\vec g), f^{(A_1)}(\vec
      g), \dots, f^{(A_k)(\vec g)}} \otimes \gamma \\
      & \mapsto \!\!\!\sum_{[n] = A_1 \mid \cdots \mid
      A_k} \alpha (\spn{f^{(\emptyset)}(\vec g), f^{(A_1)}(\vec
      g), \dots, f^{(A_k)}(\vec g)} \otimes \gamma) 
  \end{align*}
  so that $([Y,Z]f)(\alpha)$ is precisely the composite
  \begin{equation*}
    Q\A(\thg, A) \otimes Y \xrightarrow{\tilde f \otimes 1} Q\A(\thg,
    B) \otimes Y \xrightarrow{\alpha} Z
  \end{equation*}
  as desired. Similarly, for any $\alpha \in [Y,Z]A$, its
  differential $\mathrm{D}\alpha \in [Y,Z](A \times A)$ is obtained by partially evaluating~\eqref{eq:44} at
    $\alpha$ and $\spn{\pi_0, \pi_1}$ in its first two parameters. So the value of
    $\mathrm{D}\alpha \colon Q\A(\thg, A \times A) \otimes Y
    \rightarrow Z$ at an element
    $\spn{(f_0, g_0), \dots, (f_n, g_n)} \otimes \gamma = \spn{h_0,
      \dots, h_n} \otimes \gamma $ of $Q\A(C,A \times A) \otimes YC$
    is given by
\begin{align*}
    & \ \ \ \ \ \ \ \alpha \otimes \spn{\pi_0, \pi_1} \otimes
    \spn{h_0, \dots, h_n} 
    \ \ \ \ \ \mapsto \sum_{\substack{[n] = A_1 \mid \dots \mid A_k \\ \theta \colon [1] \simeq [k]}} \alpha \otimes \spn{
      \pi_{\theta_{(1)}} \otimes
      \spn{h_{A_{\theta_{(2)}}}}} \otimes \gamma \\
    & \mapsto\  \sum_{\substack{[n] = A_1 \mid \dots \mid A_k \\ \theta \colon [1] \simeq [k]}} \alpha\Bigl(\spn{\,
      \pi_{\theta_{(1)}}^{(A_{\theta_{(2)}})}(
      \vec h)\,}\Bigr) 
    \ = \sum_{\theta \colon [1] \simeq [n]} \alpha\Bigl(\spn{\,
      \pi_{\theta_{(1)}}(g_{\theta_{(2)}})\,} \otimes \gamma\Bigr) \\
    & = \ \alpha(\spn{f_0, _1, \dots, f_n, g_0} \otimes \gamma) + \textstyle \sum_{1 \leqslant i \leqslant
      n} \alpha^{(n)}(\spn{f_0, \dots, f_{i-1}, g_i, f_{i+1}, \dots,
      f_n} \otimes \gamma)\rlap{ .}
  \end{align*}
  as desired.
\end{proof}

We can now read off from the above a description of the cartesian closed
differential category associated to the pointwise differential
modality on $\cat{D\P sh}_\ell(\A)$---which, in light of the preceding
proposition, is equally well the cartesian closed differential category
associated to the $\kmod^Q$-category $\cat{Psh}(\A)$---together with
its embedding of $\A$.

\begin{Defn}
  \label{def:47}
  Let $\A$ be a cartesian differential category. 
  \begin{itemize}[itemsep=0.25\baselineskip]
  \item A \emph{Fa\`a di Bruno map}
    $\alpha^{(\bullet)} \colon X \rightsquigarrow Y$ of differential
    presheaves on $\A$ comprises a family of
    maps $\alpha_A^{(\bullet)} \colon XA \rightsquigarrow YA$ in
    $\faa(\kmodw)$ which are natural in $A$, i.e., each square of the following
    form commutes:
    \begin{equation*}
      \cd{
        {XA \times (XA)^n} \ar[r]^-{\alpha^{(n)}_A} \ar[d]_{Xf \times
        (Xf)^n} &
        {YA} \ar[d]^{Yf} \\
        {XB \times (XB)^n} \ar[r]^-{\alpha^{(n)}_B} &
        {YB}
      }
    \end{equation*}
    and which respect the differential, in the sense that
    \begin{multline*}
      \mathrm{D}\bigl(\alpha_A^{(n)}(\xi_0, \dots, \xi_n)\bigr) =
      \alpha_{A \times A}^{(n+1)}(\xi_0 \pi_0, \dots,
        \xi_n \pi_0, \mathrm{D}\xi_0) \\ {} + \textstyle\sum_{1 \leqslant i
        \leqslant n} \alpha_{A \times A}^{(n)}(\xi_0 \pi_0, \dots, \xi_{i-1} \pi_0,
        \mathrm{D}\xi_{i}, \xi_{i+1}\pi_0, \dots, \xi_n \pi_0)\rlap{ .}
      \end{multline*}
      We write $\cat{D\P sh}_{f}(\A)$ for the left-$k$-linear 
    category of differential presheaves and Fa\`a di Bruno maps, with
    $k$-module structure on the homs and composition given pointwise
    as in $\faa(\kmodw)$.
  \item The \emph{cartesian product} $X \times Y$ of differential
    presheaves $X,Y$ has values gives by $(X \times Y)A = XA \times YA$ and
    $(X \times Y)f = Xf \times Xf$, and componentwise differential;
    the projection maps
    $\pi_0 \colon X \leftarrow X \times Y \rightarrow Y \colon \pi_1$
    are given pointwise as in $\faa(\kmodw)$. The \emph{terminal}
    differential presheaf is constant at the terminal object $1 \in
    \kmodw$, with the only possible differential.
  \item For each $A \in \A$, the representable differential presheaf
    $yA$ has underlying functor
    $\A(\thg, A) \colon \A^\mathrm{op} \rightarrow \kmod$, and 
    differential operator inherited from $\A$;
  \item For each $f \colon A \rightarrow B$ in $\A$, the Fa\`a di
  Bruno map $yf \colon yA \rightsquigarrow yB$ has
  components given by the higher-order derivatives in $\A$:
  \begin{align*}
    yf_X^{(n)} \colon \A(X,A) \times \A(X,A)^n & \rightarrow \A(X,B)
    \\
    (g_0, \dots, g_n) & \mapsto f^{(n)}(g_0, \dots, g_n)\rlap{ ;}
  \end{align*}
  we write $y \colon \A \rightarrow \cat{DPsh}_f(\A)$ for the functor
  so induced.

\item The \emph{pointwise cartesian differential structure} on
  $\cat{D\P sh}_f(\A)$ has the differential
  $\mathrm{D}f \colon X \times X \rightsquigarrow Y$ of a Fa\`a di
  Bruno map $f \colon X \rightsquigarrow Y$ given pointwise as in
  $\faa(\kmodw)$.
  
  \item The \emph{exponential} $Z^Y$ of differential presheaves $Y,Z$
    has values $(Z^Y)A = \cat{D\P sh}_f(\A)(yA \times Y, Z)$ and
    $(Z^Y)f = (\thg) \circ (yf \times 1)$. Its differential operator
    sends $\alpha \colon yA \times Y \rightsquigarrow Z$ to the
    composite
    \begin{equation*}
      y(A \times A) \times Y \xrightarrow{\cong} yA \times yA \times Y
      \xrightarrow{\mathrm{D}_1 f} Z
    \end{equation*}
    whose second component is the partial derivative in the pointwise
    cartesian differential structure.
  \end{itemize}
\end{Defn}
\begin{Prop}
  \label{prop:24}
  Let $\A$ be a cartesian differential category. The cartesian closed
  differential category $\cat{D\P sh}_f(\A)$ is induced by the
  pointwise differential modality on $\cat{D\P sh}_\ell(\A)$, and so
  isomorphic to the cartesian closed differential category of
  $\kmod^Q$-presheaves $\cat{Psh}(\A)$. Under this
  isomorphism, the Yoneda embedding
  $y \colon \A \rightarrow \cat{DPsh}_f(\A)$ corresponds to the
  enriched Yoneda embedding $\A \rightarrow \cat{Psh}(\A)$.
\end{Prop}
\begin{proof}
  A map $X \rightarrow Y$ in the co-Kleisli category of the pointwise
  differential modality is a linear map of differential presheaves
  $QX \rightarrow_\ell Y$, and we can read off from
  Definition~\ref{def:45} that these are precisely Fa\`a di Bruno maps
  from $X$ to $Y$. Aside from the exponentials, the identification of
  the remaining structure of the co-Kleisli category with that of
  $\cat{D\P sh}_f(\A)$ follows since the differential modality on
  $\cat{D\P sh}_\ell(\A)$ is induced pointwise from the initial
  differential modality $Q$ on $\kmod$, and since by
  Proposition~\ref{prop:4} we have $\cat{Kl}(Q) \cong \faa(\kmodw)$ as
  cartesian differential categories.

  As for the exponentials: recall that these are obtained from the
  internal homs in $\cat{D\P sh}_\ell$ via the formula
  $Z^Y \defeq [QY, Z]$. Expanding this definition out, we see that
  $(Z^Y)A = \cat{D\P sh}_\ell(QyA \otimes QY, Z) \cong \cat{D\P
    sh}_\ell(Q(yA \times Y), Z) \cong \cat{D\P sh}_f(yA \times Y, Z)$,
  using the storage isomorphisms~\eqref{eq:3} at the second step.
  Transporting the action on maps and differential operator on
  $[QY, Z]$ across these isomorphisms yields, by an easy argument, the
  formula indicated above.
\end{proof}
Putting all of the above together, we get the following concrete form
of the embedding theorem.

\begin{Thm}
  \label{thm:2}
  Let $\A$ be a small cartesian differential category. The Yoneda
  embedding $y \colon \A \rightarrow \cat{D\P sh}_f(\A)$ of
  Definition~\ref{def:47} is a full structure-preserving embedding of
  $\A$ into the cartesian closed differential category induced
  by the monoidal differential modality described in
  Definition~\ref{def:45}.
\end{Thm}

\begin{Rk}
  \label{rk:2}
  We know from the general theory that
  $y \colon \A \rightarrow \cat{D\P sh}_f(\A)$ is a fully faithful
  embedding of cartesian differential categories, but this may not be
  immediately apparent from the concrete description. As a sanity
  check, let us conclude by giving a direct argument for the full
  fidelity.

  A Fa\`a di Bruno map
  $\alpha^{(\bullet)} \colon yA \rightsquigarrow yB$ is a linear map
  $QyA \rightarrow_\ell yB$ of differential presheaves, i.e., a
  natural transformation
  $\alpha \colon Q\A(\thg, A) \Rightarrow \A(\thg, B) \colon
  \A_0^\mathrm{op} \rightarrow \kmod$ which commutes with the
  differentials. Forgetting about the differentials for the moment,
  just to give a natural transformation of this form is to give a map
  of $\cat{Kl}_\A(Q)$ as in Definition~\ref{def:23}, which is by
  Proposition~\ref{prop:5} the same as a map
  $f^{(\bullet)} \colon A \rightsquigarrow B$ in $\faa(\A)$;
  concretely, the correspondence is given by
  \begin{equation*}
    f^{(n)} = \alpha^{(n)}_{A \times A^n}(\pi_0, \dots, \pi_n) \in
    yB(A \times A^n) = \A(A \times A^n \rightarrow B)\rlap{ .}
  \end{equation*}
  Now
  adding back in the condition that $\alpha$ preserves the
  differential is, by a short calculation, the same as requiring that
  each $f^{(n)}$ is, in fact, the $n$th derivative of $f^{(0)}$, so
  that the Fa\`a di Bruno map $\alpha^{(\bullet)} \colon yA
  \rightsquigarrow yB$ is necessarily given by
  \begin{align*}
    \alpha_X^{(n)} \colon \A(X,A) \times \A(X,A)^n &\rightarrow \A(X,B)
    \\
    (g_0, \dots, g_n) & \mapsto f^{(n)}(g_0, \dots, g_n)
  \end{align*}
  for a unique map $f \colon A \rightarrow B$ in $\A$.
\end{Rk}

\paragraph{\textbf{Acknowledgements:}} The second author would like to thank Geoff Cruttwell, Rory Lucyshyn-Wright, Robin Cockett, Jonathan Gallagher, and Ben MacAdam for useful discussions regarding the embedding theorem.

\bibliographystyle{acm}
\bibliography{bibdata}

\end{document}